\documentclass[journal]{IEEEtran}

\usepackage{amsmath,amsfonts,amsthm,amssymb,mathrsfs,bbm}

\usepackage{algorithm}
\usepackage{algpseudocode}

\usepackage{array,booktabs,multirow}
\usepackage{enumitem}
\usepackage{graphicx}
\usepackage[caption=false,font=normalsize,labelfont=sf,textfont=sf]{subfig}

\usepackage{textcomp}
\usepackage{stfloats}
\usepackage{url}
\usepackage{verbatim}

\usepackage{cite}
\usepackage{cancel}

\usepackage{balance}
\usepackage{subfig}

\algrenewcommand\algorithmicrequire{\textbf{Input:}}
\algrenewcommand\algorithmicensure{\textbf{Output:}}

\def\N			{\mathbb N}
\def\Z			{\mathbb Z}
\def\R			{\mathbb R}
\def\C			{\mathbb C}

\def\Fourier	{\mathcal F}

\def\Mod		{\mathscr M}

\def\ind		{\mathbbm 1}

\def\Cont		{\mathscr C}
\def\Lebesgue	{\mathrm L}
\def\PW			{\mathsf{PW}}

\def\d			{\mathrm d}
\def\e			{\mathrm e}
\def\i			{\mathrm i}
\def\T			{\mathrm T}

\def\bfV		{\mathbf{V}}
\def\bfc		{\mathbf{c}}

\def\bfr		{\mathbf{r}}
\def\bfs		{\mathbf{s}}

\def\Ell		{\mathbb L}
\def\I			{\mathbb I}
\def\A			{\mathcal A}
\def\B			{\mathcal B}

\def\P			{\mathcal P}
\def\S			{\mathcal S}

\def\OF			{\mathsf{OF}}

\DeclareMathOperator{\sinc}{sinc}
\DeclareMathOperator{\supp}{supp}
\DeclareMathOperator{\sgn}{sgn}

\DeclareMathOperator*{\argmax}{arg\,max}
\DeclareMathOperator*{\argmin}{arg\,min}

\newtheorem{definition}{Definition}
\newtheorem{proposition}{Proposition}
\newtheorem{theorem}{Theorem}

\newtheorem{assumption}{Assumption}

\begin{document}
	
\title{Generalized Modulo Hysteresis Encoding}
	
\author{Matthias Beckmann and J{\"u}rgen Jeschke%
\thanks{This work was supported by Deutsche Forschungsgemeinschaft (DFG) - Project number 530863002.}%
\thanks{M.~Beckmann is with the Dept.\ of Mathematics, University of Hamburg, Germany~(Email: \texttt{research@mbeckmann.de}).}%
\thanks{J.~Jeschke is with the Center for Industrial Mathematics, University of Bremen, Germany (Email:\texttt{jjeschke@uni-bremen.de}).}%
\thanks{The authors are listed in alphabetical order.}}

\maketitle
	
\begin{abstract}
	Unlimited sensing and its extension via modulo hysteresis provide an efficient encoding scheme for high dynamic range signals by folding the signal's amplitude into the range of the analog-to-digital converter prior to sampling.
	Motivated by shortcomings of the modulo hysteresis model, in this work we introduce generalized modulo encoders that flexibly design the folding by incorporating a general transition function and allowing the folding times to depend on the nonideality of the folding transition.
	We rigorously study mathematical properties of the respective operators and derive conditions under which a bandlimited signal is uniquely determined by its generalized modulo samples.
	Finally, we propose a recovery algorithm backed by mathematical guarantees and illustrate our theoretical results through numerical experiments.
\end{abstract}
	

\section{Introduction}

By the classical Shannon sampling theorem~\cite{Shannon1948} any band\-limited signal $g$ with bandwidth $\Omega \, \tfrac{\mathrm{rad}}{\mathrm{s}}$ is uniquely determined by equispaced samples $\{g(n\T) \mid n \in \Z\}$ with sampling rate $\T > 0$ if sampled at or above Nyquist rate, i.e., if $\T \leq \tfrac{\pi}{\Omega}$.
Moreover, recovery is achieved by employing the well-known Whittaker-Shannon-Kotelnikov interpolation formula
\begin{equation*}
	g(t) = \sum\nolimits_{n \in \Z} g(n\T) \, \sinc\bigl(\tfrac{\pi}{\T}(t-n\T)\bigr).
\end{equation*}
In hardware, this fundamental principle is implemented by analog-to-digital converters (ADCs)~\cite{Walden1999}.
By design, conventional ADCs operate at a fixed dynamic range and can only capture a quantized signal depending on the available bit-budget resulting in inevitable quantization noise.
This imposes a fundamental limit on the digital resolution leading to an intrinsic trade-off between resolution and dynamic range.
In particular, if the input signal exceeds the capacity of the ADC, it leads to saturation or clipping and, hence, a permanent loss of information, which compromises the reconstruction quality.

To account for these limitations, the Unlimited Sampling Framework (USF) has been introduced in \cite{Bhandari2017, Bhandari2020, Bhandari2021}, a fundamentally different sampling paradigm based on the observation that any function can be decomposed into its \emph{integer part}, corresponding to the quantized signal, and its \emph{fractional part}, representing the quantization noise.
The key enabler is then the following mathematical insight: \emph{for smooth functions, their fractional part encodes the integer part}~\cite{Zhu2025a}.
Hence, USF suggests to capture the fractional part by folding the input signal prior to sampling into the dynamic range of the ADC by means of modulo arithmetic.
More precisely, the \emph{ideal modulo encoder} $\Mod_\lambda$ with threshold $\lambda > 0$ is for $g\colon \R \to \R$ defined via
\begin{equation*}
\Mod_\lambda g(t) = g(t) - 2\lambda \Bigl\lfloor\frac{g(t) + \lambda}{2\lambda}\Bigr\rfloor,
\end{equation*}
where $\lfloor\cdot\rfloor$ denotes the floor function and $\Mod_\lambda g(t) \in [-\lambda,\lambda]$.
The Unlimited Sampling Theorem~\cite[Theorem 1]{Bhandari2017} then states that $g \in \PW_\Omega$ can be recovered (up to an additive constant) from $\{\Mod_\lambda g(n\T) \mid n \in \Z\}$ if $\T < \frac{1}{\Omega\e}$.
The first modulo ADC prototype was introduced in~\cite{Bhandari2021} and since then several hardware implementations have been proposed, e.g.~\cite{Mulleti2023, Zhu2024, Zhu2025, Zhu2025a, Guo2025, Li2026}, along with various reconstruction approaches, see~\cite{Rudresh2018, Romanov2019, Weiss2022, Azar2022, Shah2023, Guo2023, Azar2025, Patricio2025, Kobayashi2026}.
Moreover, USF has been extended to diverse applications like, for instance, imaging~\cite{Bhandari2020a, Zhou2020, Monroy2026}, communication~\cite{Ordonez2021, Liu2023, Liu2025}, electro\-encephalography~\cite{Geng2023}, computerized tomography~\cite{Beckmann2022, Beckmann2024, Beckmann2025} and radar~\cite{Feuillen2023, Feuillen2025, Zhang2025}.

\begin{figure}
	\centering
	\includegraphics[width=\linewidth]{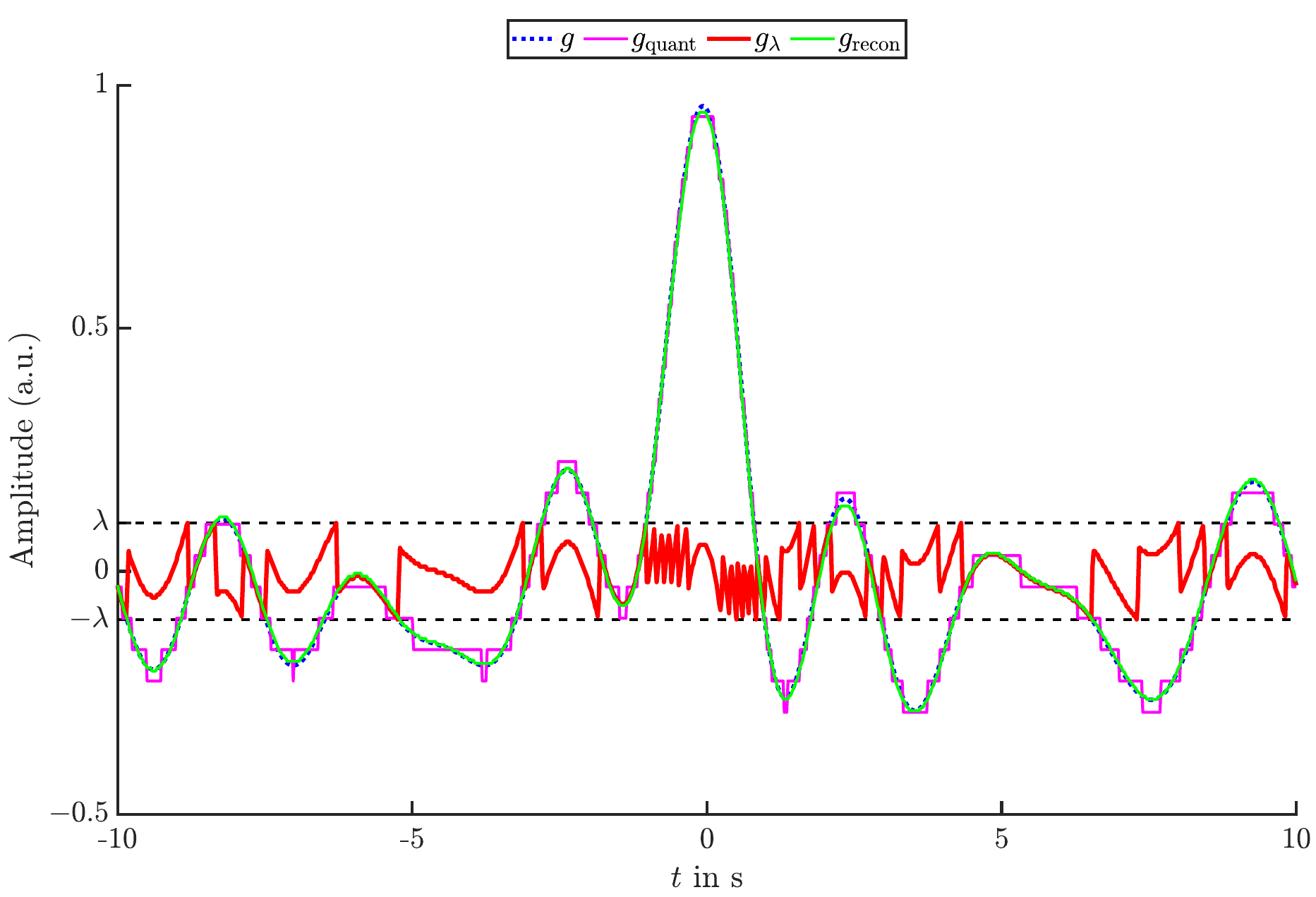}
	\caption{Illustration of generalized modulo encoding.
	Capturing the quantized signal $g_{\mathrm{quant}}$ with $5~\text{bits}$ yields a mean squared error of $3.6074 \cdot 10^{-4}$.
	As opposed to this, capturing the modulo signal $g_\lambda$ with the same bit-budget allows for the reconstruction $g_{\mathrm{recon}}$ with mean squared error of $3.2712 \cdot 10^{-5}$.}
	\label{fig:Mod_recon_quantized}
\end{figure}

The ideal modulo encoder $\Mod_\lambda$ assumes an instantaneous fold whenever the input signal reaches the dynamic threshold $\pm\lambda$.
However, as argued in~\cite{Florescu2022} such exact folding transitions cannot be realized due to hardware imperfections.
As opposed to this, in~\cite{Florescu2022, Florescu2022a, Florescu2025} the modulo hysteresis encoder $\Mod_H$ with $H = (\lambda,h,\alpha)$ is introduced, which accounts for non-instantaneous folds by incorporating a hysteresis parameter $h \geq 0$ that models an imperfect alignment of the reset threshold and the post-reset value, and a transient parameter $\alpha \geq 0$ accounting for an inexact folding transition.
However, $\Mod_H$ has the following limitations that motivate further refinement:
\begin{enumerate}[label = \Roman*)]
\item $\Mod_H g$ is not always guaranteed to stay in $[-\lambda,\lambda]$;\label{item:P1}
\item folding transition is modelled as a straight line;\label{item:P2}
\item folding times are independent of the transient nonideality.\label{item:P3}
\end{enumerate}

In our preliminary work~\cite{Beckmann2025a} we addressed \ref{item:P1} by proposing the modified modulo hysteresis operator $\Mod_\lambda^{h,\alpha}$, which makes the folding positions dependent on the folding transition, as also observed in modulo hardware in~\cite{Zhu2025}.
In its definition the folding transition is still modelled as a linear function.
As this is a rather restrictive assumption that appears unrealistic in practical implementations, in this paper we generalize the operators $\Mod_H$ and $\Mod_\lambda^{h,\alpha}$ by incorporating a general folding function $j_{\alpha}$.
To emphasize their origins, we reuse the notation $\Mod_H$ with $H = (\lambda,h,j_{\alpha})$ for the former and $\Mod_\lambda^{h,j_\alpha}$ for the latter.
As before, we will see that $\Mod_H g$ may exceed the range $[-\lambda,\lambda]$, while $\Mod_\lambda^{h,j_\alpha} g$ is guaranteed to stay in $[-\lambda,\lambda]$ under assumptions only on $j_\alpha$.
To allow for a larger class of folding functions, we introduce a reset time $\sigma>0$ to $\Mod_\lambda^{h,j_\alpha}$ as an additional model parameter and will see that the resulting operator $\Mod_{\lambda,\sigma}^{h,j_\alpha}$ exceeds the range $[-\lambda,\lambda]$ in a controlled way, which has also been reported for some modulo ADC prototypes in~\cite{Zhu2025,Li2026,Florescu2022}.
We study the mathematical properties of the three new operators we refer to as \emph{generalized modulo hysteresis operators}.
To this end, we focus on Lipschitz continuous input functions, which include the classical bandlimited functions, but also other function classes studied in the literature like B-splines, recently discussed in~\cite{Guo2025a}.
Our generalized modulo hysteresis operators lead to a \emph{perturbed} decomposition of the signal $g$ into integer and fractional part.
Under the additional assumption that $g$ is bandlimited, we formulate conditions under which the signal can be identified by its perturbed fractional part $g_\lambda$.
For inversion we then describe a generic recovery approach that can be adapted to the specific model to improve reconstruction quality.
This yields a flexible and effective encoding scheme for bandlimited signals, which is illustrated in Fig.~\ref{fig:Mod_recon_quantized} using only $5$ bits for encoding.

\paragraph*{Contribution}
We advance the existing art and theory of generalized modulo encoders by
\begin{enumerate}[label = \roman*)]
\item developing flexible models for modulo hysteresis with general folding transition functions that can be adapted to pre-described desired properties of the encoding and its actual hardware implementation,\label{item:C1}
\item rigorously analysing the mathematical properties of the introduced generalized modulo hysteresis operators and conditions that guarantee
\begin{itemize}
    \item bounded range and folding at $\pm \lambda$,
    \item separation of folds,
    \item return property for sufficiently large $t \in \R$,
\end{itemize}\label{item:C2}
\item proving identifiability conditions for bandlimited functions from generalized modulo samples with minimal oversampling,\label{item:C3}
\item describing a generic reconstruction algorithm from discrete data based on orthogonal matching pursuit, which is applicable to all settings.\label{item:C4}
\end{enumerate}
Our accurate modelling in \ref{item:C1} allows us to adapt the reconstruction algorithm in \ref{item:C4} to the respective generalized modulo operator, which leads to improved recovery results especially if the transient parameter is larger than the sampling rate.

\paragraph*{Paper overview}
This paper is organized as follows.
In Section~\ref{sec:GenMod} we introduce and analyze our novel generalized modulo hysteresis encoders.
In Section~\ref{sec:comparison} we compare the different models and summarize the mathematical properties in Table~\ref{tab:Mod_comparison}.
Section~\ref{sec:reconstruction} is devoted to our recovery approach with theoretical guarantees and adaptation to the respective modulo operators.
In Section~\ref{sec:numerics} we provide numerical experiments which showcase the methods developed in this paper.
Finally, we summarize our results and point to future research directions in Section~\ref{sec:conclusion}.

\section{Generalized Modulo Hysteresis Operators} \label{sec:GenMod}

We start by recalling the generalized modulo encoder $\Mod_{H}$ from \cite{Florescu2022}, where $H = (\lambda,h,\alpha)$ with modulo threshold $\lambda$, hysteresis parameter $h$ and transient parameter $\alpha$.
For a real-valued function $g \equiv g(t)$ it is defined via a sequence of folding points $(\tau_p)_{p\in\N}$, where $\tau_1 = \inf\{t > \tau_0 \mid \Mod_\lambda (g(t) + \lambda) = 0\}$ and $\tau_{p+1} = \inf\{t > \tau_p \mid \Mod_\lambda(g(t) - g( \tau_p) + h s_p) = 0\}$ with $s_p = \sgn(g(\tau_p) - g(\tau_{p-1}))$.
Then, for $t \in \R$ the output is given by $\Mod_H g(t) = g(t) - (2\lambda - h) \sum_{p \in \N} s_p \, \varepsilon_\alpha(t - \tau_p)$ with $\varepsilon_0(t) = \ind_{[0, \infty)}(t)$ and $\varepsilon_\alpha(t) = \frac{t}{\alpha} \, \ind_{[0, \alpha)}(t) + \ind_{[\alpha, \infty)}(t)$.

\begin{figure}
	\centering
	\includegraphics[width=\linewidth]{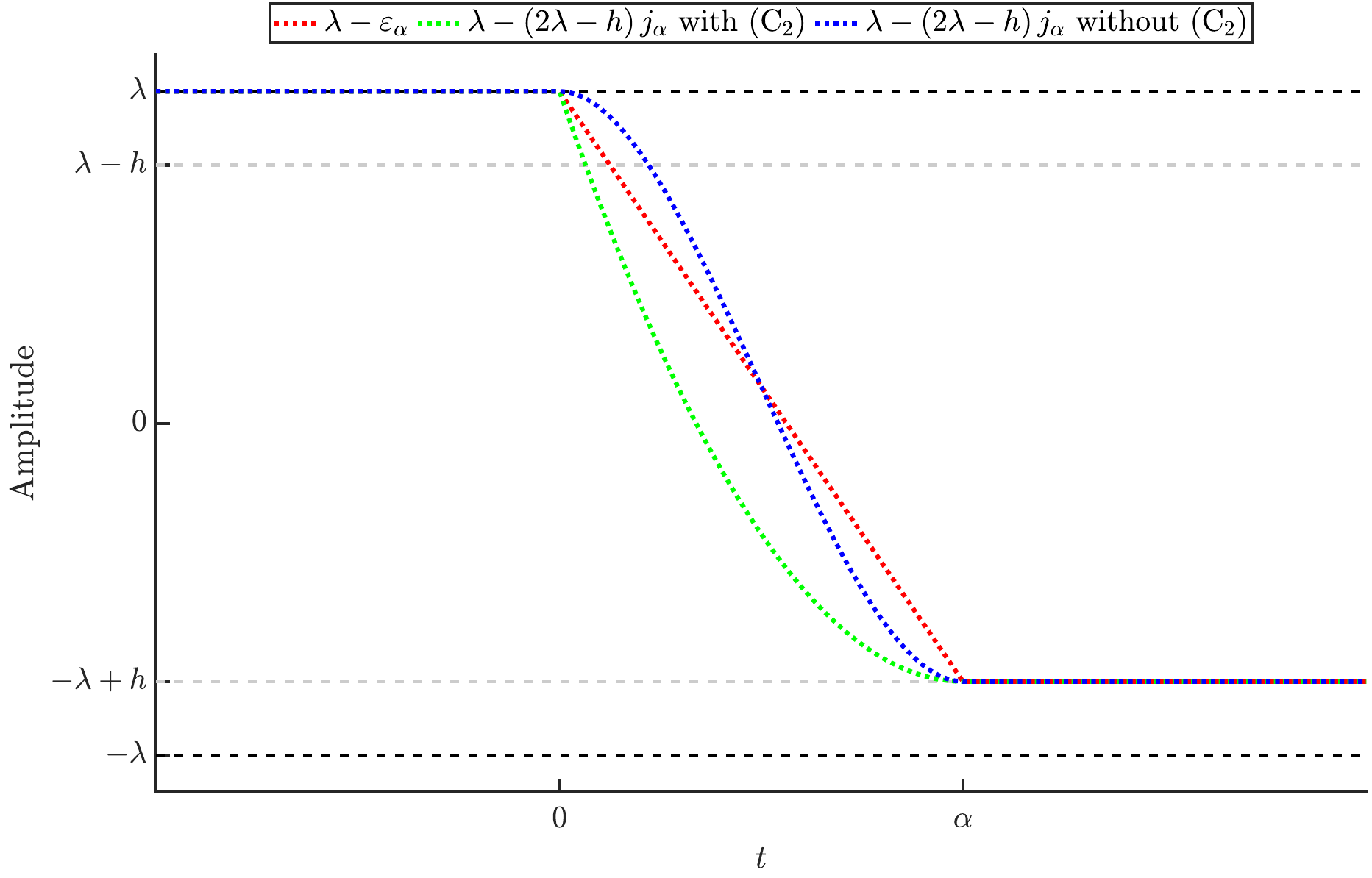}
	\caption{Illustration of the transition function in~\cite{Florescu2022, Beckmann2025a} (red) and a folding function from Definition~\ref{def:fold_fun} satisfying~\eqref{eq:fold_con2} (green) and violating~\eqref{eq:fold_con2} (blue).}
	\label{fig:FoldFun_example}
\end{figure}

To account for limitation \ref{item:P2}, we now generalize this operator by incorporating a function that models the way how a non-instantaneous fold with transient $\alpha$ is applied.
In the following, this function is referred to as {\em folding function}, see Fig.~\ref{fig:FoldFun_example} for illustration.	

\begin{definition} \label{def:fold_fun}
	A function $ j_\alpha \colon \R \to \R $ with transient $\alpha \geq 0$ is called {\em folding function} if $j_0 = \ind_{[0, \infty)}$ and, for $\alpha > 0$, $j_\alpha$ is continuous with
	\begin{gather*}
		\forall \, t \leq 0 \colon ~ j_{\alpha}(t) = 0, \\
		\forall \, t \geq \alpha \colon ~ j_{\alpha}(t) = 1.
	\end{gather*}
	When stated explicitly, we additionally assume that
	\begin{equation}\label{eq:fold_con1}
		\forall \, t \in [0,\alpha] \colon ~ \frac{t}{\alpha} \leq j_{\alpha}(t) \leq 1
		\tag{$\text{C}_1$}
	\end{equation}
	and that the right derivative $\partial_+ j_\alpha$ exists on $[0,\alpha]$ with
	\begin{equation}\label{eq:fold_con2}
		\forall \, s \leq t \in [0,\alpha] \colon ~ 0 \leq \partial_+ j_{\alpha}(t) \leq \partial_+ j_{\alpha}(s).
		\tag{$\text{C}_2$}
	\end{equation}
\end{definition}

\subsection{Transient independent folding points}

In the following, we consider the triplet $H = (\lambda,h,j_\alpha)$ and modify the definition of the generalized modulo encoder $\Mod_H$ accordingly.
To ensure well-definedness, we restrict ourselves to the function class
\begin{equation*}
	\Cont_{\lambda,\tau_0}^{0,1} = \bigl\{g \in \Cont^{0,1}(\R) \bigm| |g(t)| < \lambda ~ \forall \, t \leq \tau_0\bigr\}
\end{equation*}
with arbitrary $\tau_0 \in \R$, where $\Cont^{0,1}(\R)$ denotes the space of Lipschitz continuous functions on $\R$, which are differentiable almost everywhere with essentially bounded derivative, whose essential supremum yields a Lipschitz constant of the function.

\begin{definition} \label{def:Mod_H}
	Let $\lambda > 0$, $h \in (0, 2\lambda)$ and $j_\alpha$ be a folding function with transient $\alpha \geq 0$.
	For a function $g \in \Cont_{\lambda,\tau_0}^{0,1}$ with $\tau_0 \in \R $ let $(\tau_p)_{p \in \N}$ be given by
	\begin{equation} \label{eq:Mod_H_tau}
		\begin{gathered}
			\tau_1 = \inf\bigl\{t > \tau_0 \bigm| \Mod_\lambda(g(t)+ \lambda) = 0\bigr\}, \\
			\tau_{p+1} = \inf\bigl\{t > \tau_{p} \bigm| \Mod_\lambda(g(t) - g(\tau_p) + h s_p) = 0\bigr\},
		\end{gathered}
	\end{equation}
	where $s_p = \sgn(g(\tau_p) - g(\tau_{p-1}))$ if $\tau_p < \infty$ and $\inf \emptyset = \infty$.
	Then, the output of the {\em generalized modulo encoder} $\Mod_H$ with $H = (\lambda,h,j_\alpha)$ is defined as
	\begin{equation} \label{eq:Mod_H}
		\Mod_H g (t) = g(t) - (2 \lambda -h) \sum\nolimits_{\tau_p < \infty} s_p \, j_{\alpha}(t-\tau_{p}).
	\end{equation}
\end{definition}

Note that choosing $j_\alpha = \varepsilon_\alpha$ in Definition~\ref{def:Mod_H} yields the same generalized modulo encoder as in~\cite{Florescu2022}, which justifies the clash of notation.
We now study elementary properties of $\Mod_H$ and begin with its well-definedness by proving that the series in~\eqref{eq:Mod_H} converges pointwise.

\begin{theorem} \label{theo:Mod_H_separation}
	Let $g \in \Cont_{\lambda,\tau_0}^{0,1}$.
	Then, the folding points $(\tau_p)_{p \in \N}$ satisfy the separation property
	\begin{equation*}
		\tau_p < \infty
		\implies
		\tau_{p+1} - \tau_p \geq \frac{\min\{h, 2\lambda-h\}}{\|g^\prime\|_\infty} > 0
	\end{equation*}		
	and the series $\sum\nolimits_{p \in \N} s_p \, j_{\alpha}(t-\tau_{p})$ converges for all $t \in \R$.
\end{theorem}

\begin{proof}
	Let $\tau_{p} < \infty$ and assume that $\tau_{p+1} < \infty$.
	Otherwise, the separation is trivially true.
	By definition, $\tau_{p+1}$ is the first point larger than $\tau_p$ which satisfies $\Mod_\lambda(g(\tau_{p+1}) - g(\tau_p) + h s_p) = 0$.
	Because of the continuity of $g$ this is the first point where $g(\tau_{p+1}) - g( \tau_p) + h s_p = 0$ or $g(\tau_{p+1}) - g(\tau_p) + h s_p = 2\lambda s_p$.
	Therefore, this implies that 
	\begin{equation*}
		\begin{aligned}
			|g(\tau_{p+1}) - g( \tau_p)| \geq \min\{h, 2\lambda-h\}.
		\end{aligned}
	\end{equation*}
	As $g$ is Lipschitz continuous, the difference $|g(\tau_{p+1}) - g(\tau_p)|$ can be bounded via $|g(\tau_{p+1}) - g(\tau_p)| \leq \|g^\prime\|_\infty \, (\tau_{p+1} - \tau_p)$.
	With this, we obtain the stated separation
	\begin{equation*}
		\begin{aligned}
			\tau_{p+1}  - \tau_p \geq \frac{ \min \{ h, 2\lambda -h \}  }{ \| g' \|_{\infty} } .
		\end{aligned}
	\end{equation*}
	Now, for fixed $t \in \R$ there exists $q \in \N$ such that $t \leq \tau_{q+1}$, e.g., $q = \max\bigl\{0 ,\bigl\lceil (t - \tau_0) \, \frac{\|g'\|_\infty}{\min\{h, 2\lambda-h\}} \bigr\rceil\bigr\}$.
	Therefore, the series
	\begin{equation*}
		\sum\nolimits_{p \in \N} s_p \, j_{\alpha}(t-\tau_{p}) = \sum\nolimits_{p = 1}^{q+1} s_p \, j_{\alpha}(t-\tau_{p}).
	\end{equation*}
	reduces to a finite sum and, hence, converges.
\end{proof}

Theorem~\ref{theo:Mod_H_separation} shows that $\Mod_H$ is well-defined on $\Cont_{\lambda,\tau_0}^{0,1}$.
We now study conditions to guarantee that $\Mod_H g$ returns to $g$ for sufficiently large $t \in \R$, see also Fig.~\ref{fig:Mod_H_return}.
To this end, for $\delta > 0$ and $\rho \in \R$ we consider the function class
\begin{equation*}
	\B^\delta_\rho = \bigl\{g\colon \R \to \R \bigm| |g(t)| < \delta ~ \forall \, t \geq \rho\bigl\}.
\end{equation*}
	
\begin{figure}
	\centering
	\includegraphics[width=\linewidth]{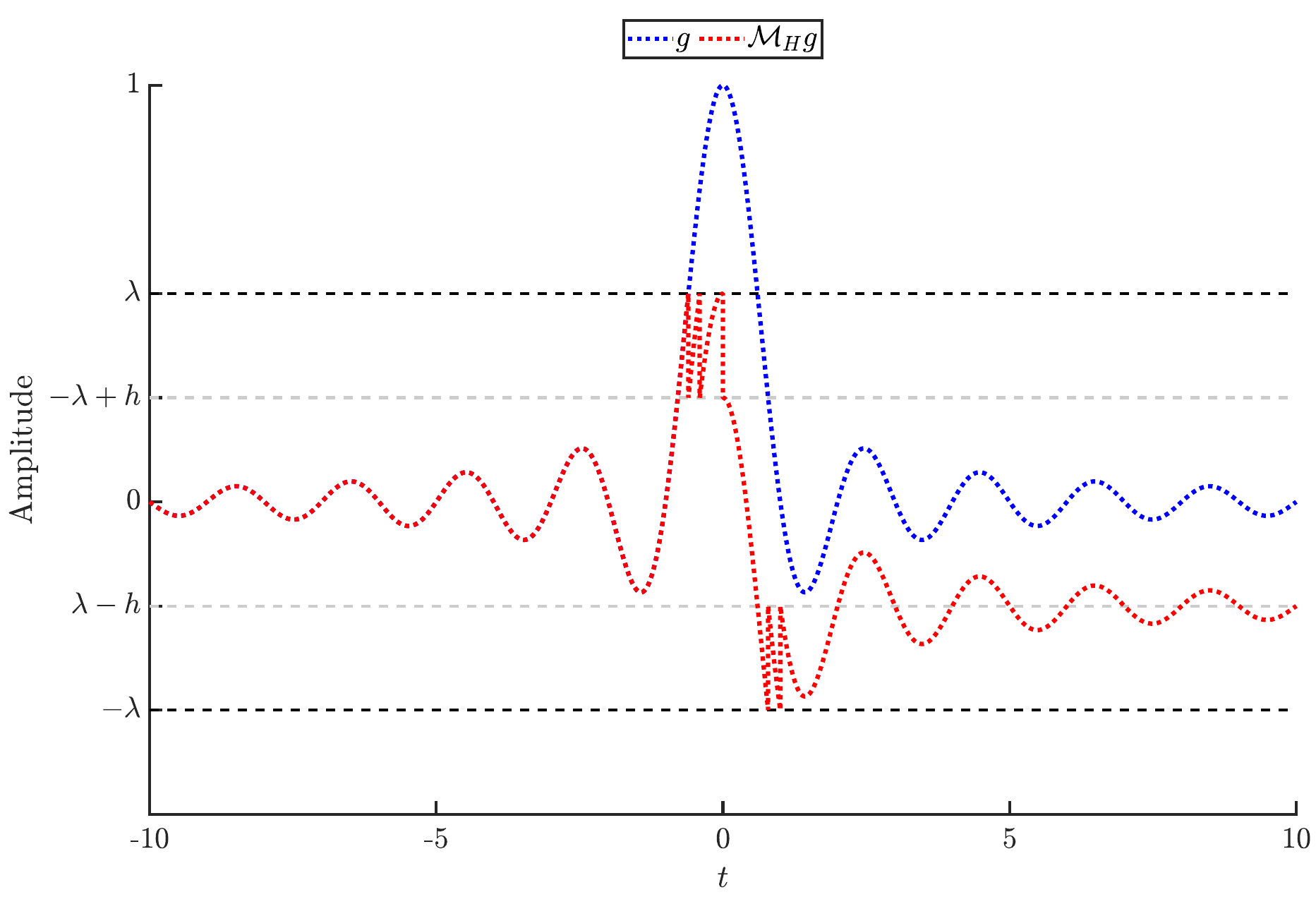}
	\caption{Illustration of the failure of the return property in the case $h > \lambda$ using $g(t) = \sinc(t)$, where we choose $\lambda = 0.5$, $h = 0.75$ and $\alpha = 0$.}
	\label{fig:Mod_H_return}
\end{figure}

\begin{proposition} \label{prop:Mod_H_return}
	Let $0<h<\lambda$ and $g \in \Cont^{0,1}_{\lambda,\tau_0} \cap \B^{\lambda-h}_\rho$ with $\rho > \tau_0$.
	Then, $\Mod_H g(t) = g(t)$ for all $t \geq \rho + \alpha$.
\end{proposition}

\begin{proof}
	First consider $\alpha = 0$.
	Then, we can write $\Mod_H g(t) = g(t) + c \, (2 \lambda-h)$ for some $c \in \Z$ dependent on $t$.
	As $g \in \B^{\lambda-h}_\rho$, for $t \geq \rho$ we get $g(t) + c \, (2 \lambda-h) < \lambda - h + c \, (2 \lambda-h)$, which can only be larger than or equal to $-\lambda$ if $c \geq 0$.
	Analogously, we get $g(t) + c \, (2 \lambda-h) > -\lambda + h + c \, (2 \lambda-h)$, which can only be smaller than or equal to $\lambda$ if $c \leq 0$.
	This implies that $c = 0$ so that indeed $\Mod_H g(t) = g(t)$ for all $t \geq \rho$.
	
	To prove the statement for $\alpha > 0$, consider $H_0 = (\lambda,h,j_0)$.
	We have already proven that $\Mod_{H_0} g(t) = g(t)$ for all $t \geq \rho$ and $\tau_p < \rho$ or $\tau_p = \infty$.
	Because $\Mod_{H_0} g(t) = \Mod_H g(t)$ by definition for all $t \notin \bigcup_{\tau_p < \infty} [\tau_p, \tau_p + \alpha) \subset [\tau_0, \rho + \alpha)$ we also have  $\Mod_H g(t) = g(t)$ for all $t \geq \rho + \alpha$.
\end{proof}

If $0 < h< \lambda$ and $g \in \B^{\lambda-h}_\rho$ with $\rho > \tau_0$, Proposition~\ref{prop:Mod_H_return} implies that $\tau_{p} < \rho$ or $\tau_p = \infty$.
This in combination with Theorem~\ref{theo:Mod_H_separation} gives
\begin{equation*}
	\sum\nolimits_{p \in \N} s_p \, j_{\alpha}(t-\tau_{p}) = \sum\nolimits_{p =1}^{q+1} s_p \, j_{\alpha}(t-\tau_{p})
\end{equation*}
with $q = \max\bigl\{0 ,\bigl\lceil(\rho - \tau_0) \, \frac{\|g^\prime\|_{\infty}}{h}\bigr\rceil\bigr\}$ for all $t \in \R$.
Therefore, we can guarantee uniform convergence of $\sum\nolimits_{p \in \N} s_p \, j_{\alpha}(\cdot-\tau_{p})$ for $0 < h < \lambda$ and $g \in \B^{\lambda-h}_\rho$. 

\begin{figure}
	\centering
	\includegraphics[width=\linewidth]{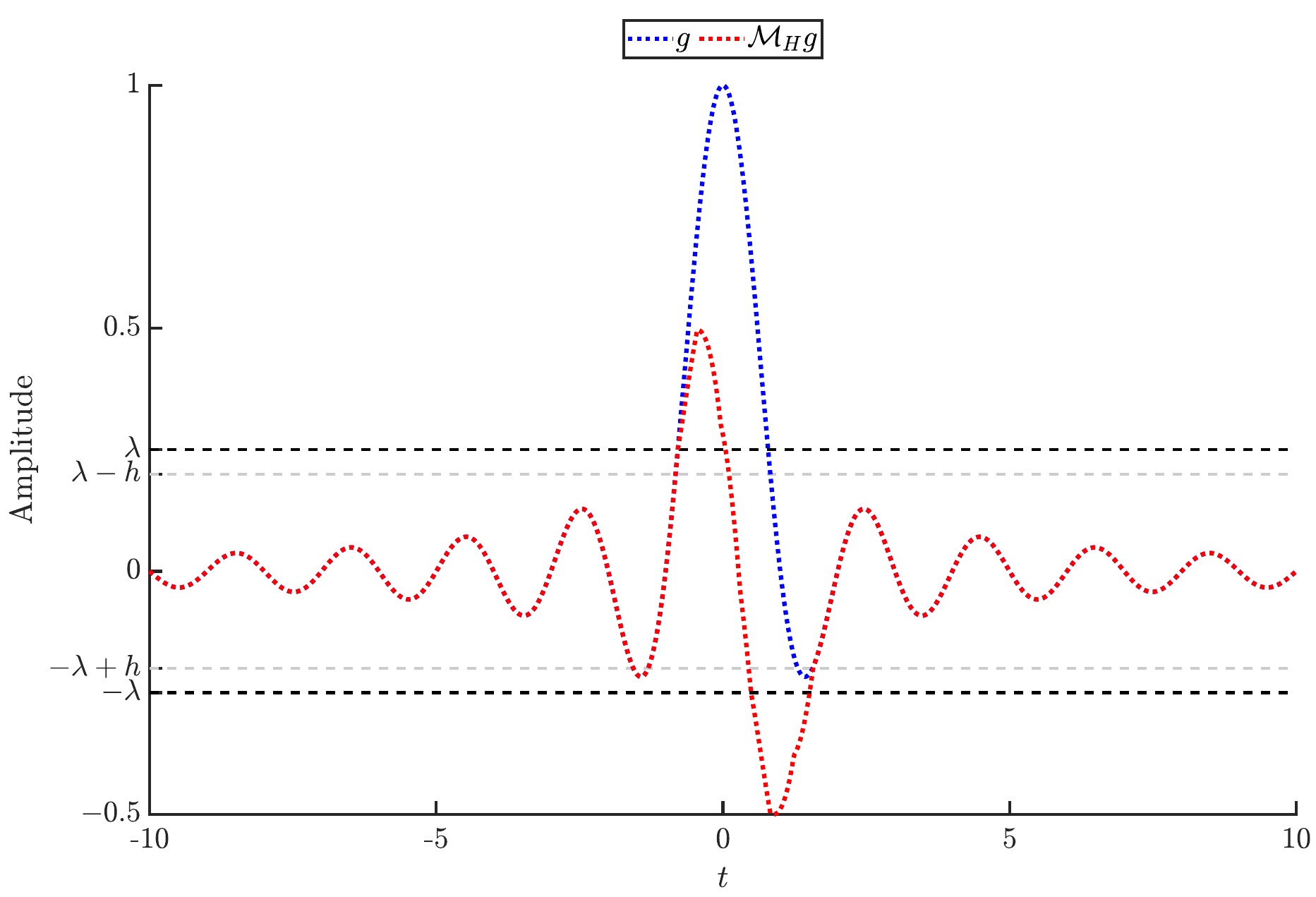}
	\caption{Example for $\Mod_H g (t) \notin [- \lambda, \lambda]$ using $g(t) = \sinc(\pi t)$, where we choose $j_\alpha(t) = \frac{t}{\alpha}$ for $t \in [0,\alpha]$, $\lambda = 0.25$, $h = 0.05$ and $\alpha = 0.75$.}
	\label{fig:Mod_H_exceedance}
\end{figure}

Next, we discuss limitation \ref{item:P1} and study conditions under which $\Mod_H g$ stays in the prescribed range $[-\lambda,\lambda]$.
In the case $\alpha = 0$, the condition $g \in \Cont^{0,1}_{\lambda,\tau_0}$ suffices to guarantee that $|\Mod_H g(t)| \leq \lambda$ for all $t \in \R$ and $|\Mod_{H} g (\tau_p)| = \lambda - h$ if $\tau_{p} < \infty$.
However, by defining the sequence of folding points $(\tau_p)_{p \in \N}$ independently of the transient $\alpha$ it may occur that $\Mod_H g (t) \notin [- \lambda, \lambda]$ and $|\Mod_H g(\tau_p)| \neq \lambda$ if $\alpha > 0$, see Fig.~\ref{fig:Mod_H_exceedance}.
To prevent this, we need additional conditions on the first and second derivative of the input function and consider the function class
\begin{equation*}
	\Cont^{1,1}_{\lambda,\tau_0} = \bigl\{g \in \Cont^{0,1}_{\lambda,\tau_0} \bigm| g^\prime \in \Cont^{0,1}(\R)\bigr\}.
\end{equation*}

\begin{theorem} \label{theo:Mod_H_range}
	Let $\alpha > 0$ and let $j_\alpha$ satisfy~\eqref{eq:fold_con1}.
	Moreover, let $g \in \Cont^{1,1}_{\lambda,\tau_0}$ with $\|g^\prime\|_\infty \leq \frac{2 \lambda - h}{\alpha}$ and $\|g^{\prime\prime}\|_\infty \leq \frac{2 h}{\alpha^2}$.
	Then, we have $|\Mod_H g(t)| \leq \lambda$ for all $t \in \R$.
	If $\tau_0 < \tau_p < \infty$, we further obtain $\tau_{p+1} - \tau_p \geq \alpha$ and $|\Mod_H g(\tau_p)| = \lambda$.
\end{theorem}

\begin{proof}
	We start with showing $\tau_{p+1} - \tau_p \geq \alpha$ for $\tau_0 < \tau_p < \infty$.
	If $\tau_{p+1} < \infty$, we either have $g(\tau_{p+1}) = g(\tau_{p}) - h s_p$ and $s_{p+1} = -s_p$ or $g(\tau_{p+1}) = g(\tau_{p}) + (2\lambda - h) s_p$ and $s_{p+1} = s_p$.
	In the latter case, the assumption $\|g^\prime\|_\infty \leq \frac{2 \lambda - h}{\alpha}$ yields
	\begin{equation*}
		\tau_{p+1} - \tau_p \geq \frac{|g(\tau_{p+1}) - g(\tau_p)|}{\|g^\prime\|_\infty} \geq \alpha.
	\end{equation*}
	In the former case, we have $s_p \cdot g^\prime(\tau_p) \geq 0$ and $\|g^{\prime\prime}\|_\infty \leq \frac{2 h}{\alpha^2}$ yields
	\begin{equation*}
		\tau_{p+1} - \tau_p \geq \sqrt{\frac{2h}{\|g^{\prime\prime}\|_\infty}} \geq \alpha.
	\end{equation*}
	We now prove by induction on $p \in \N$ that $|\Mod_H g(t)| \leq \lambda$ for $t < \tau_p$ and $\Mod_H g(\tau_p) = \lambda s_p$ if $\tau_0 < \tau_p < \infty$.
	Since $\tau_p \to \infty$ for $p \to \infty$ due to Theorem~\ref{theo:Mod_H_separation}, this implies that $|\Mod_H g(t)| \leq \lambda$ for all $t \in \R$.
	By definition of $\tau_1$ we have $|\Mod_H g(t)| \leq \lambda$ for all $t < \tau_1$ as well as $\Mod_H g(\tau_1) = g(\tau_1) = \lambda s_1$ if $\tau_1 < \infty$.
	Now, let $p \in \N$ such that $|\Mod_H g(t)| \leq \lambda$ for $t < \tau_p$ and $\Mod_H g(\tau_p) = \lambda s_p$ if $\tau_p < \infty$.
	For $t \in (\tau_p, \tau_{p+1})$ we can write
	\begin{equation*}
		\Mod_H g(t) = \Mod_H g(\tau_p) - g(\tau_p) + g(t) - (2\lambda - h) \, s_p \, j_\alpha(t - \tau_p),
	\end{equation*}
	where $0 \leq j_\alpha \leq 1$.
	If $s_p = 1$, we have $g(\tau_p) - h \leq g(t) \leq g(\tau_p) + 2\lambda - h$ as well as $\Mod_H g(\tau_p) = \lambda$ and can directly conclude that
	\begin{equation*}
		\Mod_H g(t) \geq \lambda - g(\tau_p) + g(\tau_p) - h - 2\lambda + h = -\lambda.
	\end{equation*}
	Moreover, for $t \geq \tau_p + \alpha$ we have $j_\alpha(t - \tau_p) = 1$ and, therefore, 
	\begin{align*}
		\Mod_H g(t) &= \Mod_H g(\tau_p) - g(\tau_p) + g(t) - 2\lambda + h \\
		&\leq \lambda - g(\tau_p) +  g(\tau_p) + 2\lambda - h - 2\lambda + h = \lambda.
	\end{align*}
	For $t \leq \tau_p + \alpha$, condition~\eqref{eq:fold_con1} gives $j_\alpha(t - \tau_p) \geq \frac{t - \tau_p}{\alpha}$ and, hence,
	\begin{align*}
		\Mod_H g(t) &= \Mod_H g(\tau_p) - (2\lambda - h) \, j_\alpha(t - \tau_p) + g(t) - g(\tau_p) \\
		&\leq \lambda - (2\lambda - h) \, \frac{t - \tau_p}{\alpha} + (t - \tau_p) \, \|g^\prime\|_\infty \leq \lambda.
	\end{align*}
	On the other hand, if $s_p = -1$, we have $g(\tau_p) - (2\lambda - h) \leq g(t) \leq g(\tau_p) + h$ and $\Mod_H g(\tau_p) = -\lambda$ so that
	\begin{equation*}
		\Mod_H g(t) \leq -\lambda - g(\tau_p) + g(\tau_p) + h + 2\lambda - h = \lambda.
	\end{equation*}
	As before, for $t > \tau_p + \alpha$, we have $j_\alpha(t - \tau_p) = 1$ and, hence,
	\begin{align*}
		\Mod_H g(t) &= \Mod_H g(\tau_p) - g(\tau_p) + g(t) + 2\lambda - h \\
		&\geq -\lambda - g(\tau_p) +  g(\tau_p) - (2\lambda - h) + 2\lambda - h = -\lambda.
	\end{align*}
	For $t \leq \tau_p + \alpha$, condition~\eqref{eq:fold_con1} gives $j_\alpha(t - \tau_p) \geq \frac{t - \tau_p}{\alpha}$ and
	\begin{align*}
		\Mod_H g(t) &= \Mod_H g(\tau_p) + (2\lambda - h) \, j_\alpha(t - \tau_p) + g(t) - g(\tau_p) \\
		&\geq -\lambda + (2\lambda - h) \, \frac{t - \tau_p}{\alpha} - (t - \tau_p) \, \|g^\prime\|_\infty \geq \lambda.
	\end{align*}
	Consequently, we have $|\Mod_H g(t)| \leq \lambda$ for $t < \tau_{p+1}$.
	Finally, if $\tau_{p+1} < \infty$, we either have $g(\tau_{p+1}) = g(\tau_{p}) - h s_p$ and $s_{p+1} = -s_p$ implying $\Mod_H g(\tau_{p+1}) = \lambda s_p + (2\lambda - h) s_p - (2\lambda - h) s_p = \lambda s_{p+1}$ or $g(\tau_{p+1}) = g(\tau_{p}) + (2\lambda - h) s_p$ and $s_{p+1} = s_p$ so that $\Mod_H g(\tau_{p+1}) = \lambda s_p - h s_p - (2\lambda - h) s_p = \lambda s_{p+1}$.
\end{proof}

\subsection{Transient dependent folding points}

As described in limitation \ref{item:P1} and seen in Fig.~\ref{fig:Mod_H_exceedance}, the above definition of $\Mod_H$ does not guarantee that $\Mod_\lambda g(t) \in [-\lambda,\lambda]$ for $t \geq \tau_0$.
One reason for this is that the folding points $(\tau_p)_{p \in \N}$ are independent of the folding function~$j_\alpha$ if $\alpha > 0$, as already indicated in limitation \ref{item:P3}.
To account for this, Theorem~\ref{theo:Mod_H_range} included an additional condition on the growth of the input function.
Instead, we now wish to propose a modified definition of the generalized modulo encoder.
To this end, we rewrite $\Mod_H g$ by defining the function sequence $(\eta_n^{(\alpha)})_{n \in \N_0}$ via
\begin{equation*}
\eta_{n+1}^{(\alpha)} = \eta_n^{(\alpha)} - (2\lambda - h) \sgn(\eta_n^{(\alpha)}(\tau_{n+1})) \, j_\alpha(\cdot - \tau_{n+1})
\end{equation*}
if $\tau_{n+1} < \infty$ and $\eta_{n+1}^{(\alpha)} = \eta_n^{(\alpha)}$ if $\tau_{n+1} = \infty$,
where $\eta_0^{(\alpha)} = g$.
With this, we then obtain
\begin{equation*}
\tau_{n+1} = \inf \bigl\{t > \tau_n \bigm| |\eta_n^{(0)}(t)| \geq \lambda \bigr\}
\end{equation*}
and
\begin{equation*}
\Mod_H g(t) = \lim_{n \to \infty} \eta_{n}^{(\alpha)}(t).
\end{equation*}
The idea is to make $\tau_{n+1}$ dependent on $\eta_n^{(\alpha)}$ rather than $\eta_n^{(0)}$.

\begin{definition} \label{def:Mod_j}
	Let $\lambda > 0$, $h < 2\lambda$ and $j_\alpha$ be a folding function with transient $\alpha > 0$ satisfying~\eqref{eq:fold_con2}.
	For a function $g \in  \Cont_{\lambda,\tau_0}^{0,1}$ with $\tau_0 \in \R$ we define the two sequences $(\kappa_n)_{n \in \N_0}$ and $(\zeta_n)_{n \in \N_0}$ via
	\begin{equation} \label{eq:Mod_j_kappa}
		\begin{gathered}
			\kappa_{n+1} = \inf \bigl\{t>\kappa_n \bigm| |\zeta_n(t)|\geq \lambda \bigr\},\\
			\zeta_{n+1} = \zeta_n - (2 \lambda -h) \sgn(\zeta_n(\kappa_{n+1})) \, j_{\alpha}(\cdot - \kappa_{n+1})
		\end{gathered}
	\end{equation}
	and $\zeta_{n+1} = \zeta_n$ if $\kappa_{n+1} = \infty$,
	where $\kappa_0 = \tau_0$ and $\zeta_0 = g$.
	Then, the output of the {\em generalized modulo operator} $\Mod_\lambda^{h,j_\alpha}$ is defined as
	\begin{equation} \label{eq:Mod_j}
		\Mod_\lambda^{h,j_\alpha} g(t) = \lim\nolimits_{n \to \infty} \zeta_n(t).
	\end{equation}
\end{definition}

Note that for $\alpha = 0$ and $0 < h < 2\lambda$ Definition~\ref{def:Mod_j} would yield the same generalized modulo encoder as Definition~\ref{def:Mod_H} so that $\Mod_\lambda^{h,j_0} = \Mod_{H_0}$ if $h > 0$.
Hence, we require that $\alpha > 0$, in which case also a negative $h < 2\lambda$ is applicable.
We now study basic properties of $\Mod_\lambda^{h,j_\alpha}$ and begin with its well-definedness by proving that the sequence in~\eqref{eq:Mod_j} converges.

\begin{proposition} \label{prop:Mod_j_convergence}
	For $g \in \Cont_{\lambda,\tau_0}^{0,1}$ the function sequence $(\zeta_n)_{n \in \N_0}$ converges pointwise and for every $t \in \R$ there exists $p \in \N_0$ with $\Mod_\lambda^{h,j_\alpha} g(t) = \zeta_p(t)$.
\end{proposition}

\begin{proof}
    We first prove the statement for $g \in \Cont_{\lambda,\tau_0}^{0,1} \cap \Lebesgue^{\infty}(\R)$ by showing that $(\kappa_{n})_{n \in \N_0}$ diverges to infinity.
	Then, for each $t \in \R$ there exists $p \in \N_0$ such that $t < \kappa_p$ and we obtain $\zeta_n(t) = \zeta_p(t)$ for all $n \geq p$.
	
	We consider the subsequence $(\kappa_{m_n})_{n \in \N_0}$ with $\kappa_{m_1} = \kappa_1$, where $\kappa_{m_{n+1}}$ is the next folding point with different sign, i.e.\ $\kappa_{m_{n+1}} = \inf_{k > m_n}\{\kappa_k \mid \sgn(\zeta_{k-1}(\kappa_k)) \neq \sgn(\zeta_{m_n} (\kappa_{m_n}))\}$.
	We now show that $m_{n+1} - m_n \leq n \, c_g$ by induction on $n \in \N$ with $c_g = \bigl\lceil\frac{\|g\|_\infty}{ \lambda}\bigr\rceil \cdot \bigl\lceil\frac{\alpha \, \|g^\prime\|_\infty}{2\lambda -h}\bigr\rceil$.
	For $n = 1$, we directly have
	\begin{equation*}
		m_2 - m_1 \leq \biggl\lceil \frac{\|g\|_\infty}{\lambda}\biggr\rceil \cdot \biggl\lceil\frac{\alpha \, \|g^\prime\|_\infty}{2\lambda - h}\biggr\rceil = c_g,
	\end{equation*}
	where $\bigl\lceil\frac{\|g\|_\infty}{\lambda}\bigr\rceil$ bounds the number of folds with distance larger than or equal to $\alpha$ and $\bigl\lceil\frac{\alpha \, \|g^\prime\|_\infty}{2\lambda - h}\bigr\rceil$ bounds the number of folds that can happen while resetting in an interval of length $\alpha$.
	For $n \geq 2$, we have to incorporate the case that previous non-instantaneous folds may induce additional folds, leading to
	\begin{align*}
		m_{n+1} - m_n &\leq \biggl\lceil\frac{\|g\|_\infty}{\lambda}\biggr\rceil \cdot \biggl\lceil\frac{\alpha \, \|g^\prime\|_\infty}{2\lambda - h}\biggr\rceil \\
		&\qquad +  \biggl\lceil  \frac{\partial_+ (\zeta_{m_n-1}(\kappa_{m_n-1}) - g(\kappa_{m_n-1}))}{(2\lambda - h) \cdot \partial_+ j_\alpha(0)} \biggr\rceil \\
		&\leq \biggl\lceil\frac{\|g\|_\infty}{\lambda}\biggr\rceil \cdot \biggl\lceil\frac{\alpha \, \|g^\prime\|_\infty}{2\lambda - h}\biggr\rceil + \biggl\lceil\sum_{i = m_{n-1}}^{m_n-1} \frac{\partial_+ j_\alpha(t_i)}{\partial_+ j_\alpha(0)} \biggr\rceil
	\end{align*}
	with certain points $t_i \in [0,\alpha]$.
	As Assumption~\eqref{eq:fold_con2} gives $\partial_+ j_\alpha(t_i) \leq \partial_+ j_\alpha(0)$, we can bound $m_{n+1} - m_n$ via
	\begin{align*}
		m_{n+1} - m_n &\leq \biggl\lceil\frac{\|g\|_\infty}{\lambda}\biggr\rceil \cdot \biggl\lceil\frac{\alpha \, \|g^\prime\|_\infty}{2 \lambda - h}\biggr\rceil + m_n - m_{n-1} \\
		&\leq  n \, \biggl\lceil\frac{\|g\|_\infty}{\lambda}\biggr\rceil \cdot \biggl\lceil\frac{\alpha \, \|g^\prime\|_\infty}{2 \lambda - h}\biggr\rceil = n \, c_g.
	\end{align*}
	By definition of $\kappa_{m_n}$, we have
	\begin{align*}
		\kappa_{m_n} &= \inf\{t>\kappa_{m_n-1} \mid |\zeta_{m_n-1}(t)|\geq \lambda\} \\
		&= \inf\{t>\kappa_{m_n-1} \mid |\zeta_{m_n-1}(t) - \zeta_{m_n-1}(\kappa_{m_n-1})| \geq 2 \lambda\}
	\end{align*}
	and, for $\kappa_{m_n-1} < t \leq \kappa_{m_n}$,
	\begin{align*}
		&|\zeta_{m_n-1}(t) - \zeta_{m_n-1}(\kappa_{m_n-1})| \\
		&\quad \leq \biggl((2\lambda - h) \, \partial_+ j_\alpha(0) \, \Bigl(m_n - m_{n-1} + \Bigl\lceil \frac{\alpha \, \|g^\prime\|_\infty}{2\lambda - h}\Bigr\rceil\Bigr)\biggr) \\
		&\quad\qquad \cdot (t - \kappa_{m_n-1}) \\
		&\quad \leq \bigl((2\lambda - h) \, \partial_+ j_\alpha(0) \, (m_n - m_{n-1} + c_g)\bigr) \, (t - \kappa_{m_n-1})
	\end{align*}
	Hence,
	\begin{equation*}
		\kappa_{m_n} \geq \kappa_{m_n-1} + \frac{2\lambda}{(2\lambda - h) \, \partial_+ j_\alpha(0)} \, (m_n - m_{n-1} + c_g)^{-1}
	\end{equation*}
	and the above estimate gives 
	\begin{equation*}
		\kappa_{m_n} \geq \kappa_{m_n-1} + \frac{2\lambda}{(2\lambda - h) \, \partial_+ j_\alpha(0)} \, \bigl(n \, c_g\bigr)^{-1}.
	\end{equation*}
	As $\kappa_{m_\ell-1} \geq \kappa_{m_{\ell-1}}$ for all $1 < \ell \leq n$, we can conclude that
	\begin{equation*}
		\kappa_{m_{n+1}} \geq \kappa_1 +\frac{2\lambda}{(2\lambda - h) \, \partial_+ j_\alpha(0) \, c_g} \sum_{\ell=1}^n \frac{1}{\ell}.
	\end{equation*}
	Consequently, the subsequence $(\kappa_{m_n})_{n \in \N}$ and, therefore, also the full sequence $(\kappa_n)_{n \in \N_0}$ diverge to infinity for $n \to \infty$.
    
    Now, assume that $g \notin \Lebesgue^{\infty}(\R)$.
    Because $g$ has an essentially bounded derivative and  $|g(s)| < \lambda$ for all $s \leq \tau_0$, we can find for any $t \in \R$ a function $\widetilde{g} \in \Cont_{\lambda,\tau_0}^{0,1} \cap \Lebesgue^{\infty}(\R)$ with $g(s) = \widetilde{g}(s)$ for all $s \leq t$.
    This also implies that
	\begin{equation*}
		\Mod_\lambda^{h,j_\alpha} g(s) = \Mod_\lambda^{h,j_\alpha} \widetilde{g}(s)
		\quad \forall \, s \leq t
	\end{equation*}	    
    and, hence, the statement is also true for $g \notin \Lebesgue^{\infty}(\R)$.
\end{proof}

As opposed to $\Mod_H$, we now show that $\Mod_\lambda^{h,j_\alpha}$ satisfies $\Mod_\lambda^{h,j_\alpha} g (t) \in [- \lambda, \lambda]$ and $|\Mod_\lambda^{h,j_\alpha} g(\kappa_{n})| = \lambda$ if $\kappa_{n} < \infty$ for all $g\in \Cont_{\lambda,\tau_0}^{0,1}$ without additional restrictions on the first and second derivative of $g$.
However, making the folding points depend on $j_\alpha$ with positive transient $\alpha > 0$ has a disadvantageous side effect, namely that additional folds can be caused by the transient of prior folds in the opposite direction.
This is illustrated in Fig.~\ref{fig:Mod_j_continuation}, where we use $j_\alpha(t) = \frac{t}{\alpha}$ for $t \in [0,\alpha]$ and observe that the folds do not end.
To prevent this, we have to make an additional assumption on the second derivative of $g$, which allows us to show that folds in different directions are separated by $\alpha$.
	
\begin{theorem} \label{theo:Mod_j_range_separation}
	Let $g\in \Cont_{\lambda,\tau_0}^{0,1}$.
	Then, $\Mod_\lambda^{h,j_\alpha} g(t) \in [-\lambda,\lambda]$ for all $t \in \R$ and $|\Mod_\lambda^{h,j_\alpha} g(\kappa_n)| = \lambda$ if $\kappa_n < \infty$.
	If, in addition, $g \in \Cont^{1,1}_{\lambda,\tau_0}$ with $\|g^{\prime\prime}\|_\infty \leq \frac{2 h}{\alpha^2}$, then, for $n \in \N$ and $\kappa_{n+1} < \infty$ we have
	\begin{equation*}
		\Mod_\lambda^{h,j_\alpha} g(\kappa_n) \cdot \Mod^{h,j_\alpha}_\lambda g(\kappa_{n+1}) < 0
		\implies
		\kappa_{n+1} - \kappa_n \geq \alpha.
	\end{equation*}
\end{theorem}
	
\begin{proof}
	Since $\Mod_\lambda^{h,j_\alpha} g(t) = \zeta_{n-1}(t)$ for all $t \leq \kappa_n$, the definition in~\eqref{eq:Mod_j_kappa} gives $|\Mod_\lambda^{h,j_\alpha} g(\kappa_n)| = |\zeta_{n-1}(\kappa_n)| = \lambda$.
	Now assume that there is $t \in (\kappa_n, \kappa_{n+1})$ so that $|\Mod_\lambda^{h,j_\alpha} g(t)| > \lambda$.
	This in turn implies that $|\zeta_n(t)| > \lambda$ and, hence, $t > \kappa_{n+1}$ in contradiction to the assumption.
	To prove the separation property, as in the proof of Proposition~\ref{prop:Mod_j_convergence} consider the subsequence $ (\kappa_{m_n})_{n \in \N}$, where $\kappa_{m_1} = \kappa_{1}$ and $\kappa_{m_{n+1}}$ is the next folding point with different sign, i.e.,
	\begin{equation*}
		\kappa_{m_{n+1}} = \inf_{k > m_n} \bigl\{\kappa_k \bigm| \sgn(\zeta_{k-1}(\kappa_k)) \neq \sgn(\zeta_{m_n}(\kappa_{m_n}))\bigr\} 
	\end{equation*}
	so that $\Mod_\lambda^{h,j_\alpha} g(\kappa_{m_n}) = \pm \lambda  = -\Mod_\lambda^{h,j_\alpha} g(\kappa_{m_{n+1}})$.
	Without loss of generality let $\Mod_\lambda^{h,j_\alpha} g(\kappa_{m_1}) = \lambda$.
	Then, we obtain $\Mod_\lambda^{h,j_\alpha} g(\kappa_{m_2-1}) = \lambda$ and $\Mod_\lambda^{h,j_\alpha} g(\kappa_{m_2}) = -\lambda$.
	Additionally, the monotonicity of $\partial_+ j_\alpha$ due to Assumption~\eqref{eq:fold_con2} gives $g^\prime(\kappa_{m_2-1}) \geq \partial_+\bigl(\sum\nolimits_{\kappa \in K \setminus \{\kappa_{m_2-1}\}} (2\lambda - h) \, j_\alpha(t - \kappa)\bigr)$ for $t > \kappa_{m_2-1}$ and with $K = \{\kappa_k \in (\kappa_{m_2-1} - \alpha, \kappa_{m_2-1}]\}$.
	Therefore, $g^\prime(\kappa_{m_2-1}) \, (t - \kappa_{m_2-1}) \geq \sum\nolimits_{\kappa \in K \setminus \{\kappa_{m_2-1}\}} (2\lambda - h) \, \bigl(j_\alpha(t - \kappa) -j_\alpha(\kappa_{m_2-1} - \kappa))\bigr)$ and using~\eqref{eq:fold_con2} we can estimate $\zeta_{m_2-1}(t)$ as
	\begin{align*}
		&\zeta_{m_2-1}(t) = \zeta_{m_2-1}(\kappa_{m_2-1}) - \zeta_{m_2-1}(\kappa_{m_2-1}) + \zeta_{m_2-1}(t)\\
        &= \lambda - g(\kappa_{m_2-1}) + g(t) - (2\lambda - h) \sum\nolimits_{n=1}^{m_2-1}  \sgn(\zeta_{n-1}(\kappa_n)) \\
        &\qquad \cdot \bigl(j_\alpha(t - \kappa_n) - j_\alpha(\kappa_{m_2-1} - \kappa_n)\bigr) \\
        &\geq \lambda  + g^\prime(\kappa_{m_2-1}) \,  (t - \kappa_{m_2-1}) - \frac{\|g^{\prime\prime}\|_\infty}{2} \, (t - \kappa_{m_2 -1})^2 \\
        &\qquad - (2\lambda - h) \sum\nolimits_{\kappa \in K}  \bigl(j_\alpha(t - \kappa) - j_\alpha(\kappa_{m_2-1} - \kappa)\bigr) \\
        &\geq \lambda  - \frac{h}{\alpha^2} \, (t - \kappa_{m_2-1})^2 - (2 \lambda -h) \cdot j_{\alpha} (t - \kappa_{m_2-1} ) \\
		&\geq - \lambda + h - \frac{h}{\alpha^2} \, (t - \kappa_{m_2-1})^2 .
	\end{align*}
	With this we can conclude that
	\begin{align*}
		\kappa_{m_2} &= \inf\bigl\{t > \kappa_{m_2 -1} \bigm| \zeta_{m_2-1}(t) = -\lambda\bigr\} \\
		&\geq \inf\bigl\{t > \kappa_{m_2-1} \bigm| -\lambda +h - \tfrac{h}{\alpha^2} \, (t - \kappa_{m_2-1})^2 = -\lambda\bigr\} \\
		&= \kappa_{m_2-1} + \alpha.
	\end{align*}
	We now assume that we have $\kappa_{m_n} \geq \kappa_{m_n-1} + \alpha$ for $n \geq 2$.
	Again without loss of generality let $\Mod_\lambda^{h,j_\alpha} g(\kappa_{m_n}) = \lambda$.
	Since $\kappa_{m_{n+1}-1} \geq \kappa_{m_n} \geq \kappa_{m_n-1} + \alpha$ we can proceed with the same arguments as before to obtain
	\begin{align*}
		& \kappa_{m_{n+1}} = \inf\bigl\{t>\kappa_{m_{n+1}-1} \bigm| \zeta_{m_{n+1}-1}(t) = -\lambda\bigr\} \\
		&\geq \inf\bigl\{t>\kappa_{m_{n+1}-1} \bigm| -\lambda + h - \tfrac{h}{\alpha^2} \, (t - \kappa_{m_{n+1}-1})^2 = -  \lambda\bigr\} \\
		&= \kappa_{m_{n+1}-1} + \alpha,
	\end{align*}
	which completes the proof.
\end{proof}

The separation in Theorem~\ref{theo:Mod_j_range_separation} can be used to get the return property for sufficiently large $t$ also for the operator $\Mod_\lambda^{h,j_\alpha}$.

\begin{proposition} \label{prop:Mod_j_return}
	Let $0 < h < \lambda$ and $g \in \Cont^{1,1}_{\lambda,\tau_0} \cap \B^{\lambda-h}_\rho$ with $\rho > \tau_0$ and $\|g^{\prime\prime}\|_\infty \leq \frac{2 h}{\alpha^2}$.
	Then, $\Mod_\lambda^{h,j_\alpha} g(t) = g(t)$ for all $t \geq \rho + 2 \alpha$.
\end{proposition}

\begin{proof}
	Let $\kappa_{n}$ be the largest folding point smaller than or equal to $\rho$.
	We now go through the different possible cases.
	
	\smallskip
	\underline{Case 1:}
	If $\kappa_n + \alpha \leq \rho$, there exists $k \in \Z$ such that
	\begin{align*}
		|\Mod_\lambda^{h,j_\alpha} g(\rho)| &= |k \, (2 \lambda - h) + g(\rho)| \\
		&\geq |k| \, (2 \lambda - h) - |g(\rho)| > |k| \, (2 \lambda - h) - \lambda + h.
	\end{align*}
	As $|\Mod_\lambda^{h,j_\alpha} g(\rho)| \leq \lambda$, we get $k = 0$, which implies that $\Mod_\lambda^{h,j_\alpha} g(\rho) = g(\rho)$ and $\zeta_n(t) = g(t)$ for all $t \geq \rho$.
	Inserting this into the definition of $\kappa_{n+1}$ gives	
	\begin{equation*}
		\kappa_{n+1} = \inf\{t > \rho \mid |g(t)| \geq \lambda\} = \infty,
	\end{equation*}
	from which follows that $\Mod_\lambda^{h,j_\alpha} g(t) = g(t)$ for all $t \geq \rho$.
	
	\smallskip
	\underline{Case 2:}
	If $\kappa_n + \alpha > \rho$, but $\kappa_n + \alpha \leq  \kappa_{n+1}$, we get, for some $k \in \Z$,
	\begin{align*}
		|\Mod_\lambda^{h,j_\alpha} g(\kappa_n + \alpha)| &= |k \, (2 \lambda - h) + g(\kappa_n + \alpha)| \\
		&> |k| \, (2 \lambda - h) - \lambda + h
	\end{align*}
	so that $\kappa_{n+1} = \infty$ and $\Mod_\lambda^{h,j_\alpha} g(t) = g(t)$ for all $t \geq \kappa_n + \alpha$.
	Since $\kappa_n \leq \rho$, this also holds for all $t \geq \rho + \alpha$.
	
	\smallskip
	\underline{Case 3:}
	If $\kappa_n + \alpha > \max\{\rho, \kappa_{n+1}\}$, but $\kappa_{n+1} + \alpha \leq  \kappa_{n+2}$, we get, for some $k \in \Z$,
	\begin{align*}
		|\Mod_\lambda^{h,j_\alpha} g(\kappa_{n+1} + \alpha)| &= |k \, (2 \lambda - h) + g(\kappa_{n+1} + \alpha)| \\
		&> |k| \, (2 \lambda - h) - \lambda + h
	\end{align*}
	so that $\kappa_{n+2} = \infty$ and $\Mod_\lambda^{h,j_\alpha} g(t) = g(t) $ for all $t \geq \kappa_{n+1} + \alpha$.
	Since $\kappa_{n+1} \leq \rho + \alpha$, this also holds for all $t \geq \rho + 2\alpha$.
	
	\smallskip
	\underline{Case 4:}
	Let $\kappa_n + \alpha > \max\{\rho, \kappa_{n+1}\}$, $\kappa_{n+1} + \alpha >  \kappa_{n+2}$.
	As above, without loss of generality let $\Mod_\lambda^{h,j_\alpha} g( \kappa_{n}) = \lambda$.
	Then, the separation property in Theorem~\ref{theo:Mod_j_range_separation} implies that $\lambda = \Mod_\lambda^{h,j_\alpha} g(\kappa_{n+1}) = \Mod_\lambda^{h,j_\alpha} g(\kappa_{n+2})$ and $\Mod_\lambda^{h,j_\alpha} g(\kappa_{n+2} + \alpha) \leq \zeta_{n+2}(\kappa_{n+2} + \alpha)$.
	This gives us that
	\begin{align*}
		&\Mod_\lambda^{h,j_\alpha} g( \kappa_{n+2} + \alpha) \leq \lambda + \zeta_{n+2}( \kappa_{n+2} + \alpha)  - \zeta_{n+2}( \kappa_{n} ) \\
		&\leq \lambda + g( \kappa_{n+2} + \alpha)  - g( \kappa_{n} ) - 3 ( 2 \lambda-h) \\
		&< \lambda + 2 (\lambda - h)- 3 (2 \lambda - h) = -3\lambda + h < -2\lambda
	\end{align*}
	in contradiction to $ \Mod_\lambda^{h,j_\alpha} g(t) \in [-\lambda,\lambda]$.
\end{proof}

\begin{figure}
	\centering
	\includegraphics[width=\linewidth]{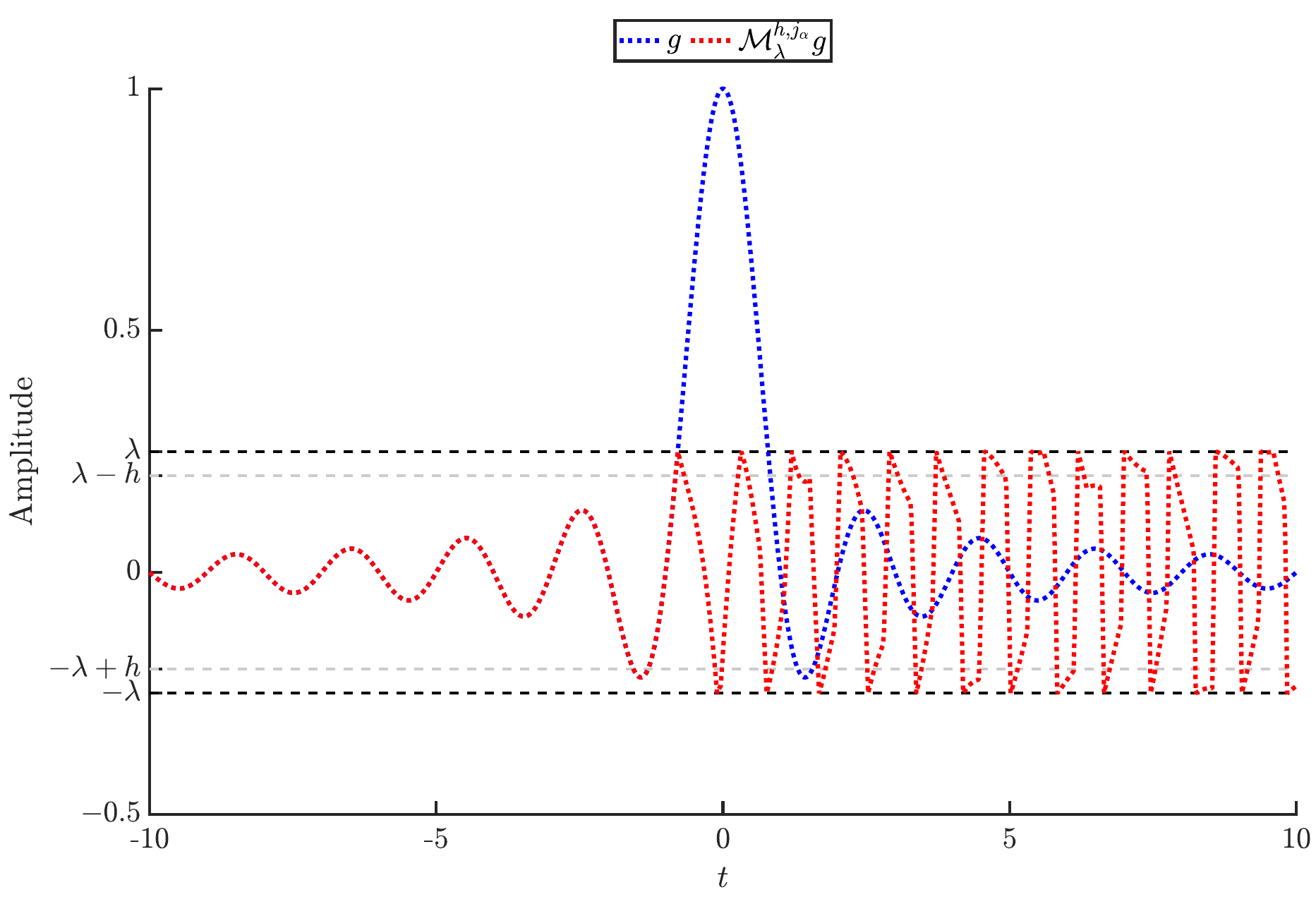}
	\caption{Example of an endless folding in $\Mod_\lambda^{h,j_\alpha} g(t)$ using $g(t) = \sinc(\pi t)$, where $j_\alpha(t) = \frac{t}{\alpha}$ for $t \in [0,\alpha]$, $\lambda = 0.25$, $h = 0.05$ and $\alpha = 0.75$.}
	\label{fig:Mod_j_continuation}
\end{figure}

We now show that the behaviour in Fig.~\ref{fig:Mod_j_continuation} is not due to numerical inaccuracies but can be proven to occur.
As the construction depends on the specific form of $j_\alpha$, we restrict the following discussion to the folding function from~\cite{Beckmann2025a}, i.e.,
\begin{equation*}
	j_\alpha(t) = \begin{cases} \frac{t}{\alpha} & \text{if } t \in [0,\alpha] \\ 0 & \text{otherwise}. \end{cases}
\end{equation*}
Endless folding may happen if folds in one direction are triggered while folds in the opposite direction are still incomplete.
More precisely, we make the following assumption.

\begin{assumption} \label{as:Mod_j_continuation}
	There are $c_0 \in \N$, $c_0 < c_1 \in \N$, $c_2 \in \{0,1\}$, $c_3 = c_1 + c_2$, $c_4 \in \{0,1\}$ as well as $t_1 \in \R$, $t_2 \in [t_1,t_1+\alpha)$, $t_3 \in (\max \{t_2 ,t_1 + \tfrac{\alpha}{2}\}, t_1+\alpha)$ and $t_4 \in[t_3,t_3+\alpha)$ so that
	\begin{align*}
		\zeta_n(t) &= g(t) + \sigma \, r(t) - (2\lambda - h) \, \sigma \, \bigl(c_0 - c_1 \, j_\alpha(t - t_1) \\
		&\qquad - c_2 \, j_\alpha(t - t_2) + c_3 \, j_\alpha(t - t_3) + c_4 \, j_\alpha(t - t_4)\bigr)
	\end{align*}
	for all $t \geq t_3$ and 
	\begin{equation*}
		\sigma \cdot \zeta_{n}(t) < \lambda
		\quad \forall \, t > t_4,
	\end{equation*}
	where $\sigma \in \{\pm 1\}$ and $r \colon \R \to \R$ is a function with $r(t) \geq0$ for all $t \in \R $, $r(t) = 0$ for all $t \geq t_1 + \alpha $ and $\| \partial_+ r \|_{\infty} \leq \frac{2\lambda - h}{\alpha}$.
\end{assumption}

Using Assumption~\ref{as:Mod_j_continuation} we can guarantee the continuation of foldings as follows.

\begin{proposition} \label{prop:Mod_j_continuation}
	Let $h < \frac{\lambda}{2}$ and $j_\alpha(t) = \frac{t}{\alpha}$ for $t \in [0,\alpha]$.
	Moreover, let $g \in \Cont^{1,1}_{\lambda,\tau_0}$ and $n \geq 5$ such that
	\begin{enumerate}
		\item Assumption~\ref{as:Mod_j_continuation} is satisfied with $\sigma \in \{\pm 1\}$ and $t_i = \kappa_{m_i}$ for $m_1 \leq m_2 \leq m_3 \leq m_4 = n$,
		\item for all $t > \kappa_n - 2\alpha$ there exists at most one $s \in [t,t+2 \alpha]$ with $g^\prime(s) = 0$,
		\item $|g(t)| < h$ for all $t \geq t_3$,
		\item $|g(t) - g(s)| < h$  for all $t,s \geq t_3$ ,
		\item $|g^\prime(t)| < \frac{2\lambda - h}{\alpha}$ for all $t \geq t_1$.
	\end{enumerate}
	Then, for all $t \in \R$ there exists $q \in \N$ so that $t < \kappa_{q} < \infty$.
\end{proposition}
    
\begin{proof}
	Without loss of generality, assume that $\sigma = -1$.
	Then, for $t \in (t_3, t_1 + \alpha]$ we get
	\begin{align*}
		\zeta_n(t) &= \zeta_n(t) - \zeta_n(t_3) - \lambda \\
		&= -\lambda + g(t) - g(t_3) + r(t_3) - r(t) \\
		&\qquad -(2\lambda - h) \bigl(c_1 \, j_\alpha(t - t_1) + c_2 \, j_\alpha(t - t_2) \\
		&\qquad - c_3 \, j_\alpha(t - t_3) - c_4 \, j_\alpha(t - t_4) - c_1 \, j_\alpha(t_3 - t_1) \\
		&\qquad - c_2 \, j_\alpha(t_3 - t_2)\bigr) \\
		&< -\lambda + h  -r(t) + r(t_3) + (2\lambda - h) c_4 \, j_\alpha(t - t_4) \\
		&\leq -\lambda + h + 2 \, \frac{2\lambda - h}{\alpha} \Bigl(t_1 + \alpha - t_1 - \frac{\alpha}{2}\Bigr) = \lambda.
	\end{align*}
    Therefore, we have $\kappa_{n+1} > t_1 + \alpha$.
    In the same way we obtain $\zeta_n(t) < \lambda$ for $t \in (t_3, t_3 + \frac{\alpha}{2}]$ and, hence, $\kappa_{n+1} > t_3 + \frac{\alpha}{2}$.
    On the other hand, we have 
	\begin{align*}
		\zeta_n(t_3 + \alpha) &= g(t_3 + \alpha) + (2\lambda - h) \, (c_0 + c_4 \, j_\alpha(t_3 + \alpha - t_4))\\
		&\geq -h + c_0 \, (2\lambda - h) > \lambda.
	\end{align*}
	Thus, we conclude that $\kappa_{n+1} \in (\max\{t_1 + \alpha, t_3 + \frac{\alpha}{2}\} , t_3 + \alpha)$ and $\sgn(\zeta_n(\kappa_{n+1})) = 1 = -\sigma = -\sgn(\zeta_n(\kappa_n))$.
	
	Let us first assume that $\kappa_{n+1} \in [t_2 + \alpha, t_3 + \alpha)$.
	Then, we have $\kappa_{n+1} = \ldots = \kappa_{n+m}$ with
	\begin{align*}
		 m &= \biggl\lfloor \frac{\partial_+ \zeta_n(\kappa_{n+1})}{\frac{2\lambda - h}{\alpha}} + 1\biggr\rfloor = \biggl\lfloor \frac{g^\prime(\kappa_{n+1})}{\frac{2\lambda - h}{\alpha}} + c_3 + c_4 + 1\biggr\rfloor \\
		&= \begin{cases}
			c_3 + c_4 +1 & \text{if } g^\prime(\kappa_{n+1}) \geq 0 \\
			c_3 + c_4 & \text{otherwise}.
		\end{cases}
	\end{align*}
	For $m = c_3+c_4+1$ setting $\tilde{t}_3 = \tilde{t}_4 = \kappa_{n+m}$, $\tilde{t}_1 = t_3$, $\tilde{t}_2 = t_4$, $\tilde{c_0} = c_1 + c_2 - c_0$, $\tilde{c}_1 = c_3$, $\tilde{c}_2 = c_4$, $\tilde{c}_3 = c_3 + c_4$, $\tilde{c}_4 = 1$ as well as $\tilde{r}(t) = 0$ gives
	\begin{equation}\label{eq:zeta_n+m}
		\begin{aligned}
			\zeta_{n+m}(t) &= g(t) + \tilde{r}(t) - (2\lambda - h) \bigl(\tilde{c_0} -   \tilde{c}_1 \, j_\alpha(t - \tilde{t}_1) \\
				&\qquad - \tilde{c}_2 \, j_\alpha(t - \tilde{t}_2) + \tilde{c}_3 \, j_\alpha(t - \tilde{t}_3) \\
				&\qquad + \tilde{c}_4 \, j_\alpha(t -\tilde{t}_4)\bigr)
		\end{aligned}
	\end{equation}
	for all $t \geq \tilde{t}_3$ and $\zeta_{n+m}(t) < \lambda$ for all $t > \tilde{t}_4$ as $|g^\prime(t)|< \frac{2\lambda - h}{\alpha}$ for all $t \geq t_1$.
	Since $\sgn(\zeta_n(\kappa_n)) = -\sgn(\zeta_n(\kappa_{n+1}))$, we have $\tilde{t}_3 = \kappa_{n+m} = \kappa_{n+1} > \kappa_n = t_4 = \tilde{t}_2$.
	Moreover, $\tilde{t}_3 = \kappa_{n+m} = \kappa_{n+1} \geq t_3 + \frac{\alpha}{2} = \tilde{t}_1 + \frac{\alpha}{2}$.
	For $m = c_3 + c_4$ we distinguish two cases.
	If there exists no $t \in (\kappa_{n+m}, t_3 + \alpha)$ with $\zeta_{n+m}(t) = \lambda$, we get the same equation for $\zeta_{n+m}$ as in \eqref{eq:zeta_n+m} with $\tilde{c}_4 = 0$.
	On the other hand, if $t \in (\kappa_{n+m}, t_3 + \alpha)$ exists with $\zeta_{n+m}(t) = \lambda$, then $\kappa_{n+m+1}$ must be the first such point and setting $\tilde{t}_3 = \kappa_{n+m}$, $\tilde{t}_4 = \kappa_{n+m+1}$, $\tilde{t}_1 = t_3$, $\tilde{t}_2 = t_4$, $\tilde{c_0} = c_1 + c_2 - c_0$, $\tilde{c}_1 = c_3$, $\tilde{c}_2 = c_4$, $\tilde{c}_3 = c_3 + c_4$, $\tilde{c}_4 = 1$ as well as $\tilde{r}(t) = 0$ gives
	\begin{equation}\label{eq:zeta_n+m+1}
		\begin{aligned}
			\zeta_{n+m+1}(t) &= g(t) + \tilde{r}(t) - (2\lambda - h) \bigl(\tilde{c_0} -   \tilde{c}_1 \, j_\alpha(t - \tilde{t}_1) \\
			&\qquad - \tilde{c}_2 \, j_\alpha(t - \tilde{t}_2) + \tilde{c}_3 \, j_\alpha(t - \tilde{t}_3) \\
			&\qquad + \tilde{c}_4 \, j_\alpha(t - \tilde{t}_4)\bigr)
		\end{aligned}
	\end{equation}
	for all $t \geq \tilde{t}_3$ and $\zeta_{n+m+1}(t) < \lambda$ for all $t > \tilde{t}_4$.
	As before, we obtain $\tilde{t}_3 > \tilde{t}_2$ and $\tilde{t}_3 > \tilde{t}_1 + \frac{\alpha}{2}$.
	
	Let us now assume that $\kappa_{n+1} \in (t_3 + \frac{\alpha}{2}, t_2 + \alpha)$, which can only happen if $c_2 = 1$ and $t_1 < t_2$.
	As $|g^\prime(t)| < \frac{2\lambda - h}{\alpha}$ for all $t \geq t_1$ and $j_\alpha(s) = \frac{s}{\alpha}$ for $s \in [0,\alpha]$, we get $g^\prime(t_1) < 0$ and $g^\prime(t_2) \geq 0$ so that there exists $\tilde{t} \in (t_1, t_2]$ with $g^\prime(\tilde{t}) = 0$.
	Using 2) yields $g^\prime(t) > 0$ for $t \in (\tilde{t},\tilde{t}+2\alpha)$ and, hence, also for $t \in (t_2, t_2 + \alpha)$ so that, in particular, $g^\prime(\kappa_{n+1}) \geq 0$ and $\zeta_n$ is monotonously increasing on $(t_1 + \alpha, t_2 + \alpha)$.
	Therefore, we have $\kappa_{n+1} = \ldots = \kappa_{n+m}$ with
	\begin{align*}
		m &= \biggl\lfloor \frac{\partial_+ \zeta_n(\kappa_{n+1})}{\frac{2\lambda - h}{\alpha}} + 1\biggr\rfloor = \biggl\lfloor \frac{g^\prime(\kappa_{n+1})}{\frac{2\lambda - h}{\alpha}} -c_2 + c_3 + c_4 + 1\biggr\rfloor \\
		&= c_3 + c_4.
	\end{align*} 
	If there exists no $t \in (\kappa_{n+m}, t_3 + \alpha)$ with $\zeta_{n+m}(t) = \lambda$, we get the same equation for $\zeta_{n+m}$ as in \eqref{eq:zeta_n+m} with $\tilde{c}_4 = 0$ and $\tilde{r}(t) = c_2 (2\lambda - h) (1 - j_\alpha(t - t_2))$.
	If there does exist $t \in (\kappa_{n+m}, t_3 + \alpha)$ with $\zeta_{n+m}(t) = \lambda$, then $\kappa_{n+m+1}$ must be the first such point and we get the same equation for $\zeta_{n+m+1}$ as in \eqref{eq:zeta_n+m+1} with $\tilde{r}(t) = c_2 (2\lambda - h) (1 - j_\alpha(t - t_2))$. 
	
	All in all, Assumption~\ref{as:Mod_j_continuation} is satisfied for $\zeta_{n+m+\ell}$ with opposite sign $-\sigma$ and folding points $\kappa_{m_3}$, $\kappa_n$, $\kappa_{n+m}$, $\kappa_{n+m+\ell}$, where either $\ell = 0$ or $\ell = 1$.
    In particular, we have $\kappa_{n+m} \in [\kappa_n + \frac{\alpha}{2}, \kappa_n + \alpha)$.
    Repeating this argument shows that for all $t \in \R$ there exists some $q \in \N$ with $t < \kappa_q < \infty$.
\end{proof}

	We wish to stress that endless folding does not occur only in a niche folding regime.
	In fact, for any choice of parameters $\lambda>0$, $h<\frac{\lambda}{2}$ and $\alpha>0$ one can find functions satisfying the assumptions of Proposition~\ref{prop:Mod_j_continuation}.
    For example, for $\lambda = 1$, $h=0.35$ and $\alpha = \frac{\pi}{2}$ the assumptions are satisfied for the benign function $g(t) = \sinc(t)$.

\subsection{Transient dependent folding points and reset time}
	
The definition of the generalized modulo operator requires the rather restrictive condition~\eqref{eq:fold_con2} on the folding function $j_\alpha$ with $\alpha > 0$, which, in particular, implies that $j_\alpha$ cannot be differentiable at $0$, whereas it may be arbitrarily smooth at~$\alpha$.
We now modify the modulo operator $\Mod_\lambda^{h,j_\alpha}$ to allow for more general folding functions not necessarily satisfying~\eqref{eq:fold_con2}.
To this end, we introduce a minimal reset time $\sigma > 0$ between two consecutive folding points.
	
\begin{definition} \label{def:Mod_reset}
	Let $\lambda > 0$, $h<2\lambda$, $\sigma > 0$ and let $j_\alpha$ be a folding function with transient $\alpha \geq 0$.
	For a function $g \in  \Cont_{\lambda,\tau_0}^{0,1}$ with $\tau_0 \in \R$ we define the two sequences $(\nu_n)_{n \in \N_0}$ and $(\xi_n)_{n \in \N_0}$ via
	\begin{equation} \label{eq:Mod_reset_kappa}
		\begin{gathered}
			\nu_{n+1} = \inf \bigl\{t > \nu_n + \sigma \bigm| |\xi_n(t)| \geq \lambda\bigr\},\\
			\xi_{n+1} = \xi_n - (2\lambda - h) \sgn(\xi_n(\nu_{n+1})) \, j_\alpha(\cdot - \nu_{n+1})
		\end{gathered}
	\end{equation}
	and $\xi_{n+1} = \xi_n$ if $\nu_{n+1} = \infty$, where $\nu_0 = \tau_0-\sigma$ and $\xi_0 = g$.
	Then, the output of the {\em delayed generalized modulo operator} $\Mod_{\lambda,\sigma}^{h,j_\alpha}$ is defined as
	\begin{equation} \label{eq:Mod_reset}
		\Mod_{\lambda,\sigma}^{h,j_\alpha} g(t) = \lim\nolimits_{n \to \infty} \xi_n(t).
	\end{equation}
\end{definition}

\begin{figure}
	\centering
	\includegraphics[width=\linewidth]{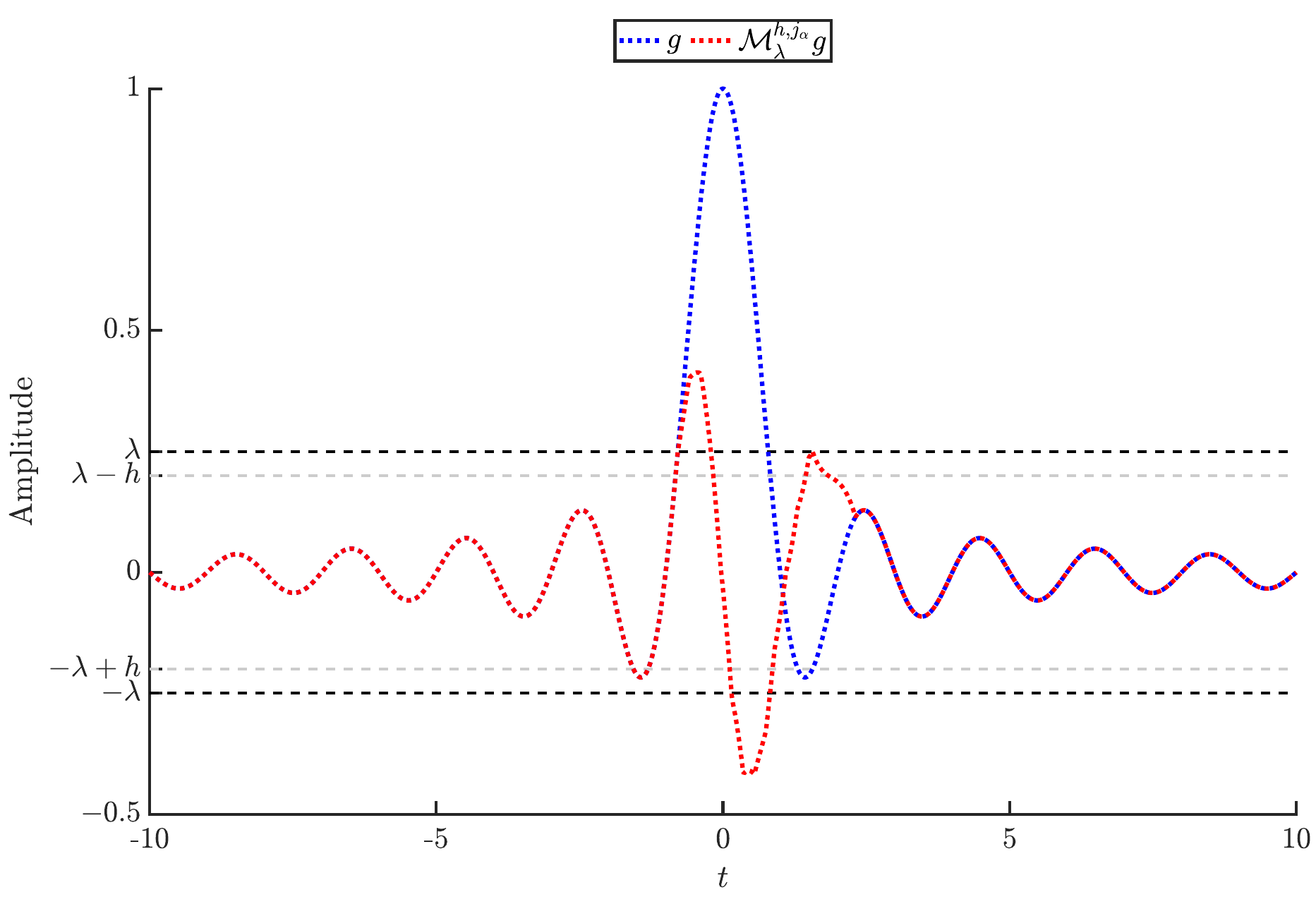}
	\caption{Illustration of $\Mod_{\lambda,\sigma}^{h,j_\alpha} g(t)$ using $g(t) = \sinc(\pi t)$, where we choose $j_\alpha(t) = \frac{t}{\alpha}$ for $t \in [0,\alpha]$, $\lambda = 0.25$, $h = 0.05$, $\alpha = 0.75$ and $\sigma = 0.2$.}
	\label{fig:Mod_reset}
\end{figure}

As we now enforce a minimum spacing of $\sigma > 0$ between folding points, for each $t \in \R$ we can find $n \in \N$ so that $t \leq \tau_0 + n \cdot \sigma$ and $t \leq \nu_{n+1}$.
This implies $\Mod_{\lambda,\sigma}^{h,j_\alpha} g(t) = \xi_{n+1}(t)$ and, hence, the well-definedness of $\Mod_{\lambda,\sigma}^{h,j_\alpha}$ without additional assumptions. 
We now collect further properties of $\Mod_{\lambda,\sigma}^{h,j_\alpha}$.

\begin{proposition}
	Let $g \in  \Cont_{\lambda,\tau_0}^{0,1}$.
	Then, $|\Mod_{\lambda,\sigma}^{h,j_\alpha} g(\nu_n)| \geq \lambda$ for all $\nu_n < \infty$ and $|\Mod_{\lambda,\sigma}^{h,j_\alpha}g(t)| < \lambda$ for all $t \in [\nu_n + \sigma, \nu_{n+1})$.
\end{proposition}

\begin{proof}
	Assume that $\nu_n < \infty$ and $|\Mod_{\lambda,\sigma}^{h,j_\alpha} g(\nu_n)| < \lambda$.
	By definition this also implies $|\xi_{n-1}(\nu_n)| < \lambda$ and, therefore, $\nu_n \notin \overline{\{t > \nu_{n-1} + \sigma \mid |\xi_{n-1}(t)| \geq \lambda\}}$ in contradiction to the definition of $\nu_n$.
	Now, assume that there exists $s \in [\nu_n + \sigma , \nu_{n+1})$ with $|\Mod_{\lambda,\sigma}^{h,j_\alpha} g(s)| \geq \lambda$.
	This also implies $|\xi_n(s)| \geq \lambda$ and, hence, $s \in  \overline{\{t > \nu_n + \sigma \mid |\xi_n(t)| \geq \lambda\}}$.
	But this gives $s \geq \nu_{n+1}$ in contradiction to the assumption.
\end{proof}

While well-definedness of $\Mod_{\lambda,\sigma}^{h,j_\alpha}$ does not require the conditions \eqref{eq:fold_con1} and \eqref{eq:fold_con2}, to guarantee a bounded range and the return property, we need additional assumptions that are weaker than \eqref{eq:fold_con1} and \eqref{eq:fold_con2}, but result in a tension between modelling flexibility and theoretical guarantees.

\begin{theorem}
	Let $0 < \sigma \leq \alpha$ and $j_\alpha$ satisfy
	\begin{equation}\label{eq:fold_con3}
		\forall \, t \in [0,\alpha] \colon ~ j_\alpha(t) \leq 1
		\tag{$\text{C}_1^\ast$}
	\end{equation}
	and
	\begin{equation}\label{eq:fold_con4}
		\forall \, s \leq t \in [\sigma,\alpha] \colon ~ 0 \leq \partial_+ j_\alpha(t) \leq \partial_+ j_\alpha(s)
		\tag{$\text{C}_2^\ast$}
	\end{equation}
	Moreover, let $g \in \Cont_{\lambda,\tau_0}^{1,1}$ satisfy $\|g^\prime\|_\infty \leq (2\lambda - h) \cdot \frac{j_\alpha(\sigma)}{\sigma}$ and $\|g^{\prime\prime}\|_\infty \leq \frac{2 h}{\alpha^2}$.
	Then, we have $\Mod_{\lambda,\sigma}^{h,j_\alpha} g(t) \in [-\lambda-r,\lambda+r]$ with 
	\begin{equation*}
		r = \sup_{t \in (0,\sigma)} \bigl(t \cdot \|g^\prime\|_\infty - (2\lambda - h) \cdot j_\alpha(t)\bigr) \geq 0,
	\end{equation*}
	and
	\begin{equation*}
		\Mod_{\lambda,\sigma}^{h,j_\alpha} g(\nu_n) \cdot \Mod_{\lambda,\sigma}^{h,j_\alpha} g(\nu_{n+1}) < 0
		\implies
		\nu_{n+1} - \nu_n \geq \alpha
	\end{equation*}
	as well as $\Mod_{\lambda,\sigma}^{h,j_\alpha} g(\nu_n) = \pm \lambda$ for all $n \in \N$ with $\nu_n < \infty$.
	If, in addition, $g \in \B^{\lambda-h}_\rho$ with $\rho > \tau_0$, we further obtain $\Mod_{\lambda,\sigma}^{h,j_\alpha} g(t) = g(t)$ for all $t \geq \rho + 2\alpha$.
\end{theorem}

\begin{proof}
	By definition, we have $\Mod_{\lambda,\sigma}^{h,j_\alpha}g(t) = g(t) \in (-\lambda,\lambda)$ for all $t < \nu_1$ and $g(\nu_1) = \pm \lambda$ if $\nu_1 < \infty$ as $\nu_1 = \kappa_1$.
	Without loss of generality let $g(\nu_1) = \lambda$ so that $g^\prime(\nu_1) \geq 0$.
	For $t \in (\nu_1, \nu_1 + \sigma)$ we obtain
	\begin{align*}
		\Mod_{\lambda,\sigma}^{h,j_\alpha}g(t) &= g(t) - (2\lambda - h) \cdot j_\alpha(t - \nu_1) \\
		&= \lambda - g(\nu_1) + g(t) - (2\lambda - h) \cdot j_\alpha(t - \nu_1) \\
		&\leq \lambda + (t - \nu_1) \cdot \|g^\prime\|_\infty - (2\lambda - h) \cdot j_\alpha(t - \nu_1) \\
		&\leq \lambda + r
	\end{align*}
	and $g^\prime(\nu_1) \geq 0$ in combination with~\eqref{eq:fold_con3} gives
	\begin{align*}
		\Mod_{\lambda,\sigma}^{h,j_\alpha}g(t) &= \lambda - g(\nu_1) + g(t) - (2\lambda - h) \cdot j_\alpha(t - \nu_1) \\
		&\geq \lambda - \frac{\|g^{\prime\prime}\|_\infty}{2} \, (t - \nu_1)^2 - (2\lambda - h) \cdot j_\alpha(t - \nu_1) \\
		&\geq - \lambda + h - \frac{h}{\alpha^2} \, (t - \nu_1)^2 > -\lambda.
	\end{align*}
	Moreover,
	\begin{equation*}
		\Mod_{\lambda,\sigma}^{h,j_\alpha} g(\nu_1 + \sigma) \leq \lambda + \sigma \cdot \|g^\prime\|_\infty - (2\lambda - h) \cdot j_\alpha(\sigma) \leq \lambda
	\end{equation*}
	and
	\begin{equation*}
		\Mod_{\lambda,\sigma}^{h,j_\alpha} g(\nu_1 + \sigma) \geq \lambda - \frac{\sigma^2}{2} \, \|g^{\prime\prime}\|_\infty - (2\lambda - h) \cdot j_\alpha(\sigma) \geq -\lambda.
	\end{equation*}
	We already know that $|\Mod_{\lambda,\sigma}^{h,j_\alpha}g(t)| < \lambda$ for all $t \in [\nu_1 + \sigma, \nu_2)$ and, hence, $\Mod_{\lambda,\sigma}^{h,j_\alpha} g(t) \in [-\lambda,\lambda+r]$ for all $t < \nu_2$.
	Moreover, $\Mod_{\lambda,\sigma}^{h,j_\alpha} g(\nu_2) = \pm\lambda$ if $\nu_2 < \infty$ and $\Mod_{\lambda,\sigma}^{h,j_\alpha} g(\nu_2) = -\lambda$ can only happen if $\nu_2 \geq \nu_1 + \alpha$.
	
	Now let $2 \leq \ell \leq n \in \N$ such that $\nu_{\ell-1} + \alpha \leq \nu_\ell \leq \nu_n < \infty$ and $\Mod_{\lambda,\sigma}^{h,j_\alpha} g(\nu_\ell) = \ldots = \Mod_{\lambda,\sigma}^{h,j_\alpha} g(\nu_n) = \pm\lambda$.
	Without loss of generality assume that $\Mod_{\lambda,\sigma}^{h,j_\alpha} g(\nu_n) = \xi_{n-1}(\nu_n) = \lambda$.
	Then, for $t \in (\nu_n, \nu_n + \sigma)$ we get, as before and using the monotonicity of $j_\alpha$ on $[\sigma,\alpha]$ due to Assumption~\eqref{eq:fold_con4},
	\begin{align*}
		&\Mod_{\lambda,\sigma}^{h,j_\alpha} g(t) = \xi_{n-1}(t) - (2\lambda - h) \cdot j_\alpha(t - \nu_n) \\
		&= \lambda - \xi_{n-1}(\nu_n) + \xi_{n-1}(t) - (2\lambda - h) \cdot j_\alpha(t - \nu_n) \\
		&\leq \lambda - g(\nu_n) + g(t) - (2\lambda - h) \cdot j_\alpha(t - \nu_n) \leq \lambda + r
	\end{align*}
	and $\Mod_{\lambda,\sigma}^{h,j_\alpha} g(\nu_n + \sigma) \leq \lambda$.
	Additionally, the monotonicity of $\partial_+ j_\alpha$ on $[\sigma,\alpha]$ due to Assumption~\eqref{eq:fold_con4} gives $g^\prime(\nu_n) \geq \partial_+ \Bigl(\sum\nolimits_{\nu \in K \setminus \{\nu_n\}} (2\lambda - h) \cdot j_\alpha(t - \nu)\Bigr)$ for $t > \nu_n$ and with $K = \{\nu_m \in (\nu_n - \alpha, \nu_n]\}$.
	With this we get for $t > \nu_n$
	\begin{align*}
		\xi_n(t) &\geq \xi_n(\nu_n) + g^\prime(\nu_n) \, (t - \nu_n) - \frac{\|g^{\prime\prime}\|_\infty}{2} \, (t - \nu_n)^2 \\
		&\qquad - (2\lambda - h) \sum\nolimits_{\nu \in K} \bigl(j_\alpha(t - \nu) - j_\alpha(\nu_n - \nu)\bigr) \\
		&\geq \lambda  - \frac{h}{\alpha^2} \, (t - \nu_n)^2 - (2\lambda - h) \cdot j_\alpha(t - \nu_n) \\
		&\geq - \lambda + h - \frac{h}{\alpha^2} \, (t - \nu_n)^2.
	\end{align*}
	Therefore, for $t \in (\nu_n,\nu_n+\sigma]$ we have $\xi_n(t) \geq -\lambda$ and $\xi_n(t) = -\lambda$ can only occur for $t \geq \nu_n + \alpha$.
	We already know that $|\Mod_{\lambda,\sigma}^{h,j_\alpha}g(t)| < \lambda$ for all $t \in [\nu_n + \sigma, \nu_{n+1})$ and, hence, $\Mod_{\lambda,\sigma}^{h,j_\alpha} g(t) \in [-\lambda,\lambda+r]$ for all $t < \nu_{n+1}$.
	Moreover, $\Mod_{\lambda,\sigma}^{h,j_\alpha} g(\nu_{n+1}) = \pm\lambda$ if $\nu_{n+1} < \infty$ and $\Mod_{\lambda,\sigma}^{h,j_\alpha} g(\nu_{n+1}) = -\lambda$ can only happen if $\nu_{n+1} \geq \nu_n + \alpha$.
	
	In conclusion, we have $\Mod_{\lambda,\sigma}^{h,j_\alpha} g(t) \in [-\lambda-r,\lambda+r]$ for all $t \in \R$ and 
	\begin{equation*}
		\Mod_{\lambda,\sigma}^{h,j_\alpha} g(\nu_n) \cdot \Mod_{\lambda,\sigma}^{h,j_\alpha} g(\nu_{n+1}) < 0
		\implies
		\nu_{n+1} - \nu_n \geq \alpha
	\end{equation*}
	as well as $\Mod_{\lambda,\sigma}^{h,j_\alpha} g(\nu_n) = \pm\lambda$ if $\nu_n < \infty$.
    With this, for $g \in \B^{\lambda-h}_\rho$ with $\rho > \tau_0$ finally follows that $\Mod_{\lambda,\sigma}^{h,j_\alpha} g(t) = g(t)$ for all $t \geq \rho + 2\alpha$ as in the proof of Proposition~\ref{prop:Mod_j_return}.
\end{proof}

\section{Comparison of the Modulo Operators} \label{sec:comparison}

\begin{table*}[t]
    \caption{Conditions for the modulo operators to guarantee certain properties for $g \in \Cont_{\lambda,\tau_0}^{0,1}$ ("--", if no further conditions are needed).}
	\label{tab:Mod_comparison}
    \centering
	\begin{tabular}{llccc}
		\toprule
		& & $\Mod_H g$ & $\Mod_\lambda^{h,j_\alpha} g$ & $\Mod_{\lambda,\sigma}^{h,j_\alpha} g$ \\
		\midrule
		well-definedness & condition & -- & \eqref{eq:fold_con2} & -- \\
		\midrule
		\multirow{3}*{folding at $\pm \lambda$} & \multirow{3}*{condition} & \eqref{eq:fold_con1} & \eqref{eq:fold_con2} & \eqref{eq:fold_con3}; \eqref{eq:fold_con4} \\[0.25ex]
		\multirow{3}*{\& bounded range} & & $g \in \Cont^{1,1}_{\lambda,\tau_0}$; $\|g^{\prime\prime}\|_\infty \leq \frac{2h}{\alpha^2}$; & & $g \in \Cont_{\lambda,\tau_0}^{1,1}$; $\|g^{\prime\prime}\|_\infty \leq \frac{2h}{\alpha^2}$; \\
		& & $\|g^\prime\|_\infty \leq \frac{2\lambda - h}{\alpha}$ & & $\|g^\prime\|_\infty \leq \frac{(2\lambda - h) \cdot j_\alpha(\sigma)}{\sigma}$ \\[1em]
		& range & $[-\lambda,\lambda]$ & $[-\lambda,\lambda]$ & $[-\lambda-r,\lambda+r]$ \\
		\midrule
		\multirow{2}*{separation of folds} & \multirow{2}*{condition} & \multirow{2}*{--} & \eqref{eq:fold_con2} & \multirow{2}*{--} \\[0.25ex]
		& & & $g \in \Cont^{1,1}_{\lambda,\tau_0}$; $\|g^{\prime\prime}\|_\infty \leq \frac{2h}{\alpha^2}$ \\[1em]
		& \multirow{2}*{amount} & \multirow{2}*{$\frac{\min\{h, 2\lambda-h\}}{\|g^\prime\|_\infty}$} & $\alpha$ & \multirow{2}*{$\sigma$} \\
		& & & for folds in different directions & \\
		\midrule
		\multirow{3}*{return property} & \multirow{3}*{condition} & & \eqref{eq:fold_con2} & \eqref{eq:fold_con3}; \eqref{eq:fold_con4} \\[0.25ex]
		& & \multirow{2}*{$g \in \B^{\lambda-h}_{\rho}$} & \multirow{2}*{$g \in \Cont^{1,1}_{\lambda,\tau_0} \cap \B^{\lambda-h}_{\rho}$; $\|g^{\prime\prime}\|_\infty \leq \frac{2 h}{\alpha^2}$} & $g \in  \Cont_{\lambda,\tau_0}^{1,1} \cap \B^{\lambda-h}_\rho$; $\|g^{\prime\prime}\|_\infty \leq \frac{2 h}{\alpha^2}$; \\
		& & & & $\|g^\prime\|_{\infty} \leq \frac{  (2 \lambda - h) \cdot j_{\alpha} (\sigma)}{\sigma}$ \\[1em]
		& interval & $[\rho+\alpha,\infty)$ & $[\rho+2\alpha,\infty)$ & $[\rho+2\alpha,\infty)$  \\
		\bottomrule
	\end{tabular}
\end{table*} 

As mentioned above, the generalized modulo encoder $\Mod_H$ and the generalized modulo operator $\Mod_\lambda^{h,j_\alpha}$ agree for $\alpha = 0$.
We now state a condition under which this also holds for $\alpha > 0$ and for the delayed generalized modulo operator $\Mod_{\lambda,\sigma}^{h,j_\alpha}$.

\begin{theorem} \label{theo:Mod_comparison}
	Let $\alpha \geq \sigma > 0$, $0 < h < 2\lambda$ and $j_\alpha$ satisfy \eqref{eq:fold_con1} and \eqref{eq:fold_con2}.
	Moreover, let $g \in \Cont^{1,1}_{\lambda,\tau_0}$ with $\|g^\prime\|_\infty < \frac{2 \lambda - h}{\alpha}$ and $\|g^{\prime\prime}\|_\infty < \frac{2h}{\alpha^2}$.
	Then, we have $\Mod_H g = \Mod_\lambda^{h,j_\alpha} g = \Mod_{\lambda,\sigma}^{h,j_\alpha} g$.
\end{theorem}

\begin{proof}	
	We first show that $\Mod_H g = \Mod_\lambda^{h,j_\alpha} g$.
	By definition we have $\tau_1 = \kappa_1$ and, therefore, $\Mod_H g(t) = \Mod_\lambda^{h,j_\alpha} g(t)$ for all $t \leq \tau_1 = \kappa_1$.
	Now assume that for some $p \in \N$ we have $\tau_n = \kappa_n < \infty$ for all $n \leq p$ and $\Mod_H g(t) = \Mod_\lambda^{h,j_\alpha} g(t)$ for all $t \leq \tau_p = \kappa_p$.
	From Theorem~\ref{theo:Mod_H_range} follows that $\tau_p \geq \tau_{p-1} + \alpha$ and furthermore $\tau_p + \alpha \leq \tau_{p+1}$.
	The separation in Theorem~\ref{theo:Mod_j_range_separation} gives $\kappa_p + \alpha \leq \kappa_{p+1}$ if $\Mod_\lambda^{h,j_\alpha} g(\kappa_p) \cdot \Mod^{h,j_\alpha}_\lambda g(\kappa_{p+1}) < 0$.
	Therefore, we now assume that $\Mod_\lambda^{h,j_\alpha} g(\kappa_p) \cdot \Mod^{h,j_\alpha}_\lambda g(\kappa_{p+1}) > 0$, in which case the range property in Theorem~\ref{theo:Mod_j_range_separation} gives $\Mod_\lambda^{h,j_\alpha} g(\kappa_p) = \Mod^{h,j_\alpha}_\lambda g(\kappa_{p+1})$ and, thus, $ \zeta_p(\kappa_p) =   \zeta_p(\kappa_{p+1})$.
	Consequently, we obtain
	\begin{align*}
		\kappa_{p+1} &= \inf \bigl\{t>\kappa_p \bigm| |\zeta_p(t)| \geq \lambda\bigr\} \\
		&= \inf \bigl\{t>\kappa_p \bigm| |\zeta_p(t) - \zeta_p(\kappa_p)| \geq 0\bigr\},
	\end{align*}
	where
	\begin{align*}
		|\zeta_p(t) - \zeta_p(\kappa_p)| &\leq |g(t) - g(\kappa_p)| - (2\lambda - h) \cdot j_\alpha(t- \kappa_p) \\
		&\leq \|g^\prime\|_\infty (t - \kappa_p) - (2\lambda - h) \\
		&\leq \tfrac{2\lambda - h}{\alpha} \cdot (t - \kappa_p)  - (2\lambda - h)
	\end{align*}
	and, thus,
	\begin{equation*}
		\kappa_{p+1} \geq \inf \bigl\{t>\kappa_p \bigm|  \tfrac{2\lambda - h}{\alpha} \cdot (t - \kappa_p)  - (2\lambda - h) \geq 0\bigr\} = \kappa_{p} + \alpha.
	\end{equation*}
	Consequently, $\kappa_p + \alpha \leq \kappa_{p+1}$ holds in all possible cases and $\zeta_p(\kappa_p) = \Mod_\lambda^{h,j_\alpha} g(\kappa_p) = \Mod_H g(\kappa_p) = s_p \, \lambda$ gives
	\begin{align*}
		\tau_{p+1} &= \inf \bigl\{t > \tau_{p} \bigm| \Mod_\lambda(g(t) - g(\tau_p) + h s_p) = 0\bigr\} \\
		&= \inf \bigl\{t > \kappa_{p} \bigm| \Mod_\lambda(g(t) - g(\kappa_p) - (2\lambda - h) s_p) = 0\bigr\} \\
		&= \inf \bigl\{t > \kappa_{p} \bigm| \Mod_\lambda(g(t) - g(\kappa_p) \\
		&\qquad - (2\lambda - h)  \sgn(\zeta_p(\kappa_p))) = 0\bigr\} \\
		&= \inf \bigl\{t > \kappa_p \bigm| \Mod_\lambda (\zeta_p(t) - \zeta_p(\kappa_p)) = 0\bigr\} \\
		&= \inf \bigl\{t > \kappa_p \bigm| \Mod_\lambda (\zeta_p(t) - \lambda \cdot \sgn(\zeta_p(\kappa_p))) = 0\bigr\} \\
		&= \inf \bigl\{t > \kappa_p \bigm| |\zeta_p(t)| = \lambda\bigr\} = \kappa_{p+1}
	\end{align*}
	so that $\Mod_H g(t) = \Mod_\lambda^{h,j_\alpha} g(t)$ for all $t \leq \tau_{p+1} = \kappa_{p+1}$.
	By induction we can conclude that $\Mod_H g = \Mod_\lambda^{h,j_\alpha} g$.
	 
	We now prove that $\Mod_\lambda^{h,j_\alpha}g = \Mod_{\lambda,\sigma}^{h,j_\alpha} g$.
	By definition we have $\nu_1 = \kappa_1$ and $\Mod_\lambda^{h,j_\alpha}g = \Mod_{\lambda,\sigma}^{h,j_\alpha} g$ for all $t \leq \nu_1 = \kappa_1$.
	Assume that for some $p \in \N$ we have $\nu_n = \kappa_n < \infty$ for all $n \leq p$.
	Hence, $\Mod_\lambda^{h,j_\alpha}g(t) = \Mod_{\lambda,\sigma}^{h,j_\alpha} g(t)$ for all $t \leq \nu_p = \kappa_p $ and $\zeta_{p} = \xi_{p}$.
	On the one hand,
	\begin{align*}
		\nu_{p+1} &= \inf \bigl\{t > \nu_p + \sigma \bigm| |\xi_p(t)|\geq \lambda\bigr\}\\
		&= \inf \bigl\{t > \kappa_p + \sigma \bigm| |\zeta_p(t)| \geq \lambda\bigr\} \\
		&\geq \inf \bigl\{t > \kappa_p \bigm| |\zeta_p(t)|\geq \lambda\bigr\} = \kappa_{p+1}.
	\end{align*}
	On the other hand, since $\tau_p = \kappa_p$ and $\tau_{p+1} = \kappa_{p+1}$, Theorem~\ref{theo:Mod_H_range} gives $\kappa_{p+1} = \tau_{p+1} \geq \tau_p + \alpha =  \kappa_p + \alpha$ and we obtain
	\begin{align*}
		\nu_{p+1} &= \inf \bigl\{t > \kappa_p + \sigma \bigm| |\zeta_p(t)| \geq \lambda\bigr\} \\
		&\leq \inf \bigl\{t > \kappa_p +\alpha \bigm| |\zeta_p(t)|\geq \lambda\bigr\} \\
		&= \inf \bigl\{t > \kappa_p \bigm| |\zeta_p(t)| \geq \lambda\bigr\} = \kappa_{p+1}.
	\end{align*}
	This gives $\Mod_\lambda^{h,j_\alpha}g(t) = \Mod_{\lambda,\sigma}^{h,j_\alpha} g(t)$ for all $t \leq \nu_{p+1} = \kappa_{p+1}$ and, by induction, we conclude that $\Mod_\lambda^{h,j_\alpha}g = \Mod_{\lambda,\sigma}^{h,j_\alpha} g $.
\end{proof}

If $\alpha > 0$ and the assumptions of Theorem \ref{theo:Mod_comparison} are not satisfied, the generalized modulo operators $\Mod_H$, $\Mod_\lambda^{h,j_\alpha}$ and $\Mod_{\lambda,\sigma}^{h,j_\alpha}$ in general result in different output functions.
For comparison, the conditions required to guarantee the properties studied in this paper for the three operators are summarized in Table~\ref{tab:Mod_comparison}.

\begin{figure}[t]
	\centering
	\includegraphics[width=\linewidth]{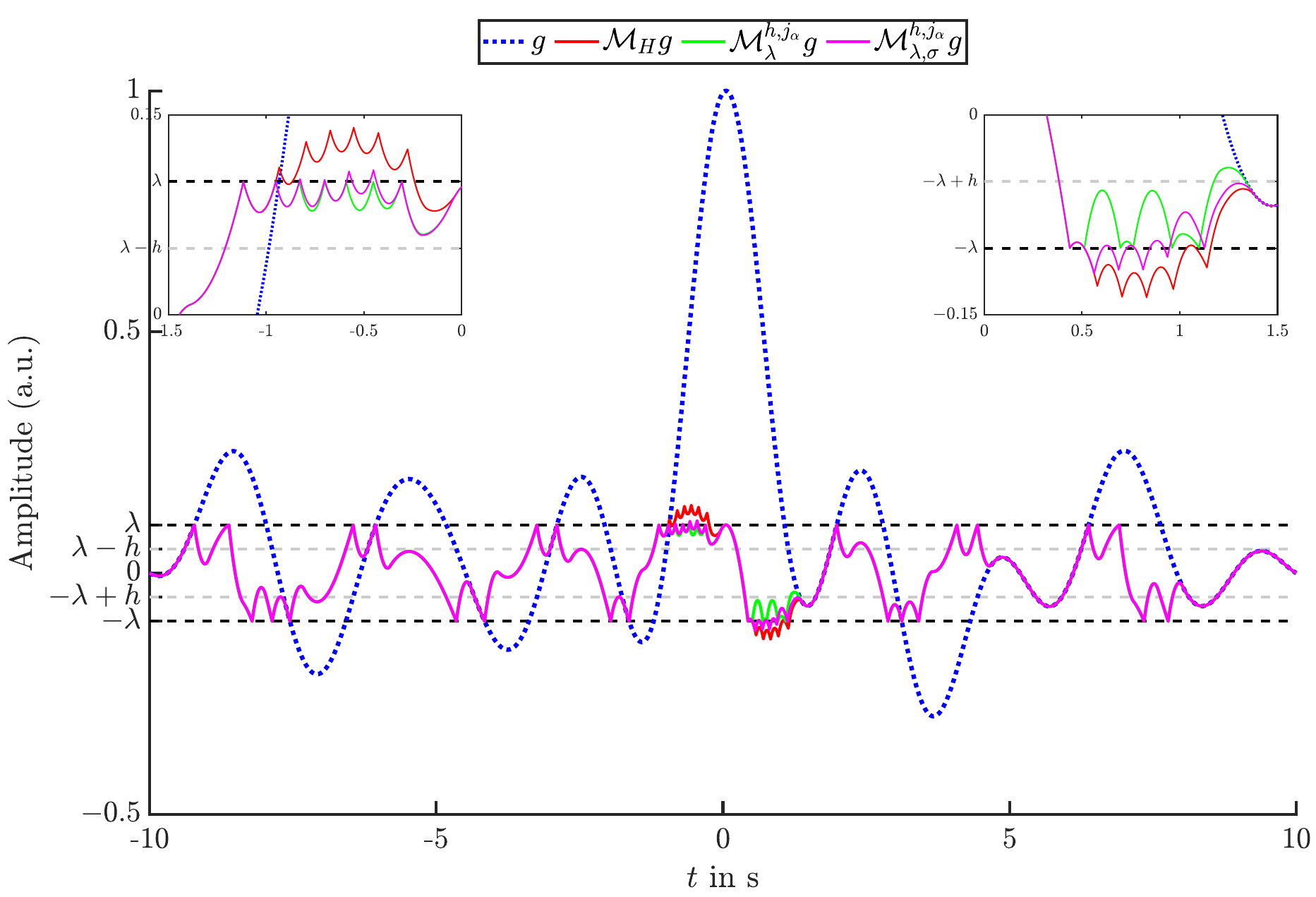}
	\caption{Illustration of $\Mod_H g$, $\Mod_{\lambda}^{h,j_\alpha} g$ and $\Mod_{\lambda,\sigma}^{h,j_\alpha} g$ for $j_\alpha = j_{\alpha}^{(2)}$ and $g \in \PW_\Omega$ with $\Omega = \pi \tfrac{\text{rad}}{\text{s}}$, where $\lambda = 0.1$, $h = 0.05$, $\alpha = 0.25$, $\sigma = 0.125$.}
	\label{fig:Mod_comparison}
\end{figure}

For a function $g$ not meeting the assumptions of Theorem~\ref{theo:Mod_comparison} we now study the behaviour of the proposed modulo operators numerically.
This requires an approximation of the folding points $(\tau_n)_{n \in \N}$, $(\kappa_n)_{n \in \N}$ and $(\nu_n)_{n \in \N}$.
To this end, we choose a sufficiently small sampling rate $d > 0$ and set
\begin{equation} \label{eq:discrete_folding}
	\begin{aligned}
		\tau_{n+1} &= \inf \bigl\{k d > \tau_n \bigm| |\eta_n^{(0)}(k d)| \geq \lambda , ~ k \in \Z \bigr\},\\
		\kappa_{n+1} &= \inf \bigl\{k d >\kappa_n \bigm| |\zeta_n(k d)|\geq \lambda, ~ k \in \Z \bigr\},\\
		\nu_{n+1} &= \inf \bigl\{k d > \nu_n + \sigma \bigm| |\xi_n(k d)| \geq \lambda, ~ k \in \Z \bigr\}.
	\end{aligned}
\end{equation}
Moreover, we consider one of the three folding functions $ j_\alpha^{(1)}, \, j_\alpha^{(2)}, \, j_\alpha^{(3)} \colon \R \to \R $, for $t \in [0,\alpha]$ given by
\begin{equation*}
	j_{\alpha}^{(1)} (t) =  \tfrac{t}{\alpha}, ~
	j_{\alpha}^{(2)} (t) =  1 - (\tfrac{\alpha-t}{\alpha})^2,~
	j_{\alpha}^{(3)} (t) =  3 (\tfrac{t}{\alpha})^2 - 2  (\tfrac{t}{\alpha})^3.
\end{equation*}
Because condition~\eqref{eq:fold_con2} is not satisfied for $j_{\alpha}^{(3)}$, we cannot combine it with the generalized modulo operator $\Mod_\lambda^{h,j_\alpha}$.

\begin{figure}[t]
	\centering
	\includegraphics[width=\linewidth]{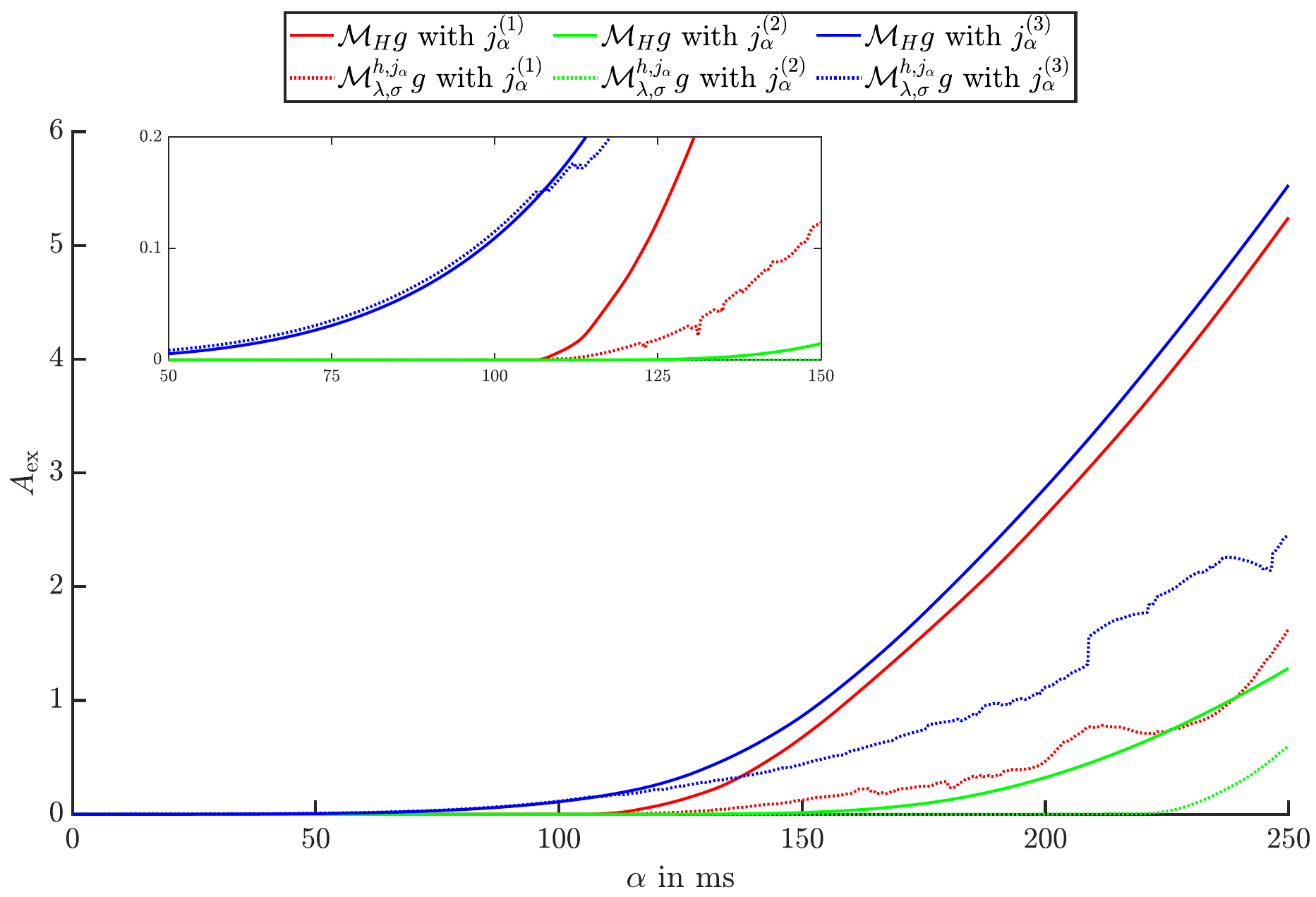}
	\caption{Exceedance $A_{\text{ex}}$ for $\Mod_H g$ (solid) and $\Mod_{\lambda,\sigma}^{h,j_\alpha} g$ (dotted) with $\sigma = \frac{\alpha}{2}$ for random function $g \in \PW_\Omega$ with $\Omega = \pi \tfrac{\text{rad}}{\text{s}}$, where $d = 0.0463 \text{ms}$ and $j_{\alpha} = j_{\alpha}^{(1)}$ (red), $j_{\alpha} = j_{\alpha}^{(2)}$ (green) or $j_{\alpha} = j_{\alpha}^{(3)}$ (blue).}
	\label{fig:Mod_exceedance_area}
\end{figure}

We are mainly interested in studying the exceedance of the range $[-\lambda,\lambda]$.
In Theorem~\ref{theo:Mod_j_range_separation} we have shown that $\Mod_\lambda^{h,j_\alpha} g$ always stays within $[-\lambda,\lambda]$.
This can indeed be observed numerically in Figure~\ref{fig:Mod_comparison}.
In contrast to this, $\Mod_H g$ and $\Mod_{\lambda,\sigma}^{h,j_\alpha} g$ leave the range $[-\lambda,\lambda]$.
To quantify the severeness, we calculate the exceedance area of $g_\lambda \in \{\Mod_H g, \Mod_{\lambda,\sigma}^{h,j_\alpha} g\}$ via
\begin{equation*}
	A_{\text{ex}} = \int_{|g_\lambda(t)|> \lambda} (|g_\lambda(t)| - \lambda) \: \d t
\end{equation*}
for all three folding functions.
The results are shown in Figure~\ref{fig:Mod_exceedance_area} for varying values of $\alpha$.
We observe that for all folding functions the exceedance of $\Mod_H g$ is more severe than for $\Mod_{\lambda,\sigma}^{h,j_\alpha} g$ and the difference is more pronounced for larger values of $\alpha$.
Moreover, the exceedance grows with smaller initial slopes of the folding function $j_\alpha$.
Indeed, $j_\alpha^{(3)}$ yields the largest exceedance area followed by $j_\alpha^{(1)}$ and, lastly, $j_\alpha^{(2)}$.

\section{Reconstruction using OMP} \label{sec:reconstruction}

We first state a condition for the sampling rate $\T > 0$ to uniquely determine $g \in \PW_\Omega \cap \Cont^{0,1}_{\lambda,\tau_0}$ with bandwidth $\Omega > 0$ from discrete samples of $g_\lambda \in \{\Mod_H g, \Mod_\lambda^{h,j_\alpha} g, \Mod_{\lambda,\sigma}^{h,j_\alpha} g\}$,
\begin{equation*}
	\{g_\lambda(k\T) \mid k \in \Z\},
\end{equation*}
where we assume that $g_\lambda$ is well-defined and that there exists $\rho > 0$ such that $g_\lambda(t) = g(t)$ for all $|t| \geq \rho$, c.f.~Table~\ref{tab:Mod_comparison} for sufficient conditions in the different cases.
In~\cite[Lemma 3]{Beckmann2025a} we have shown that any $g \in \PW_\Omega$ is determined by $\{\A g(k\T) \mid k \in \Z\}$ if the oversampling condition $\T < \frac{\pi}{\Omega}$ is met, where $\A: \PW_\Omega \to \Lebesgue^2(\R)$ is a general operator with
\begin{equation*}
\A g(t) = g(t)
\quad \forall \, t \in \R \setminus M
\end{equation*}
for some bounded set $M \subset \R$.
Hence, applying \cite[Lemma 3]{Beckmann2025a} to our setting directly gives the following identifiability result.

\begin{proposition}
	For $g \in \PW_\Omega \cap \Cont^{0,1}_{\lambda,\tau_0}$ let $g_\lambda$ be well-defined and $\rho > 0$ with $g_\lambda(t) = g(t)$ for all $|t| \geq \rho$.
	Then, $g$ is uniquely determined by the samples $\{g_\lambda(k\T) \mid k \in \Z\}$ if $\T < \frac{\pi}{\Omega}$.\qed
\end{proposition}

We now adapt the recovery approach from \cite{Beckmann2025a} to our setting.
To this end, we assume that we are given discrete samples
\begin{equation*}
	\{g_\lambda(k\T) \mid k \in \Z\}
\end{equation*}
with $\T<\frac{\pi}{\Omega}$ and $g_\lambda(k\T) = g(k\T)$ for all $|k| \geq K \in \N$.
Then, there exist $L_\lambda \in \N$ folding points $\kappa_1,\ldots,\kappa_{L_\lambda} \in [-K,K]\T$ and signs $\sigma_1,\ldots,\sigma_{L_\lambda} \in \{\pm1\}$ such that $g_\lambda = g - s_\lambda$ with $s_\lambda = (2\lambda-h) \sum_{l=1}^{L_\lambda} \sigma_l \, j_\alpha(\cdot - \kappa_l)$.
As in \cite{Beckmann2025a}, we can represent the samples of the folding function $j_\alpha$ by the sum of multiple instantaneous folds on $\T\Z$.
Hence, defining $\underline{g} = g(\cdot+\T)-g$, we obtain $\underline{g}_\lambda = \underline{g} - \underline{s}_\lambda$ with
\begin{equation*}
	\underline{s}_\lambda(n\T) = \sum\nolimits_{\ell \in \Ell_\lambda} c_\ell \, \delta(n\T - t_\ell)
\end{equation*}
for suitable $\Ell_\lambda \subseteq \{-K,\dots,K\}$, $c_\ell \in \R$ and $t_\ell \in (\T\Z) \cap [-K\T,K\T]$, where $|\Ell_\lambda| \leq L_\lambda \, (\lceil\frac{\alpha}{\T}\rceil + 1)$.
We now consider the discrete time Fourier transform (DTFT)
\begin{equation*}
	\widehat{\underline{g}}(\omega) = \sum\nolimits_{n\in\Z} \underline{g}(n\T) \, \e^{-\i n \T \omega}
	\quad \text{for } \omega \in \R.
\end{equation*}
As $g \in \PW_\Omega$ implies $\underline{g} \in \PW_\Omega$, Poisson's summation formula
\begin{equation*}
	\sum\nolimits_{n \in \Z} \Fourier \underline{g}\bigl(\omega + \tfrac{2\pi}{\T}n\bigr) = \T \sum\nolimits_{n \in \Z} \underline{g}(n\T) \, \e^{-\i n \T \omega}
\end{equation*}
implies that $\widehat{\underline{g}}$ is $\frac{2\pi}{\T}$-periodic and, for $\omega \in [0,\frac{2\pi}{\T})$,
\begin{equation*}
	\widehat{\underline{g}}(\omega) = \begin{cases}
		\tfrac{1}{\T} \Fourier \underline{g}(\omega) & \text{for } 0 \leq \omega \leq \Omega, \\
		0 & \text{for } \Omega < \omega < \tfrac{2\pi}{\T}-\Omega, \\
		\tfrac{1}{\T} \Fourier \underline{g}(\omega-\tfrac{2\pi}{\T}) & \text{for } \tfrac{2\pi}{\T}-\Omega \leq \omega < \tfrac{2\pi}{\T}.
	\end{cases}
\end{equation*}
In particular, for all $\Omega < \omega < \tfrac{2\pi}{\T}-\Omega$ we obtain
\begin{equation*}
	\widehat{\underline{g}}_\lambda(\omega) = - \widehat{\underline{s}}_\lambda(\omega) = -\sum\nolimits_{\ell \in \Ell_\lambda} c_\ell \, \e^{-\i \omega t_\ell}.
\end{equation*}
and, with $\omega_m = \frac{2\pi m}{N\T} \in [0,\frac{2\pi}{\T})$ for $m=0,\ldots,N-1$ and $N \geq K$,
\begin{equation*}
	\widehat{\underline{g}}_\lambda(\omega_m) = -\sum\nolimits_{\ell \in \Ell_\lambda} c_\ell \, \e^{-\i \underline{\omega}_0 m \ell}
	\quad \forall \, m \in \I_{N_\Omega,N},
\end{equation*}
where $\underline{\omega}_0 = \tfrac{2\pi}{N}$ and $\I_{N_\Omega,N} = \{N_\Omega+1,\ldots,N-N_\Omega-1\}$ with $N_\Omega = \lfloor\frac{\Omega N \T}{2\pi}\rfloor$.
For sake of brevity, we set $g[n] = g((n-K)\T)$ and
\begin{equation*}
	\widehat{\underline{g}}_\lambda[m] = \e^{-\i \underline{\omega}_0 m K} \, \widehat{\underline{g}}_\lambda(\omega_m) = \sum\nolimits_{n\in\Z} \underline{g}_\lambda[n] \, \e^{-\i \underline{\omega}_0 m n}.
\end{equation*}
Hence, with $\bfs = (\widehat{\underline{g}}_\lambda[m])_{m \in \I_{N_\Omega,N}} \in \C^M$, $M = N - 2 N_\Omega - 1$, we can find $\bfc = (c_k)_{k = -K}^{K} \in \C^{2K+1}$ with $c_k = 0$ for $k \not\in \Ell_\lambda$ by solving the minimization problem 
\begin{equation}\label{eq:MinProb}
	\text{minimize } \|\bfc\|_0
	\quad \text{such that}
	\quad \bfV \bfc = \bfs,
	\end{equation}
where $\bfV \in \C^{M \times (2K+1)}$ is a Vandermonde matrix with entries $\bfV_{m-N_\Omega,n+1} = \e^{-\i \underline{\omega}_0 mn}$ for $m \in \I_{N_\Omega,N}$, $n = 0,\ldots,2K$.
The solution vector $\bfc \in \C^{2K+1}$ of~\eqref{eq:MinProb} encodes the sought parameters $\{c_\ell\}_{\ell \in \Ell_\lambda}$ as its non-zero vector components and the set $\{t_\ell\}_{\ell \in \Ell_\lambda}$ is determined by the corresponding indices.

\begin{algorithm}[t]
	\caption{(OMP)}
	\label{alg:omp}
	\begin{algorithmic}[1]
		\Require Signal $\bfs \in \C^M$ and dictionary matrix $\bfV \in \C^{M \times J}$, error tolerance $\varepsilon \geq 0$
		\medskip
		\State $\bfc^{(0)} = 0 \in \C^J$, $\S^{(0)} = \emptyset$, $i = 1$
		\While{$\lVert\bfV^\ast(\bfs - \bfV \bfc^{(i-1)})\rVert_\infty > \varepsilon $}
		\smallskip
			\State $j^{(i)} = \argmax_{1 \leq j \leq J} \lvert [\bfV^\ast(\bfs - \bfV \bfc^{(i-1)})]_j \rvert$
			\State $\S^{(i)} = \S^{(i-1)} \cup \{j^{(i)}\}$
			\State $\bfc^{(i)} = \argmin_{\bfc} \left\{\lVert\bfs - \bfV \bfc\rVert_2 \mid \supp(\bfc) \subseteq 		\S^{(i)}\right\}$
			\State $i = i+1$
		\smallskip
		\EndWhile
		\medskip
		\Ensure $\bfc^{(i_{\text{end}})} \in \C^J$
	\end{algorithmic}
\end{algorithm}

As explained in \cite{Beckmann2022a, Beckmann2024} in the context of the modulo Radon transform~\cite{Beckmann2022}, the optimization problem~\eqref{eq:MinProb} can be solved by applying the orthogonal matching pursuit (OMP) algorithm, which was first proposed in~\cite{Mallat1993}.
Here we use a variant as described in Algorithm~\ref{alg:omp}.
Given the estimates $\{c_\ell,t_\ell\}_{\ell \in \Ell_\lambda}$, we can form $\underline{s}_\lambda(k\T) = \sum\nolimits_{\ell \in \Ell_\lambda} c_\ell \, \delta(k\T - t_\ell)$ for $k \in \Z$ and compute the samples $s_\lambda(k\T)$ by applying anti-differences, i.e., $s_\lambda(k\T) = \sum_{j < k} \underline{s}_\lambda(j\T)$.
Adding $s_\lambda(k\T)$ to $g_\lambda(k\T)$ finally gives $g(k\T)$ and we can recover $g$ via Shannon interpolation, i.e.,
\begin{equation*}
	g = \sum\nolimits_{k \in \Z} g(k\T) \sinc\Bigl(\frac{\pi}{\T}(\cdot - k\T)\Bigr).
\end{equation*}

In~\cite[Theorem 1]{Beckmann2024} it is shown that Algorithm~\ref{alg:omp} with exact arithmetic and $\varepsilon = 0$ recovers the parameters $\{c_\ell,t_\ell \}_{\ell \in \Ell_\lambda}$ if $\max_{\ell \in \Ell_\lambda + K} |\ell| < N-2N_\Omega-1$.
As $\max_{\ell \in \Ell_\lambda + K} |\ell| \leq 2K-1$ because of $g_\lambda(k\T) = g(k\T)$ for all $|k| \geq K$ by choice of $K \in \N$, we set $N = K + K^\prime$ and require that $2K-1 < N-2N_\Omega-1$.
Defining the oversampling factor $\OF = \frac{\pi}{\Omega\T} > 1$, this can be guaranteed if
\begin{equation*}
K^\prime > \frac{1+\OF^{-1}}{1-\OF^{-1}} K.
\end{equation*}
In total, we obtain the following reconstruction result for OMP.

\begin{theorem} \label{theo:OMP_recovery}
	For $g \in \PW_\Omega \cap \Cont^{0,1}_{\lambda,\tau_0}$ let $g_\lambda$ be well-defined and $\rho > 0$ with $g_\lambda(t) = g(t)$ for all $|t| \geq \rho$.
	Assume we are given discrete samples $\{g_\lambda(k\T) \mid k \in \Z\}$ with $\T < \frac{\pi}{\Omega}$.
	Then, $g$ is exactly recovered by OMP with $\varepsilon = 0$ and $N = K + K^\prime$, where
	\begin{equation*}
		K \geq \rho \T^{-1}, \quad
		K^\prime > \frac{1+\OF^{-1}}{1-\OF^{-1}} K.
		\tag*{\qed}
	\end{equation*}
\end{theorem}

Note that the above statement does not guarantee that OMP only selects correct folding times.
However, for falsely selected positions the corresponding coefficients will be zero.
To guarantee stable recovery of the folding times, additional conditions are needed as the restricted isometry property or mutual coherence conditions.
This will depend on the separation of folds and is beyond the scopes of this paper.

\subsection{Iterated forward differences}

For $\alpha > 0$, the above reconstruction approach splits one non-instantaneous fold into up to $\lceil\frac{\alpha}{\T}\rceil + 1$ instantaneous folds.
This makes the reconstruction via OMP more challenging, especially for large values of $\alpha$ compared to $\T$.
To deal with this case, we now modify our reconstruction approach for certain smooth folding functions.
More precisely, we assume that $j_\alpha \in \Cont^{p-2}(\R)$ for some $p \geq 2$ and $j_\alpha|_{[0,\alpha]} \in \P_{p-1}$, i.e., $j_\alpha$ is a polynomial of degree $p-1$ on $[0,\alpha]$.
Note that such a folding function cannot satisfy assumption~\eqref{eq:fold_con2} so that in the following we have to deal with $\Mod_H$ or $\Mod_{\lambda,\sigma}^{h,j_\alpha}$.
Moreover, we consider the forward difference operator $\Delta_\T \colon \Cont(\R) \to \Cont(\R)$ given by $\Delta_\T f = f(\cdot+\T) - f$ and the divided difference operator $[x_0,\ldots,x_p] \colon \Cont(\R) \to \R$ with $x_j = (j+m)\T$ for $j = 0,\ldots,p$ and some $m \in \Z$, defined via the recursion
\begin{equation*}
	[x_j,\ldots,x_k](f) = \frac{[x_{j+1},\ldots,x_k](f) - [x_j,\ldots,x_{k-1}](f)}{x_k - x_j}
\end{equation*}
for $0 \leq j < k \leq p$, where $[x_j](f) = f(x_j)$ for $j = 0,\ldots,p$.
By induction one can show that
\begin{equation*}
	\forall \, 0 \leq k \leq p ~ \forall \, 0 \leq j \leq p-k \colon [x_j,\ldots,x_{j+k}](f) = \frac{\Delta_\T^k f(x_j)}{k! \, \T^k}
\end{equation*}
and, if $f \in \Cont^p([x_0,x_\ell])$, the mean-value theorem gives
\begin{equation*}
	[x_0,\ldots,x_p](f) = \frac{f^{(p)}(\tau)}{p!}
	\quad \mbox{for some } \tau \in [x_0,x_p].
\end{equation*}
Hence, if $[x_0,x_p] \subseteq [0,\alpha]$ we observe that $\Delta_\T^p j_\alpha(m\T) = 0$.
The same holds if $[x_0,x_p] \subset (-\infty,0]$ or $[x_0,x_p] \subset [\alpha,\infty)$ so that $\Delta_\T^p j_\alpha(m\T) \neq 0$ only if $0 \in (x_0,x_p)$ or $\alpha \in (x_0,x_p)$.
Analogously, for $j_\alpha^{(\kappa)} = j_\alpha(\cdot-\kappa)$ located at $\kappa \in \R$ we have $\Delta_\T^p j_\alpha^{(\kappa)}(m\T) \neq 0$ only if $\kappa \in (x_0,x_p)$ or $\kappa+\alpha \in (x_0,x_p)$, which is equivalent to $m \in (\frac{\kappa}{\T}-p,\frac{\kappa}{\T}) \cup (\frac{\kappa+\alpha}{\T}-p,\frac{\kappa+\alpha}{\T})$.
Consequently, if $j_\alpha \in \Cont^{p-2}(\R)$ is a polynomial of degree $p-1$ on $[0,\alpha]$, applying the iterated forward difference operator $\Delta_\T^p$ splits one non-instantaneous fold into at most $2p$ instantaneous folds, which is advantageous over the previous approach if
\begin{equation*}
\Bigl\lceil\frac{\alpha}{\T}\Bigr\rceil > 2p-1.
\end{equation*}

To explain the full adapted recovery approach, we again assume that we are given discrete samples
\begin{equation*}
	\{g_\lambda(k\T) \mid k \in \Z\}
\end{equation*}
with $\T<\frac{\pi}{\Omega}$ and $g_\lambda(k\T) = g(k\T)$ for all $|k| \geq K \in \N$.
Then, there exist $L_\lambda \in \N$ folding points $\kappa_1,\ldots,\kappa_{L_\lambda} \in [-K,K]\T$ and signs $\sigma_1,\ldots,\sigma_{L_\lambda} \in \{\pm1\}$ such that $g_\lambda = g - s_\lambda$ with $s_\lambda = (2\lambda-h) \sum_{l=1}^{L_\lambda} \sigma_l \, j_\alpha(\cdot - \kappa_l)$.
Now, defining $\underline{g} = \Delta_\T^p g$, we obtain $\underline{g}_\lambda = \underline{g} - \underline{s}_\lambda$ with
\begin{equation*}
	\underline{s}_\lambda(n\T) = \sum\nolimits_{\ell \in \Ell_\lambda} c_\ell \, \delta(n\T - t_\ell)
\end{equation*}
for suitable $\Ell_\lambda \subseteq \{-K,\dots,K\}$ and $t_\ell \in (\T\Z) \cap [-K\T,K\T]$, where $|\Ell_\lambda| \leq 2p L_\lambda$.
As before, DTFT in combination with the bandlimitedness of $g \in \PW_\Omega$ gives
\begin{equation*}
	\widehat{\underline{g}}_\lambda[m] = -\sum\nolimits_{\ell \in \Ell_\lambda} c_\ell \, \e^{-\i \underline{\omega}_0 m (\ell+K)}
	\quad \forall \, m \in \I_{N_\Omega,N},
\end{equation*}
where the left-hand side is computable from given modulo data $g[n] = g((n-K)\T)$ via
\begin{equation*}
	\widehat{\underline{g}}_\lambda[m] = \sum\nolimits_{n\in\Z} \underline{g}_\lambda[n] \, \e^{-\i \underline{\omega}_0 m n}.
\end{equation*}
Consequently, we end up with the same minimization problem 
\begin{equation*}
	\text{minimize } \|\bfc\|_0
	\quad \text{such that}
	\quad \bfV \bfc = \bfs,
	\end{equation*}
which can be solved by applying OMP.
Given the estimates $\{c_\ell,t_\ell\}_{\ell \in \Ell_\lambda}$, we can form $\underline{s}_\lambda(k\T) = \sum\nolimits_{\ell \in \Ell_\lambda} c_\ell \, \delta(k\T - t_\ell)$ for $k \in \Z$ and compute the samples $s_\lambda(k\T)$ by iteratively applying anti-differences $p$-times.
Adding $s_\lambda(k\T)$ to $g_\lambda(k\T)$ finally gives $g(k\T)$ and we can recover $g$ via Shannon interpolation, i.e.,
\begin{equation*}
	g = \sum\nolimits_{k \in \Z} g(k\T) \sinc\Bigl(\frac{\pi}{\T}(\cdot - k\T)\Bigr).
\end{equation*}

\subsection{Reconstruction adapted to folding function}

The ability to model the folding function allows us to adapt the dictionary for the Vandermonde matrix $\bfV$ so that a fold can be represented by \emph{one} column of the matrix.
This is especially useful if the transient $\alpha$ is much larger than the sampling rate~$\T$.
To this end, we assume that the folding points satisfy $\kappa_1,\ldots,\kappa_{L_\lambda} \in (\T\Z) \cap [-K\T,K\T]$.
As this is usually not satisfied, we encounter a small error by approximating folds not exactly located at sampling points with ones located at the nearest sampling points.
Now, for $n \in \Z$ we have
\begin{equation}\label{eq:sLambdaDic}
	s_\lambda(n\T) =  \sum\nolimits_{\ell \in \Ell_\lambda} c_\ell \, j_\alpha(n\T - t_\ell),
\end{equation}
with $c_\ell \in (2\lambda-h)\Z$ and applying forward differences gives
\begin{equation*}
	\underline{s}_\lambda(n\T) = \sum\nolimits_{\ell \in \Ell_\lambda} c_\ell \,( j_\alpha((n+1)\T - t_\ell) - j_\alpha(n\T - t_\ell)),
\end{equation*}
where $t_\ell \in (\T\Z) \cap [-K\T,K\T]$ and $\Ell_\lambda \subseteq \{-K,\dots,K\}$ with $|\Ell_\lambda| = L_\lambda$.
Thereon, DTFT in combination with the bandlimitedness of $g \in \PW_\Omega$ yields
\begin{equation*}
	\widehat{\underline{g}}_\lambda[m] = -\sum_{\ell \in \Ell_\lambda} c_\ell \,\sum_{k = 0}^{\lceil\frac{\alpha}{\T}\rceil - 1} \underline{j}_\alpha (k\T) \, \e^{-\i \underline{\omega}_0 m(k+\ell+K)}
\end{equation*}
for all $m \in \I_{N_\Omega,N}$.
Hence, with $\bfs = (\widehat{\underline{g}}_\lambda[m])_{m \in \I_{N_\Omega,N}} \in \C^M$ we obtain the minimization problem 
\begin{equation*}
	\text{minimize } \|\bfc\|_0
	\quad \text{such that}
	\quad \bfV \bfc = \bfs,
\end{equation*}
however with $\bfV_{m-N_\Omega,n+1} = \sum_{k = 0}^{\lceil\frac{\alpha}{\T}\rceil-1} \underline{j}_\alpha (k\T) \, \e^{-\i \underline{\omega}_0 m(k+n)}$.
Having obtained the solution vector $\bfc \in \C^{2K+1}$, we can reconstruct $s_\lambda$ using \eqref{eq:sLambdaDic} and, afterwards, $g$ via $g = g_\lambda + s_\lambda$.

\section{Numerical Experiments} \label{sec:numerics}

In practice, only finitely many samples are available and we assume that we are given $\{g_\lambda(k\T) \mid k = -K,\ldots,K^\prime\}$.
Then, we set $N = K + K^\prime + 1 - p$ and approximate the left-hand side $\widehat{\underline{g}}_\lambda[m]$ for {$m = 0,\ldots,N-1$} via
\begin{equation*}
	\widehat{\underline{g}}_\lambda[m] \approx \sum\nolimits_{n=0}^{N-1} \underline{g}_\lambda[n] \, \e^{-\i \underline{\omega}_0 m n},
\end{equation*}
which can be efficiently computed with the fast Fourier transform (FFT).
Moreover, as in~\cite{Beckmann2025a}, we speed up the recovery by applying the stagewise arithmetic orthogonal matching pursuit (SAOMP) algorithm proposed in~\cite{Zhang2018}, see Algorithm~\ref{alg:saomp}, which has additional parameters $\nu$ to find multiple folding points in one iteration and $\mu$ to remove incorrectly selected ones.
Note that $\nu = 1$ and $\mu = 0$ yields the original OMP algorithm.
As a general rule, choosing a smaller value for $\nu$ allows for more folding points to be picked in one iteration and thereby can yield a faster recovery.
However, choosing $\nu$ too small might result in selecting false entries of the dictionary and can hinder accurate recovery.
To account for this, the parameter $\mu$ is incorporated for pruning the dictionary entry selection, whose value should be smaller than the ration of the largest and smallest correct non-zero entry in $\bfc$, which depends on the number of folds occurring at approximately the same time and on the flatness of the folding function.
The parameter $\varepsilon$ is independent of $\nu$ and $\mu$ and accounts for noise.
Its value should be smaller than the smallest non-zero entry in $\bfc$.

\begin{algorithm}[t]
	\caption{(SAOMP)}
	\label{alg:saomp}
	\begin{algorithmic}[1]
		\Require Signal $\bfs \in \C^M$ and dictionary matrix $\bfV \in \C^{M \times J}$, error tolerance $\varepsilon \geq 0$, initial threshold $\nu \in [0,1]$, pruning threshold $\mu \in [0,1]$, maximal iteration number $i_{\text{max}} \in \N $
		\medskip
		\State $\bfc^{(0)} = 0 \in \C^J$, $\S^{(0)} = \emptyset$, $\delta = \nu$, $i = 1$
		\While{$i \leq i_{\text{max}}$ \textbf{and} $\lVert\bfV^\ast(\bfs - \bfV \bfc^{(i-1)})\rVert_\infty > \varepsilon $}
		\smallskip
			\State $\bfr^{(i)} = \bfV^\ast(\bfs - \bfV \bfc^{(i-1)})$
			\State $\S^{(i)} = \S^{(i-1)} \cup \{j \in \{1,\ldots,J\} \mid \lvert \bfr^{(i)}_j \rvert \geq \delta \, \lVert\bfr^{(i)}\rVert_\infty\}$
			\State $\bfc^{(i)} = \argmin_{\bfc} \{\lVert\bfs - \bfV \bfc\rVert_2 \mid \supp(\bfc) \subseteq \S^{(i)}\}$
			\State $\S^{(i)} = \{j \in \{1,\ldots,J\} \mid \lvert \bfc_j^{(i)} \rvert \geq \mu \, \lVert\bfc^{(i)}\rVert_\infty\}$
			\State \textbf{for} $j \not\in \S^{(i)}$: $\bfc_j^{(i)} = 0$
			\State $\delta = \delta + \frac{1-\nu}{i_{\text{max}}}$, $i = i+1$
		\smallskip
		\EndWhile
		\medskip
		\Ensure $\bfc^{(i_{\text{end}})} \in \C^J$
	\end{algorithmic}
\end{algorithm}

We now present some numerical simulations for a proof of concept and to illustrate the versatility of our reconstruction approach.
To this end, we randomly select a function $g \in \PW_\Omega$ with $\Omega = \pi$ and approximate $g_\lambda$ numerically.
To guarantee a good discretization of the generalized modulo operators, we choose $d > 0$ sufficiently small to compute the folding points as in~\eqref{eq:discrete_folding}.
For reconstruction, however, we only use the modulo samples $\{g_\lambda(k\T) \mid k = -K,\ldots,K^\prime\}$ with $\T \gg d$.
Moreover, for simplicity, we mainly focus on the delayed generalized modulo operator $\Mod_{\lambda,\sigma}^{h,j_\alpha}$ with $\sigma = \frac{\alpha}{2}$, which allows us to compare all three folding functions $j_\alpha^{(1)}, \, j_\alpha^{(2)}, \, j_\alpha^{(3)}$ from above.

\begin{figure}[t]
	\centering
	\includegraphics[width=\linewidth]{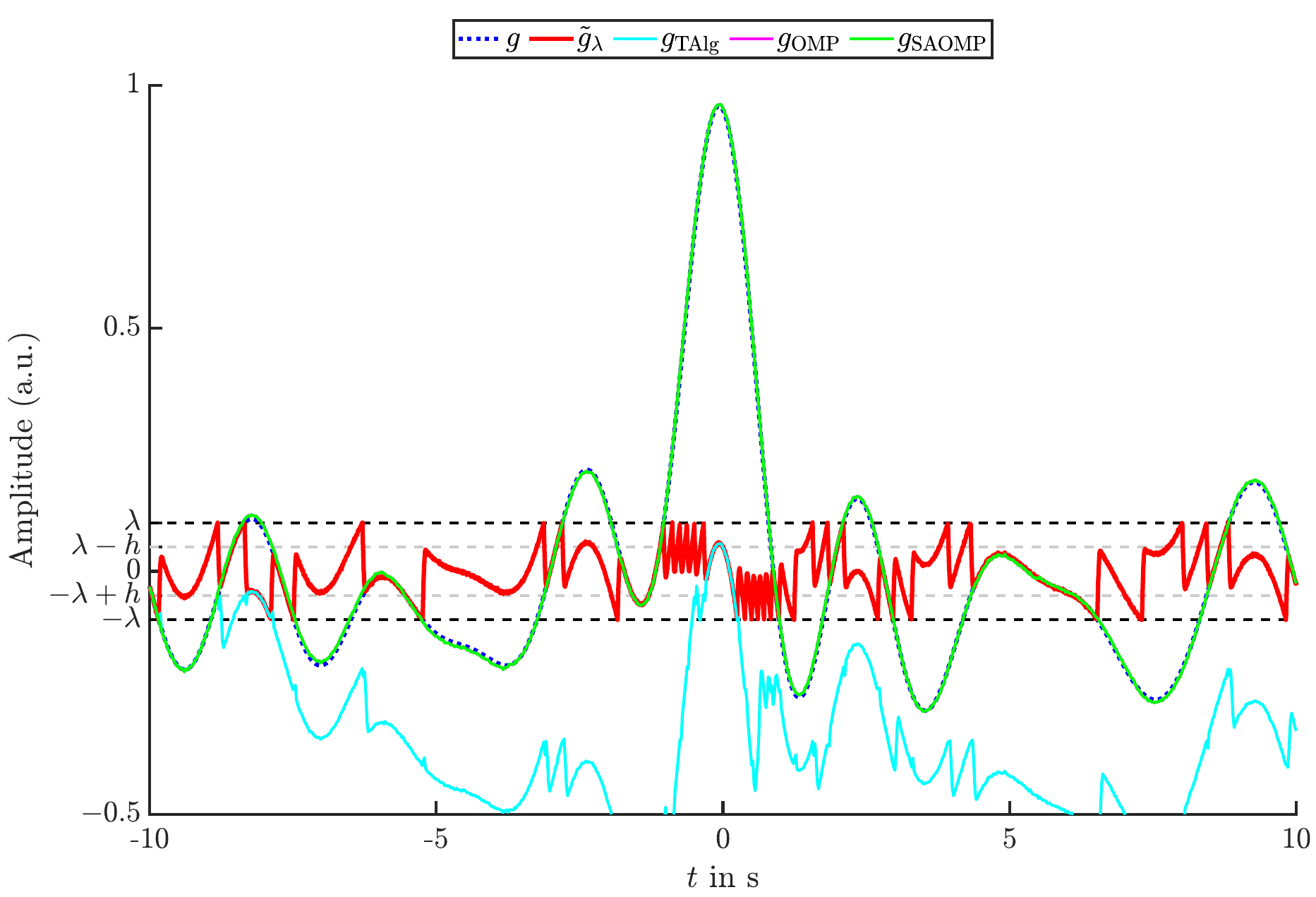}
	\caption{Illustration of TAlg, OMP and SAOMP reconstruction for random $g \in \PW_\Omega$ with $\Omega = \pi \tfrac{\text{rad}}{\text{s}}$ based on noisy modulo samples $\tilde{g}_{\lambda}(n\T)$ with an SNR of $30 \text{dB}$ and $\lambda = 0.1$, $h = 0.05$, $\alpha =  50 \text{ms}$, $\T = 20.8 \text{ms}$, $j_{\alpha} =  j_{\alpha}^{(2)}$.}
	\label{fig:Mod_recon_noisy_j2}
\end{figure}

\begin{figure}[t]
	\centering
	\includegraphics[width=\linewidth]{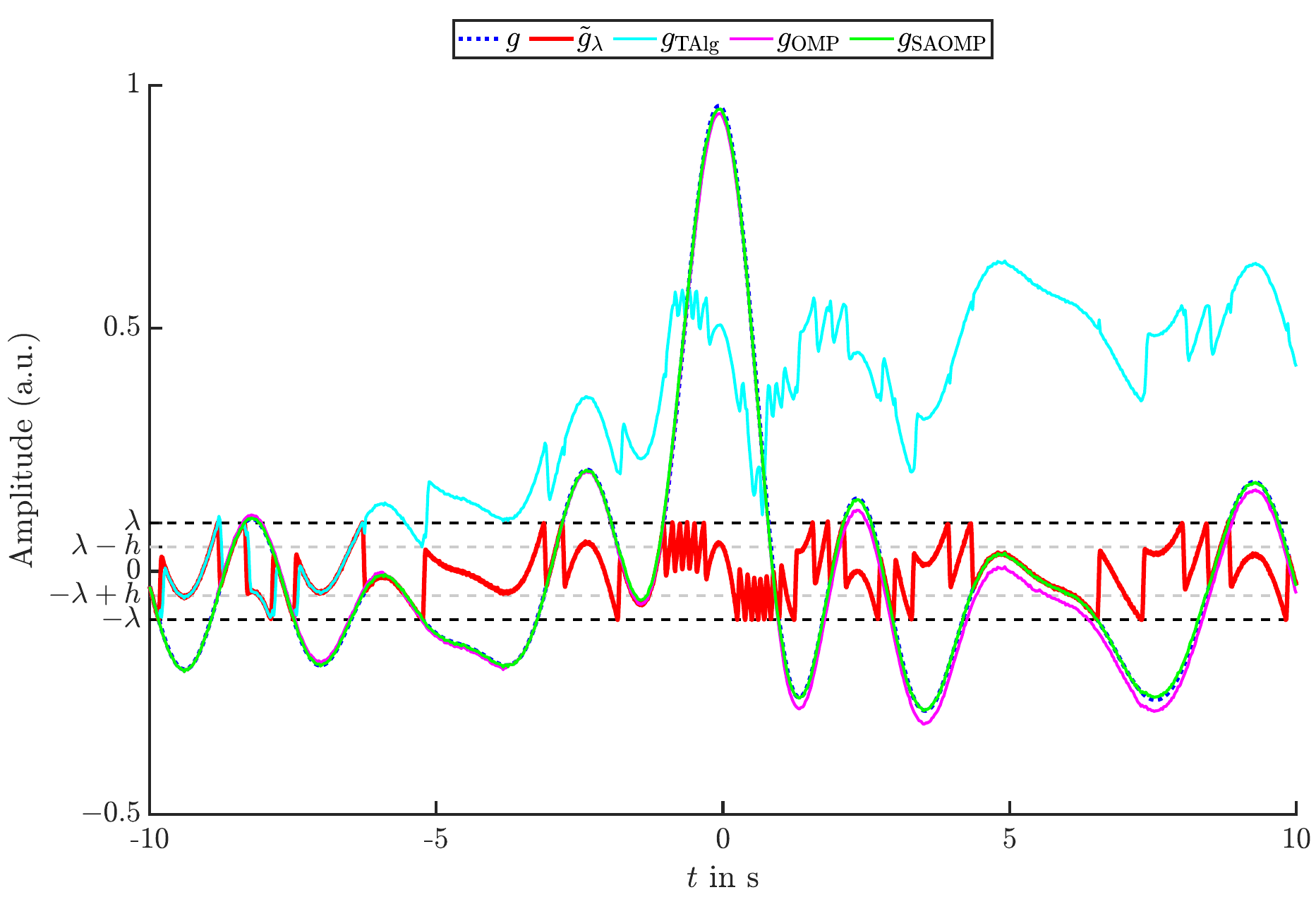}
	\caption{Illustration of TAlg, OMP and SAOMP reconstruction for random $g \in \PW_\Omega$ with $\Omega = \pi \tfrac{\text{rad}}{\text{s}}$ based on noisy modulo samples $ \tilde{g}_{\lambda}(n\T)$ with an SNR of $30 \text{dB}$ and $\lambda = 0.1$, $h = 0.05$, $\alpha =  50 \text{ms}$, $\T = 20.8 \text{ms}$, $j_{\alpha} =  j_{\alpha}^{(3)}$.}
	\label{fig:Mod_recon_noisy_j3}
\end{figure}

In our first experiments we compare our OMP and SAOMP approach with the thresholding algorithm (TAlg) from~\cite{Florescu2022}, which was designed for the generalized modulo encoder $\Mod_H$ with $j_{\alpha}^{(1)}$.
In~\cite{Beckmann2025a} we have already seen that TAlg is in general not able to recover $g$ from $g_\lambda = \Mod_{\lambda}^{h,j_\alpha} g$ with $j_\alpha = j_\alpha^{(1)}$, while SAOMP gives satisfactory results.
The case $g_\lambda = \Mod_{\lambda,\sigma}^{h,j_\alpha} g$ with $j_\alpha = j_\alpha^{(2)}$ is shown in Fig.~\ref{fig:Mod_recon_noisy_j2}, where we chose $\lambda = 0.1$, $h = 0.05$, $\alpha = 0.05$ and $\T = 0.0208$.
Moreover, we added Gaussian noise yielding noisy samples $\tilde{g}_{\lambda}(n\T)$ with a signal-to-noise ratio (SNR) of $30 \text{dB}$.
We observe that OMP and SAOMP yield the same near perfect reconstruction with a mean squared error (MSE) of $2.1 \cdot 10^{-5}$.
However, OMP needs $122$ iterations and $2.6$ seconds, whereas SAOMP only requires $86$ iterations and $1.3$ seconds.
In contrast to this, TAlg is not able to reconstruct $g$ in this scenario.
The results for $j_\alpha = j_\alpha^{(3)}$ are shown in Fig.~\ref{fig:Mod_recon_noisy_j3}, where we use the same parameters and again add Gaussian noise yielding an SNR of $30 \text{dB}$.
As before, TAlg fails, while OMP and SAOMP give satisfactory results.
Now, however, SAOMP performs better than OMP with an MSE of $1.6 \cdot 10^{-5}$, while OMP yields an MSE of $2.8 \cdot 10^{-4}$.

\begin{figure}[t]
	\centering
	\includegraphics[width=\linewidth]{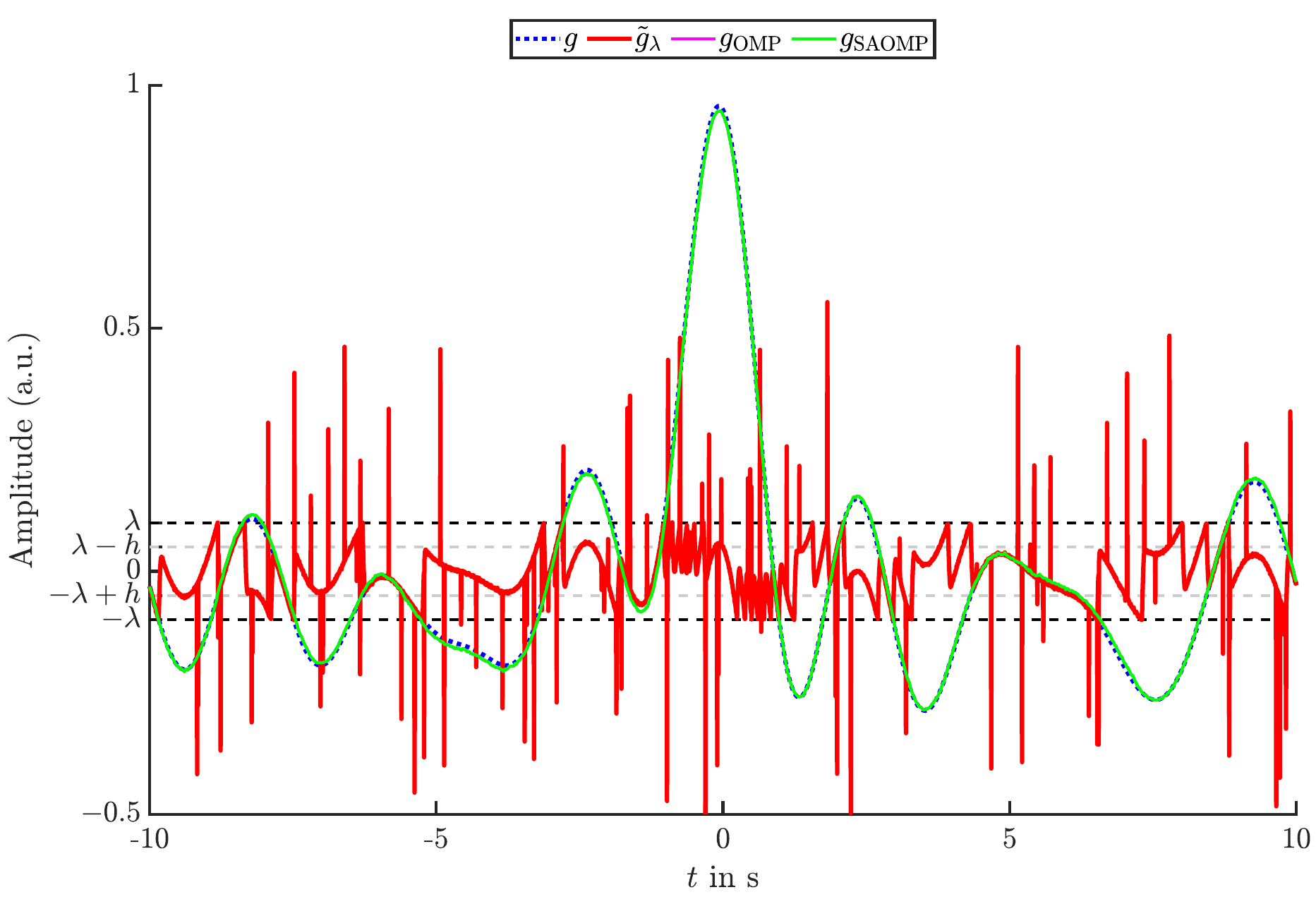}
	\caption{Illustration of (SA)OMP reconstruction for random $g \in \PW_\Omega$ with $\Omega = \pi \tfrac{\text{rad}}{\text{s}}$ based on noisy modulo samples $\tilde{g}_{\lambda}(n\T)$ with shot noise at up to $100$ positions and $\lambda = 0.1$, $h = 0.05$, $\alpha =  50 \text{ms}$, $\T = 20.8 \text{ms}$, $j_{\alpha} =  j_{\alpha}^{(2)}$.}
	\label{fig:Mod_recon_shot_noise_j2}
\end{figure}

In Fig.~\ref{fig:Mod_recon_shot_noise_j2} we repeat the experiment from Fig.~\ref{fig:Mod_recon_noisy_j2} with additional shot noise~\cite{Bhandari2022} at up to $100$ positions and with a maximal magnitude of $0.5$ yielding an SNR of $0.8 \text{dB}$.
Despite this very low SNR, both OMP and SAOMP are able to reconstruct the signal with a MSE of $4.5\cdot10^{-5}$.
This is because each position of shot noise is interpreted by the algorithms as two subsequent folds and is therefore automatically removed.

\begin{figure}[t]
	\centering
	\includegraphics[width=\linewidth]{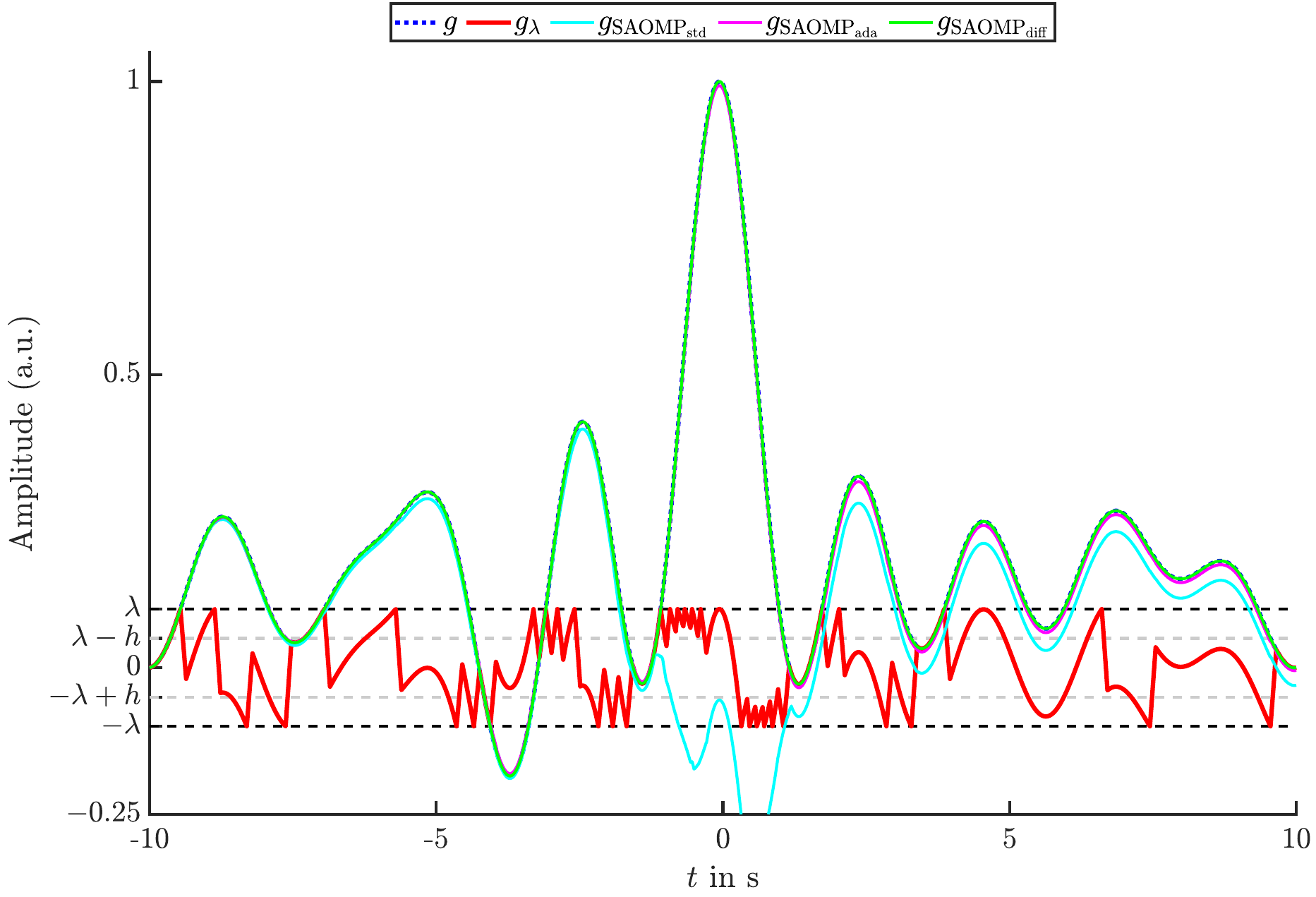}
	\caption{Illustration of SAOMP reconstruction (standard with and without iterated forward differences vs.\ adapted dictionary) for random $g \in \PW_\Omega$ with $\Omega = \pi \tfrac{\text{rad}}{\text{s}}$ based on samples $ g_{\lambda}(n\T)$ and $\lambda = 0.1$, $h = 0.05$, $\alpha =  100 \text{ms}$, $\T = 5.2 \text{ms}$ and $j_{\alpha} = j_{\alpha}^{(1)}$, where the dictionary is adapted to $j_\alpha^{(1)}$.}
	\label{fig:Mod_recon_ada_j1}
\end{figure}

\begin{figure}[t]
	\centering
	\includegraphics[width=\linewidth]{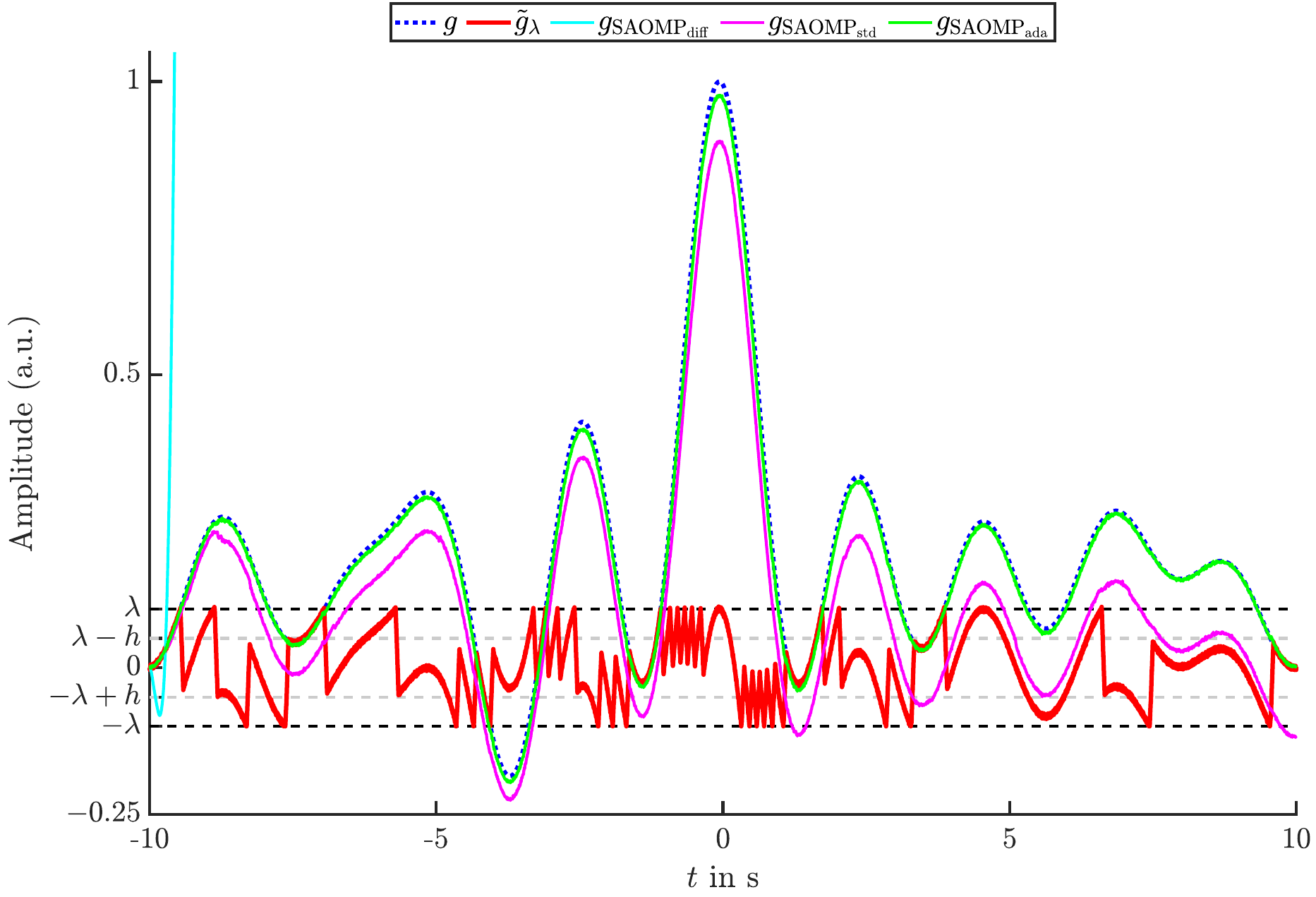}
	\caption{Illustration of SAOMP reconstruction (standard with and without iterated forward differences vs.\ adapted dictionary) for random $g \in \PW_\Omega$, $\Omega = \pi \tfrac{\text{rad}}{\text{s}}$, based on noisy samples $ \tilde{g}_{\lambda}(n\T)$ with an SNR of $30 \text{dB}$ and $\lambda = 0.1$, $h = 0.05$, $\alpha =  50 \text{ms}$, $\T = 5.2 \text{ms}$, $j_{\alpha} =  j_{\alpha}^{(1)}$.}
	\label{fig:Mod_recon_ada_noisy_j1}
\end{figure}

When increasing the quotient $\frac{\alpha}{\T}$, each non-instantaneous fold is split into a growing number of instantaneous ones and the reconstruction via standard SAOMP becomes more challenging.
In this case, our modified approaches with iterated forward differences and adapted dictionaries yield better results.
To illustrate this, in Fig.~\ref{fig:Mod_recon_ada_j1} we recover a random $g \in \PW_\pi$ from $g_\lambda(n\T) = \Mod_{\lambda,\sigma}^{h,j_\alpha} g(n\T)$ with $j_\alpha = j_\alpha^{(1)}$, $\lambda = 0.1$, $h = 0.05$, $\alpha = 0.1$ and $\T = 0.0052$.
While standard SAOMP fails, SAOMP with an adapted dictionary yields an MSE of $2.5 \cdot 10^{-5}$.
The reconstruction quality for SAOMP with iterated forward differences, where $p = 2$, is even better yielding an MSE of $6.1 \cdot 10^{-10}$.
This, however, is only true in the noise-free setting, as shown in Fig.~\ref{fig:Mod_recon_ada_noisy_j1}, where we chose $\alpha = 0.05$ and added Gaussian noise yielding noisy samples $\tilde{g}_{\lambda}(n\T)$ with an SNR of $30 \text{dB}$.
In this case, SAOMP with iterated forward differences fails, while SAOMP with adapted dictionary yields the best result with an MSE of $9.0 \cdot 10^{-5}$.

\begin{figure}[t]
	\centering
	\includegraphics[width=\linewidth,]{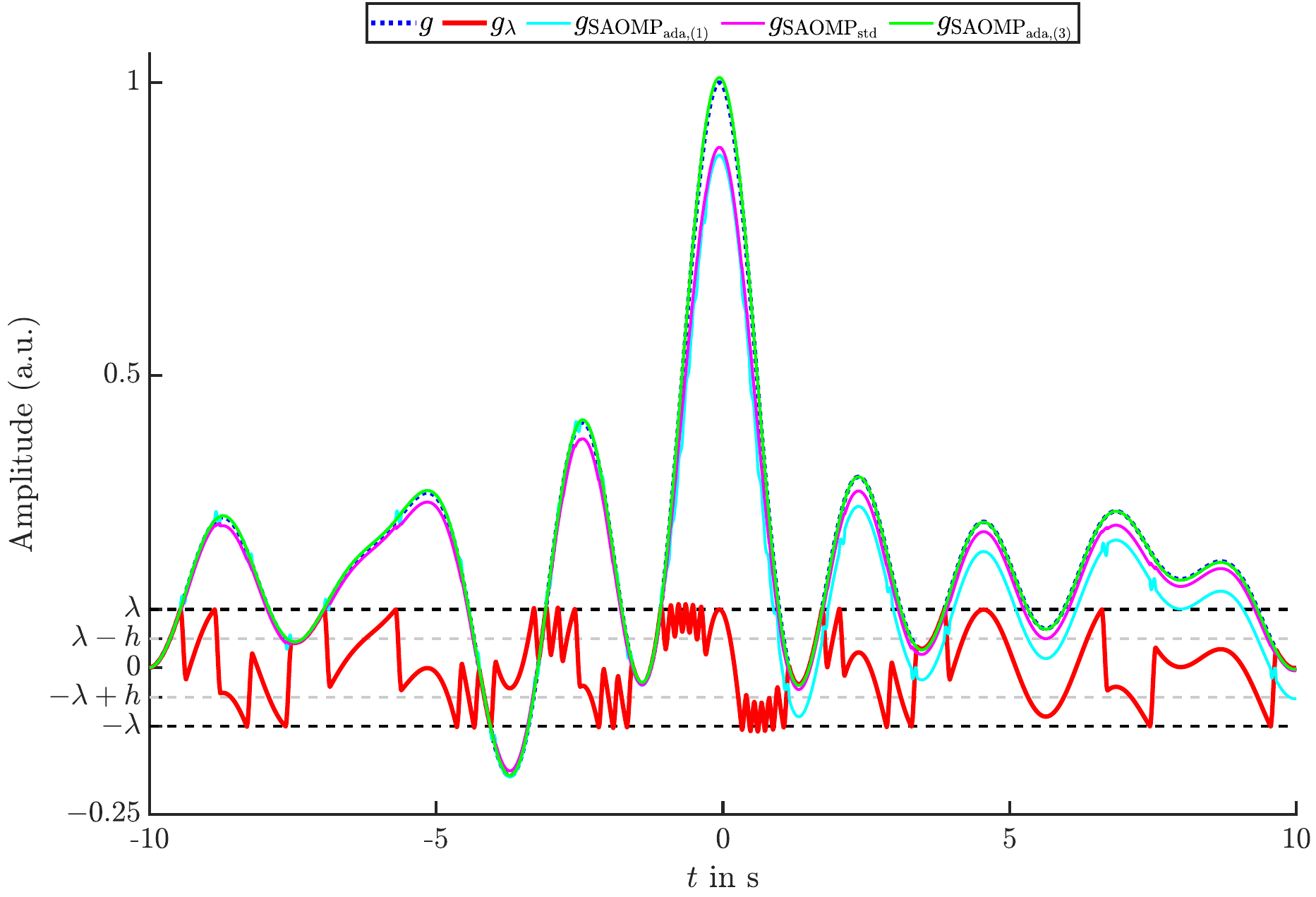}
	\caption{Illustration of SAOMP reconstruction (standard with and without iterated forward differences vs.\ adapted dictionary) for random $g \in \PW_\Omega$ with $\Omega = \pi \tfrac{\text{rad}}{\text{s}}$ based on samples $ g_{\lambda}(n\T)$ and $\lambda = 0.1$, $h = 0.05$, $\alpha =  100 \text{ms}$, $\T = 5.2 \text{ms}$ and $j_{\alpha} = j_{\alpha}^{(3)}$, where the dictionary is falsely adapted to $j_\alpha^{(1)}$ (magenta) and correctly to $j_\alpha^{(3)}$ (cyan).}
	\label{fig:Mod_recon_ada_j3}
\end{figure}

\begin{figure}[t]
	\centering
	\includegraphics[width=\linewidth,]{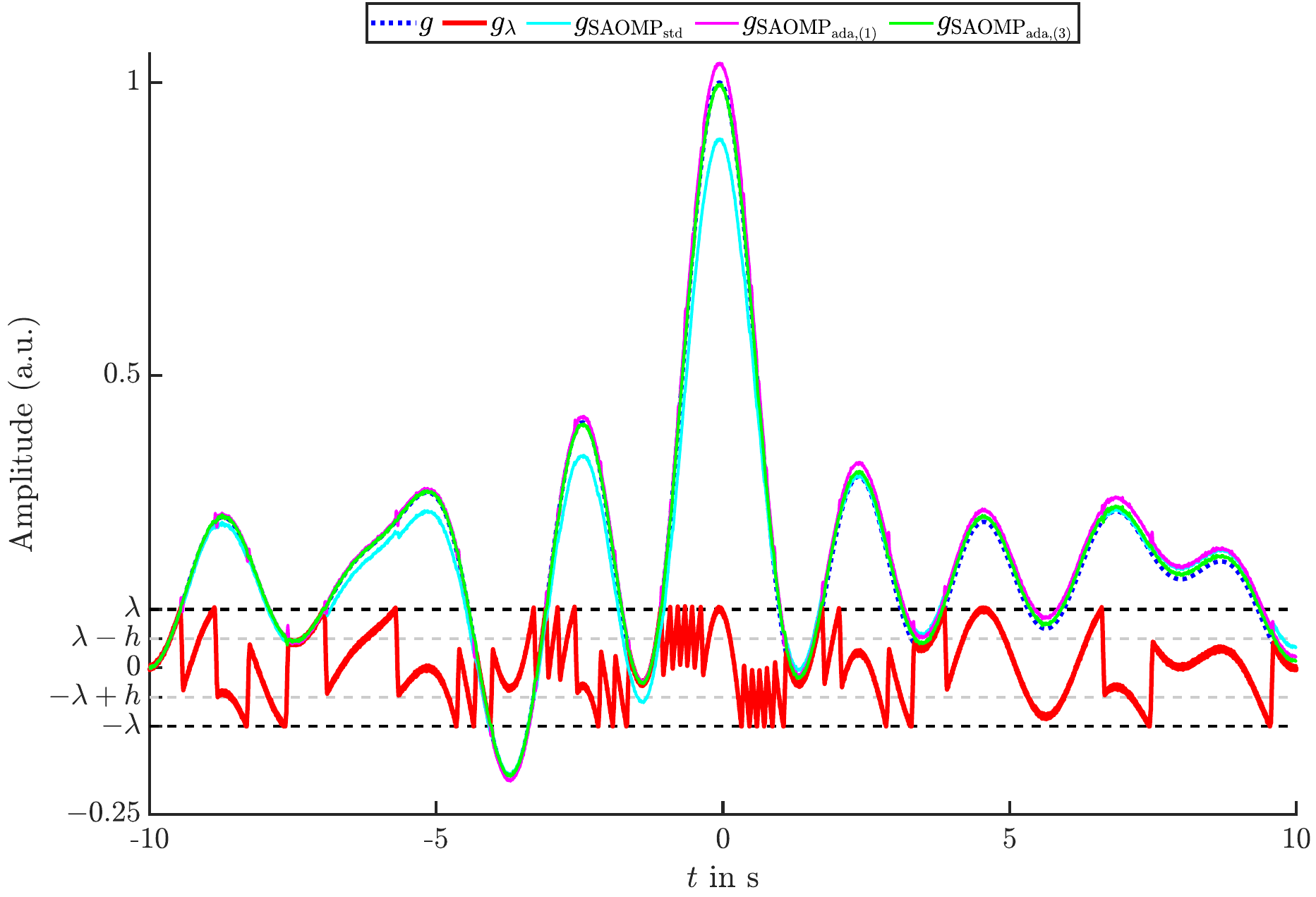}
	\caption{Illustration of SAOMP reconstruction (standard with and without iterated forward differences vs.\ adapted dictionary) for random $g \in \PW_\Omega$, $\Omega = \pi \tfrac{\text{rad}}{\text{s}}$, based on noisy samples $ \tilde{g}_{\lambda}(n\T)$ with an SNR of $30 \text{dB}$ and $\lambda = 0.1$, $h = 0.05$, $\alpha =  50 \text{ms}$, $\T = 5.2 \text{ms}$, $j_{\alpha} =  j_{\alpha}^{(3)}$.}
	\label{fig:Mod_recon_ada_noisy_j3}
\end{figure}

When considering $g_\lambda = \Mod_{\lambda,\sigma}^{h,j_\alpha} g$ with $j_\alpha = j_\alpha^{(3)}$, the correct number of forward differences is $p = 4$.
While the iterated forward differences do not create additional errors in the noise-free setting, the necessary iterated anti-differences amplify small inaccuracies in the identification of folding points and, consequently, render this approach inappropriate.
Hence, in Fig.~\ref{fig:Mod_recon_ada_j3} we compare standard SAOMP with adapted dictionaries, where in addition to $j_\alpha^{(3)}$ we falsely adapt to $j_\alpha^{(1)}$.
We observe that SAOMP with correctly adapted dictionary yields a satisfying reconstruction with an MSE of $7.3 \cdot 10^{-6}$, while the wrong dictionary leads to an MSE of $2.5 \cdot 10^{-3}$, worse than standard SAOMP with an MSE of $8.9 \cdot 10^{-4}$.
This illustrates the necessity of correctly adapting the dictionary and, thus, the need for an accurate model.
The same remains true when adding noise, as illustrated in Fig.~\ref{fig:Mod_recon_ada_noisy_j3}.

\begin{figure}[t]
	\centering
	\includegraphics[width=\linewidth,]{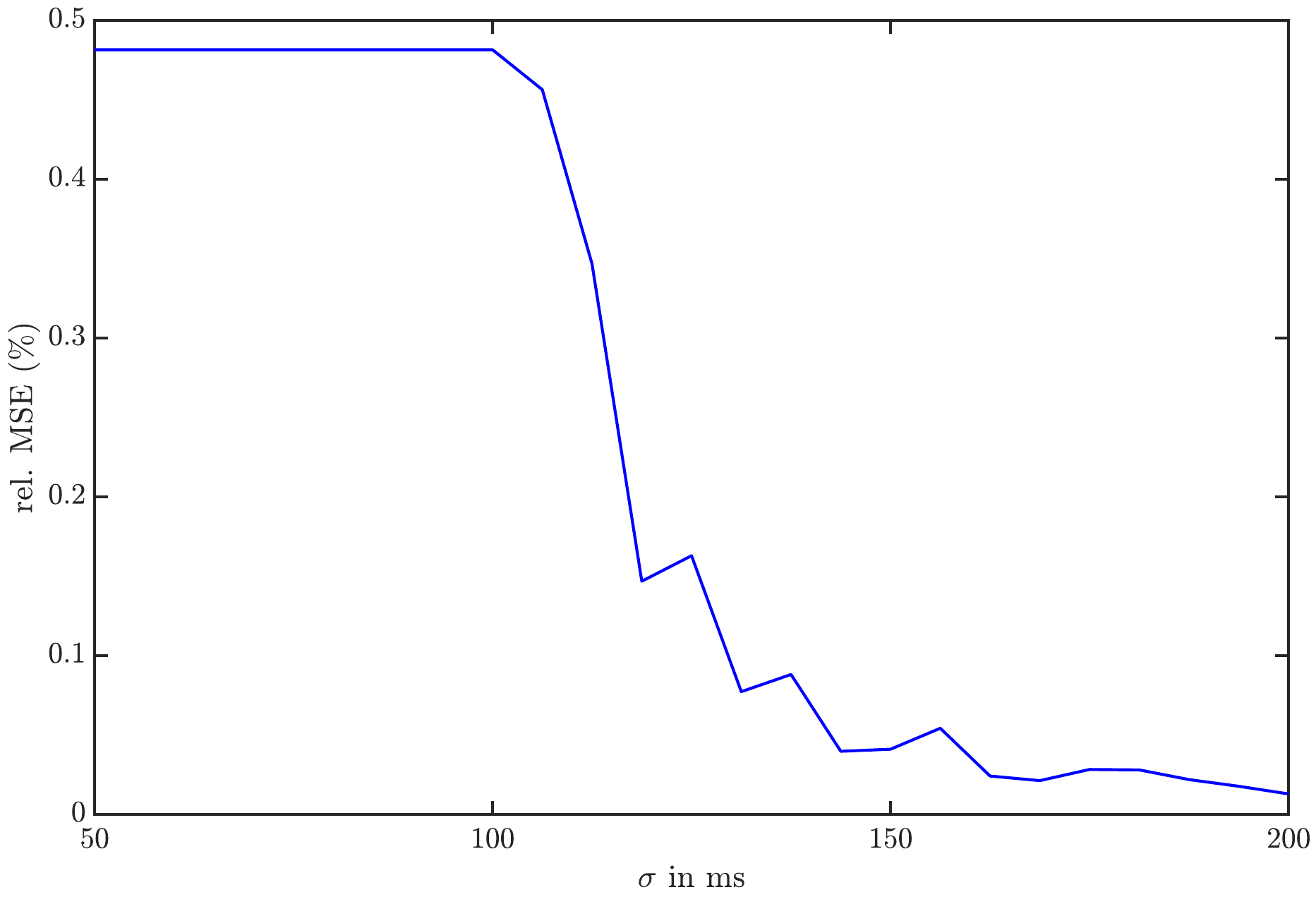}
	\caption{Median relative MSE of SAOMP reconstructions with adapted dictionary for $25$ random $g \in \PW_\Omega$ with $\Omega = \pi \tfrac{\text{rad}}{\text{s}}$ based on samples $g_\lambda(n\T)$ with $g_\lambda = \Mod_{\lambda,\sigma}^{h,j_\alpha} g$, $\lambda = 0.1$, $h = 0.05$, $\alpha =  100 \text{ms}$, $\T = 5.2 \text{ms}$, $j_{\alpha} =  j_{\alpha}^{(1)}$ and varying $\sigma \in [\frac{\alpha}{2},2\alpha]$.
	For $g_\lambda = \Mod_H g$ the median relative MSE is $0.48 \%$ and for $g_\lambda = \Mod_\lambda^{h,j_\alpha} g$ the median relative MSE is $0.37 \%$}
	\label{fig:Mod_recon_sigma}
\end{figure}

In our last numerical experiment, we study the influence of the reset time $\sigma$ on the reconstruction quality of SAOMP with adapted dictionary.
As a larger value for $\sigma$ yields a larger separation of folds, we expect that reconstruction becomes easier with increasing $\sigma$.
This can indeed be observed numerically, as illustrated in Fig.~\ref{fig:Mod_recon_sigma}, where we plot the median relative MSE as a function of $\sigma$.
More precisely, we computed the relative MSE for the SAOMP reconstruction with adapted dictionary from $g_\lambda(n\T) = \Mod_{\lambda,\sigma}^{h,j_\alpha} g(n\T)$ with $j_{\alpha} =  j_{\alpha}^{(1)}$, $\lambda = 0.1$, $h = 0.05$, $\alpha =  0.1$, $\T = 0.0052$ and varying $\sigma \in [\frac{\alpha}{2},2\alpha]$ for $25$ random functions $g \in \PW_\pi$.
We observe that for $\frac{\alpha}{2} \leq \sigma \leq \alpha$ the median relative MSE is constant and of the same level as for $g_\lambda = \Mod_h g$.
Increasing $\sigma$, however, leads to a significant drop yielding better results than for $g_\lambda = \Mod_{\lambda}^{h,j_\alpha} g$.

\section{Conclusion} \label{sec:conclusion}

The unlimited sensing framework overcomes the intrinsic trade-off between dynamic range and digital resolution in conventional Shannon-Nyquist sampling by exploiting that for any bandlimited signal its fractional part encodes the integer part.
To account for hardware imperfections, modulo hysteresis operators were proposed leading to a perturbed decomposition of the signal.
In this paper we introduced three generalized modulo hysteresis operators, which differ from the standard linear folding transition by incorporating a general transition function, which for two of the operators can take any form between zero and $\alpha$.
The operators differ in the way they determine the folding times.
While the first one determines the folding points in relation to an ideal folding transition, the other two make the position of the folds depend on the transient nonideality either with or without an additional reset time.
We rigorously analysed the mathematical properties of the three operators and studied the question under which conditions the perturbed fractional part still encodes the whole signal.
In particular, we proved conditions for identifiability of bandlimited signals from discrete generalized modulo samples.
Moreover, we described a mathematically backed reconstruction algorithm based on orthogonal matching pursuit, which can be adapted to the specific modulo operator and folding function.
This allowed us to also reconstruct signals in the challenging regime where the transition parameter is large compared to the sampling rate.
Our numerical results indicate that the reconstruction quality depends on the utilized model and recovery works best for transient depending folding points with reset time larger than the transient parameter.

Our work raises interesting questions and possible extensions for future research.
Key directions include
\begin{itemize}
\item validating our theoretical results through hardware experiments to study the applicability of our models in practice,
\item theoretically and numerically studying the stability and noise performance of our reconstruction approach,
\item designing recovery algorithms tailored to the specific model that allow for non-bandlimited input functions.
\end{itemize}


\begin{thebibliography}{10}
\bibitem{Shannon1948}
C.~E. Shannon, ``A mathematical theory of communication,'' \emph{The Bell
  System Technical Journal}, vol.~27, no.~3, pp. 379--423, 1948.

\bibitem{Walden1999}
R.~H. Walden, ``Analog-to-digital converter survey and analysis,'' \emph{IEEE
  Journal on Selected Areas in Communications}, vol.~17, no.~4, pp. 539--550,
  1999.

\bibitem{Bhandari2017}
A.~Bhandari, F.~Krahmer, and R.~Raskar, ``On unlimited sampling,'' in
  \emph{IEEE International Conference on Sampling Theory and Applications
  (SampTA)}, 2017, pp. 31--35.

\bibitem{Bhandari2020}
------, ``On unlimited sampling and reconstruction,'' \emph{IEEE Transactions
  on Signal Processing}, vol.~69, pp. 3827--3839, 2020.

\bibitem{Bhandari2021}
A.~Bhandari, F.~Krahmer, and T.~Poskitt, ``Unlimited sampling from theory to
  practice: {F}ourier-{P}rony recovery and prototype {ADC},'' \emph{IEEE
  Transactions on Signal Processing}, vol.~70, pp. 1131--1141, 2021.

\bibitem{Zhu2025a}
Y.~Zhu, R.~Guo, and A.~Bhandari, ``Ironing the modulo folds: Unlimited sensing
  with 100x bandwidth expansion,'' in \emph{IEEE International Conference on
  Sampling Theory and Applications (SampTA)}, 2025, pp. 1--5.

\bibitem{Mulleti2023}
S.~Mulleti, E.~Reznitskiy, S.~Savariego, M.~Namer, N.~Glazer, and Y.~C. Eldar,
  ``A hardware prototype of wideband high-dynamic range analog-to-digital
  converter,'' \emph{IET Circuits, Devices \& Systems}, vol.~17, no.~4, pp.
  181--192, 2023.

\bibitem{Zhu2024}
Y.~Zhu and A.~Bhandari, ``Unleashing dynamic range and resolution in unlimited
  sensing framework via novel hardware,'' in \emph{IEEE SENSORS}, 2024, pp.
  1--4.

\bibitem{Zhu2025}
J.~Zhu, J.~Ma, Z.~Liu, F.~Qu, Z.~Zhu, and Q.~Zhang, ``A modulo sampling
  hardware prototype and reconstruction algorithms evaluation,'' \emph{IEEE
  Transactions on Instrumentation and Measurement}, vol.~74, p. 2005511, 2025.

\bibitem{Guo2025}
R.~Guo, Y.~Zhu, and A.~Bhandari, ``Sub-nyquist {USF} spectral estimation: $k$
  frequencies with $6k+4$ modulo samples,'' \emph{IEEE Transactions on Signal
  Processing}, vol.~72, pp. 5065--5076, 2025.

\bibitem{Li2026}
Z.~Li, W.~Yan, L.~Gan, G.~Li, and H.~Liu, ``An {FPGA}-calibrated modulo {ADC}
  for high-dynamic-range signal acquisition,'' \emph{IEEE Open Journal of
  Circuits and Systems}, vol.~7, pp. 296--307, 2026.

\bibitem{Rudresh2018}
S.~Rudresh, A.~Adiga, B.~A. Shenoy, and C.~S. Seelamantula, ``Wavelet-based
  reconstruction for unlimited sampling,'' in \emph{IEEE International
  Conference on Acoustics, Speech and Signal Processing (ICASSP)}, 2018, pp.
  4584--4588.

\bibitem{Romanov2019}
E.~Romanov and O.~Ordentlich, ``Above the nyquist rate, modulo folding does not
  hurt,'' \emph{IEEE Signal Processing Letters}, vol.~26, no.~8, pp.
  1167--1171, 2019.

\bibitem{Weiss2022}
A.~Weiss, E.~Huang, O.~Ordentlich, and G.~W. Wornell, ``Blind modulo
  analog-to-digital conversion,'' \emph{IEEE Transactions on Signal
  Processing}, vol.~70, pp. 4586--4601, 2022.

\bibitem{Azar2022}
E.~Azar, S.~Mulleti, and Y.~C. Eldar, ``Residual recovery algorithm for modulo
  sampling,'' in \emph{IEEE International Conference on Acoustics, Speech and
  Signal Processing (ICASSP)}, 2022, pp. 5722--5726.

\bibitem{Shah2023}
S.~B. Shah, S.~Mulleti, and Y.~C. Eldar, ``Lasso-based fast residual recovery
  for modulo sampling,'' in \emph{IEEE International Conference on Acoustics,
  Speech and Signal Processing (ICASSP)}, 2023, pp. 1--5.

\bibitem{Guo2023}
R.~Guo and A.~Bhandari, ``{ITER-SIS}: {R}obust unlimited sampling via iterative
  signal sieving,'' in \emph{IEEE International Conference on Acoustics, Speech
  and Signal Processing (ICASSP)}, 2023, pp. 1--5.

\bibitem{Azar2025}
E.~Azar, S.~Mulleti, and Y.~C. Eldar, ``Unlimited sampling beyond modulo,''
  \emph{Applied and Computational Harmonic Analysis}, vol.~74, p. 101715, 2025.

\bibitem{Patricio2025}
F.~P. Patricio, F.~Krahmer, and P.~Catala, ``Stable retrieval for unlimited
  sampling via adaptive local representations,'' in \emph{IEEE Statistical
  Signal Processing Workshop (SSP)}, 2025, pp. 181--185.

\bibitem{Kobayashi2026}
H.~Kobayashi and R.~Hayakawa, ``Reconstruction of piecewise-constant sparse
  signals for modulo sampling,'' 2026, arXiv:2602.16418.

\bibitem{Bhandari2020a}
A.~Bhandari and F.~Krahmer, ``{HDR} imaging from quantization noise,'' in
  \emph{IEEE International Conference on Image Processing (ICIP)}, 2020, pp.
  101--105.

\bibitem{Zhou2020}
C.~Zhou, H.~Zhao, J.~Han, C.~Xu, C.~Xu, T.~Huang, and B.~Shi, ``{UnModNet}:
  {L}earning to unwrap a modulo image for high dynamic range imaging,'' in
  \emph{Advances in Neural Information Processing Systems (NeurIPS)}, 2020, pp.
  1--12.

\bibitem{Monroy2026}
B.~Monroy and J.~Bacca, ``Deep lightweight unrolled network for high dynamic
  range modulo imaging,'' \emph{IEEE Transactions on Computational Imaging},
  vol.~12, pp. 406--415, 2026.

\bibitem{Ordonez2021}
L.~G. Ordo{\~n}ez, P.~Ferrand, M.~Duarte, M.~Guillaud, and G.~Yang, ``On
  full-duplex radios with modulo-{ADC}s,'' \emph{IEEE Open Journal of the
  Communications Society}, vol.~2, pp. 1279--1297, 2021.

\bibitem{Liu2023}
Z.~Liu, A.~Bhandari, and B.~Clerckx, ``{$\lambda$–MIMO}: {Massive MIMO} via
  modulo sampling,'' \emph{IEEE Transactions on Communications}, vol.~71,
  no.~11, pp. 6301--6315, 2023.

\bibitem{Liu2025}
------, ``Full-duplex beyond self-interference: {T}he unlimited sensing way,''
  \emph{IEEE Communications Letters}, vol.~29, no.~1, pp. 165--169, 2025.

\bibitem{Geng2023}
T.~Geng, F.~Ji, Pratibha, and W.~P. Tay, ``Modulo {EEG} signal recovery using
  transformer,'' in \emph{IEEE International Conference on Acoustics, Speech
  and Signal Processing (ICASSP)}, 2023, pp. 1--5.

\bibitem{Beckmann2022}
M.~Beckmann, A.~Bhandari, and F.~Krahmer, ``The modulo {R}adon transform:
  {T}heory, algorithms and applications,'' \emph{SIAM Journal on Imaging
  Sciences}, vol.~15, no.~2, pp. 455--490, 2022.

\bibitem{Beckmann2024}
M.~Beckmann, A.~Bhandari, and M.~Iske, ``Fourier-domain inversion for the
  modulo {R}adon transform,'' \emph{IEEE Transactions on Computational
  Imaging}, vol.~10, pp. 653--665, 2024.

\bibitem{Beckmann2025}
M.~Beckmann and C.~Dittert, ``On an analytical inversion formula for the modulo
  {R}adon transform,'' in \emph{Scale Space and Variational Methods in Computer
  Vision (SSVM)}, 2025, pp. 56--69.

\bibitem{Feuillen2023}
T.~Feuillen, B.~Shankar, and A.~Bhandari, ``Unlimited sampling radar: Life
  below the quantization noise,'' in \emph{IEEE International Conference on
  Acoustics, Speech and Signal Processing (ICASSP)}, 2023, pp. 1--5.

\newpage
\bibitem{Feuillen2025}
T.~Feuillen and A.~Bhandari, ``Quadrature radars naturally benefit from
  unlimited sensing,'' in \emph{IEEE Statistical Signal Processing Workshop
  (SSP)}, 2025, pp. 251--255.

\bibitem{Zhang2025}
Q.~Zhang, J.~Zhu, F.~Qu, and D.~W. Soh, ``One-bit-aided modulo sampling for
  {DOA} estimation,'' in \emph{IEEE International Radar Conference (RADAR)},
  2025, pp. 1--6.

\bibitem{Florescu2022}
D.~Florescu, F.~Krahmer, and A.~Bhandari, ``The surprising benefits of
  hysteresis in unlimited sampling: {T}heory, algorithms and experiments,''
  \emph{IEEE Transactions on Signal Processing}, vol.~70, pp. 616--630, 2022.

\bibitem{Florescu2022a}
D.~Florescu and A.~Bhandari, ``Unlimited sampling via generalized
  thresholding,'' in \emph{IEEE International Symposium on Information Theory
  (ISIT)}, 2022, pp. 1606--1611.

\bibitem{Florescu2025}
------, ``Multi-dimensional unlimited sampling and robust reconstruction,''
  \emph{Applied and Computational Harmonic Analysis}, vol.~79, p. 101796, 2025.

\bibitem{Beckmann2025a}
M.~Beckmann and J.~Jeschke, ``Orthogonal matching pursuit based reconstruction
  for modulo hysteresis operators,'' in \emph{IEEE International Conference on
  Sampling Theory and Applications (SampTA)}, 2025, pp. 1--5.

\bibitem{Guo2025a}
R.~Guo and A.~Bhandari, ``Unlocking off-the-grid sparse recovery with unlimited
  sensing: Simultaneous super-resolution in time and amplitude,'' \emph{IEEE
  Journal of Selected Topics in Signal Processing}, vol.~19, no.~8, pp.
  1808--1821, 2025.

\bibitem{Beckmann2022a}
M.~Beckmann and A.~Bhandari, ``{MR. TOMP}: {I}nversion of the modulo {R}adon
  transform {(MRT)} via orthogonal matching pursuit {(OMP)},'' in \emph{IEEE
  International Conference on Image Processing (ICIP)}, 2022, pp. 3748--3752.

\bibitem{Mallat1993}
S.~Mallat and Z.~Zhang, ``Matching pursuits with time-frequency dictionaries,''
  \emph{IEEE Transactions on Signal Processing}, vol.~41, no.~12, pp.
  3397--3415, 1993.

\bibitem{Zhang2018}
Y.~Zhang and G.~Sun, ``Stagewise arithmetic orthogonal matching pursuit,''
  \emph{International Journal of Wireless Information Networks}, vol.~25, pp.
  221--228, 2018.

\bibitem{Bhandari2022}
A.~Bhandari, ``Unlimited sampling with sparse outliers: Experiments with
  impulsive and jump or reset noise,'' in \emph{IEEE International Conference
  on Acoustics, Speech and Signal Processing (ICASSP)}, 2022, pp. 5403--5407.
\end{thebibliography}
\end{document}